\documentclass[preprint, fleqn]{elsarticle}

\usepackage[utf8]{inputenc}
\usepackage{amssymb}
\usepackage{amsfonts}
\usepackage{amsthm}
\usepackage{amsmath}
\usepackage{tikz-cd}
\usepackage{cancel}
\usepackage{array}
\usepackage{bussproofs}
\usepackage{bm}
\usepackage{ tipa }
\usepackage{proof}
\usepackage{verbatim}
\usepackage{subfigure}
\usepackage{framed}
\usepackage{graphicx}
\usepackage[pagewise,mathlines]{lineno}
\usepackage{tikz}
\usetikzlibrary{matrix, arrows}
\usepackage{color}
\usepackage{hyperref}
\usepackage{url}
\usepackage{float}
\usepackage{booktabs}
\usepackage{multicol}
\usepackage{multirow}
\usepackage{natbib}
\newtheorem{definition}{Definition}
\newtheorem{proposition}{Proposition}

\newtheorem{example}{Example}
\newtheorem{theorem}{Theorem}
\newtheorem{corollary}{Corollary}
\newtheorem{lemma}{Lemma}

\newtheorem{observation}{Observation}
\newtheorem{note}{Note}
\newtheorem{appendices}{Appendix}

\newcommand{\bit}{\begin{itemize}}
	\newcommand{\eit}{\end{itemize}}
\newcommand{\ben}{\begin{enumerate}}
	\newcommand{\een}{\end{enumerate}}
\newcommand{\bds}{\begin{description}}
	\newcommand{\eds}{\end{description}}

\newcounter{romc}

\newcounter{alphc}

\newcommand{\blr}{\begin{list}{~(\roman{romc})~} {\usecounter{romc}
			\setlength{\topsep}{0pt} \setlength{\itemsep}{0pt}}}
	\newcommand{\elr}{\end{list}}
\newcommand{\bla}{\begin{list}{~(\alph{alphc})~} {\usecounter{alphc}
			\setlength{\topsep}{0pt} \setlength{\itemsep}{0pt}}}
	\newcommand{\ela}{\end{list}}
\begin{document}

\begin{frontmatter}


 \title{A non-distributive logic for semiconcepts of a context and \\ its modal extension with semantics based on Kripke contexts
 %\tnoteref{label1}
 }
 \author[IITI]{Prosenjit Howlader}
 \ead{prosen@iitk.ac.in}
 \author[IITI]{Mohua Banerjee\corref{cor2}}
 \ead{mohua@iitk.ac.in}

 \cortext[cor2]{Corresponding author}

 \address[IITI]{Department of Mathematics and Statistics,
 	Indian Institute of Technology
 	Kanpur,  India}
%\affiliation{organization={},
        %    addressline={},
           %  city={},
             %state={},
            %country={}}
 %\fntext[label3]{}

%\title{}


%\author{}

\begin{abstract}
A non-distributive two-sorted hypersequent calculus \textbf{PDBL}  and its modal extension \textbf{MPDBL} are proposed for  the classes of pure double Boolean algebras and pure double Boolean algebras with operators respectively. A relational semantics for \textbf{PDBL}  is next proposed, where any formula is interpreted as a semiconcept of a context. For \textbf{MPDBL}, the relational semantics is based on Kripke contexts, and a formula is interpreted as a semiconcept of the underlying context. The systems are shown to be sound and complete with respect to the relational semantics. Adding appropriate sequents to \textbf{MPDBL} results in logics with semantics based on reflexive, symmetric or transitive Kripke contexts. One of these systems  is  a logic for topological pure double Boolean algebras. It is demonstrated that, using \textbf{PDBL}, the basic notions and relations of conceptual knowledge can be expressed and inferences involving negations can be obtained. Further, drawing a connection with rough set theory, lower and upper approximations of semiconcepts of a context are defined. It is then shown that, using  the formulae and sequents involving modal operators in \textbf{MPDBL},  these approximation operators and their properties can be captured.
\end{abstract}




\end{frontmatter}


\section{Introduction}
\label{intro}
Rough set theory  \cite{pawlak1982rough} and formal concept analysis (FCA) \cite{FCA} are both well-established domains of study with applications in a wide range of fields, including knowledge representation and logic. 
In FCA,  a {\it context} (also called a {\it polarity}) \cite{FCA, ganter2012formal}   is a triple $\mathbb{K}:=(G, M, R)$, where $G$, $M$ are the sets of {\it objects} and {\it properties} respectively, and $R\subseteq G\times M$. %Polarity is heart of  formal concept analysis(FCA) \cite{ganter2012formal}. In FCA \cite{ganter2012formal}  a polarity is called context. 
%The sets $G$ and $M$ are called set of objects and set of properties respectively.  
For $g\in G$ and $m\in M$, $gIm$ is read as ``the object $g$ has the property $m$''. A context $\mathbb{K}$ induces a Galois connection \cite{BG, davey2002introduction}: $()^{\prime}: \mathcal{P}(G) \rightleftharpoons \mathcal{P}(M):()^{\prime}$, where  for any $A\subseteq G, B\subseteq M$, 
$A^{\prime}:=\{m \in M: \mbox{for all}~ g\in G(g\in A  \Longrightarrow{gRm})\}$, and
$ B^{\prime}:=\{g\in G:\mbox{for all}~ m\in M(m\in B \Longrightarrow {gRm})\}$. A pair $(A,B)$ is called a {\it concept} of $\mathbb{K}$ provided $A^{\prime}=B$ and $B^{\prime}=A$; 
%For a concept $(A,B)$ of $\mathbb{K}$, 
$A$ is then called its {\it extent}, denoted as $ext((A, B))$, and $B$  its {\it intent}, denoted as $int((A, B))$.
In this work, we are interested in formulating   logics with which one can reason about  objects and properties in a context, concepts of a context as well as `negations of concepts'.

As seen in  \cite{lpwille}, if the `negation of a concept' is formalized using set-complement, there is a problem of closure. 
% as well as   negation. 
%The negation of a concept was introduced by Wille \cite{lpwille, WILLE1992493}. 
%Set complement is used to defined negation of a concept of a context that leads to the problem of closure, that is 
%negation of a concept may not be a concept. To resolve the problem of closure 
%In order to the notion of concepts of 
%To introduce the notion of  negation in FCA, 
The notion of a concept was then generalized to that of a {\it semiconcept} by Wille \cite{lpwille,WILLE1992493}: 

\begin{definition}
	{\rm  \cite{wille} A pair $(A,B)$ is called a {\it  semiconcept} of $\mathbb{K}$ if and only if $A^{\prime}=B$ or $B^{\prime}=A$. }
\end{definition}
\noindent  For a context $\mathbb{K}$, $\mathcal{B}(\mathbb{K})$ and $\mathfrak{H}(\mathbb{K})$  denote the sets of all concepts and semiconcepts of $\mathbb{K}$, respectively.
%denotes the set of all concepts of $\mathbb{K}$.
%of $\mathbb{K}$, $A:=ext((A,B))$ is its {\it extent} and $B:=int((A,B))$ is its {\it intent}. The set of concept of a context $\mathbb{K}$ is denoted by $\mathfrak{B}(\mathbb{K})$. 
%An order relation $\leq$ is obtained on $\mathcal{B}(\mathbb{K})$ as follows:  $\mbox{for}~(A_{1},B_{1}),(A_{2},B_{2})\in\mathcal{B}(\mathbb{K}), (A_{1},B_{1}) \leq (A_{2},B_{2})~\mbox{if and only if}~ A_{1}\subseteq A_{2}~
%(\mbox{equivalent to}~ B_{2}\subseteq B_{1})$.
%From  the definition of concept of a context $\mathbb{K}$, one can derive that $A\subseteq G$ is an extent of some concept of $\mathbb{K}$ if and only if $A$ is a stable subset of the above Galois connection that is $A^{\prime\prime}=A$. On the other hand $B\subseteq M$ is an intent of some concept of $\mathbb{K}$ if and only if $B^{\prime\prime}=B$, a co-stable subset of the Galois connection $(\prime, \prime)$.  
%The set  $\mathcal{B}(\mathbb{K})$ of all concepts of a context $\mathbb{K}$ forms a complete lattice and the lattice is called concept lattice of $\mathbb{K}$. Moreover, every complete lattice is isomorphic to some concept lattice of some context.
It can be shown that $\mathcal{B}(\mathbb{K})\subseteq \mathfrak{H}(\mathbb{K})$.  
%Our goal, therefore, is to 
Our focus is on defining a logic for semiconcepts of a context. 

$\mathcal{B}(\mathbb{K})$  forms a complete lattice, called the {\it concept lattice} of $\mathbb{K}$, where 
the order relation $\leq$ giving the lattice structure  on $\mathcal{B}(\mathbb{K})$ is obtained as follows:  $\mbox{for}~(A_{1},B_{1}),(A_{2},B_{2})\in\mathcal{B}(\mathbb{K}), (A_{1},B_{1}) \leq (A_{2},B_{2})~\mbox{if and only if}~ A_{1}\subseteq A_{2}~
(\mbox{equivalent to}~ B_{2}\subseteq B_{1})$.
%From  the definition of concept of a context $\mathbb{K}$, one can derive that $A\subseteq G$ is an extent of some concept of $\mathbb{K}$ if and only if $A$ is a stable subset of the above Galois connection that is $A^{\prime\prime}=A$. On the other hand $B\subseteq M$ is an intent of some concept of $\mathbb{K}$ if and only if $B^{\prime\prime}=B$, a co-stable subset of the Galois connection $(\prime, \prime)$.  
%The set  $\mathcal{B}(\mathbb{K})$ of all concepts of a context $\mathbb{K}$ forms a complete lattice and the lattice is called concept lattice of $\mathbb{K}$. 
Furthermore, every complete lattice is isomorphic to  the concept lattice of some context.
On the other hand, the following operations are  defined on  $\mathfrak{H}(\mathbb{K})$ \cite{wille}. For $(A_{1},B_{1})$ and $(A_{2},B_{2})$ in $\mathfrak{H}(\mathbb{K})$,

\begin{center}
	$(A_{1},B_{1})\sqcap(A_{2},B_{2}) :=(A_{1}\cap A_{2}, (A_{1}\cap A_{2})^{'})$\\
	$(A_{1},B_{1})\sqcup(A_{2},B_{2}) :=((B_{1}\cap B_{2})^{'}, B_{1}\cap B_{2})$\\
	$\neg(A,B) :=(G\setminus A,(G\setminus A)^{'})$\\
	$\lrcorner(A,B) :=((M\setminus B)^{'}, M\setminus B)$\\
	$\top :=(G,\emptyset)$\\
	$\bot :=(\emptyset, M)$.
\end{center}

\noindent 
%$\sqcup$ and $\sqcap$ are called the join and meet. $\top$ and $\bot$ are called top and bottom element of  he set $\mathfrak{H}(\mathbb{K})$. 
$\neg$ and $\lrcorner$ are  negation operators in  $\mathfrak{H}(\mathbb{K})$. A semiconcept $x\in \mathfrak{H}(\mathbb{K})$ is called a {\it left semiconcept} \cite{BALBIANI2012260} if it is idempotent with respect to $\sqcap$, i.e. $x\sqcap x=x$. A {\it right semiconcept} \cite{BALBIANI2012260} is defined dually. In general, the set $\mathfrak{H}(\mathbb{K})$ does not form a lattice with respect to the operations  $\sqcup$ and $\sqcap$. However, $\underline{\mathfrak{H}}(\mathbb{K}):= (\mathfrak{H}(\mathbb{K}), \sqcap,\sqcup,\neg,\lrcorner,\top,\bot)$ is an instance of  a rich algebraic structure called  {\it pure double Boolean algebra}; it is called the {\it algebra of semiconcepts} of the context $\mathbb{K}$. Further, every  pure double Boolean algebra is embeddable in the algebra of semiconcepts of some context, as proved in  \cite{BALBIANI2012260}.
%(Definition \ref{DBA}, Section **) 
%with respect to the above operations.  
% **?**Using 
%this mathematical representation of conceptual knowledge, 
%a conceptual knowledge system, one is able to perform conceptual knowledge inference, where **negations of concepts of context play an important role.  It increases
%the possibilities of inferences. 
%*Thus logic for concepts and their negations is a logic that allows us to describe the three basic notions and four relations while also allowing us to perform conceptual reasoning.*  
%For negations of concepts, we need the  semiconcept of context.
%From the above discussion, it is clear that reasoning about the concept and its negation is not possible within the set of concepts of context.  This motivate us the following studies.
%Our goal, as previously stated, is to construct a logic in which, we can represent the concept of a context as well as its negation. As negations of a concept are expressed in terms of semiconcept, we can achieve so by defining logic for semiconcept of context. 
%The logic we propose here not only fulfills our goal but also using the  logic; we can represent all basic notions and relations of conceptual knowledge.
%Before we get into logic, we would like to point out
%Let us point out here  that the algebraic structure  $\underline{\mathfrak{H}}(\mathbb{K}):= (\mathfrak{H}(\mathbb{K}), \sqcap,\sqcup,\neg,\lrcorner,\top,\bot)$,  is a **pure double Boolean algebra.  
Considering these facts, our goal in this work is to formulate a logic that is sound and complete with respect to the class of all pure double Boolean algebras. 
The language of the proposed logic is taken to consist of two  disjoint countably infinite sets of propositional variables - the set  $\textbf{OV}$ of {\it object variables} and the set $\textbf{PV}$ of {\it property variables}. More precisely, we propose a  hypersequent calculus $\textbf{PDBL}$ whose language $\mathfrak{L}$
consists of the following: propositional constants $\top,\bot$, logical connectives $\sqcup,\sqcap, \neg,\lrcorner$, and the two  sets $\textbf{OV}$, $\textbf{PV}$ as mentioned above. Formulae are built over  $\textbf{OV}\cup\textbf{PV}$. $\textbf{PDBL}$ is a non-distributive logic, that is, $\sqcap$ does not  necessarily distribute over $\sqcap$ and vice versa.  It is shown that $\textbf{PDBL}$ is sound and complete with respect to the class of all pure double Boolean algebras. 

%In the field of logic,  the 
Relational semantics of non-distributive logics  where models are two-sorted and defined using contexts, have been extensively investigated.
Concepts of a context are potential interpretants for logical formulae over  models in such studies. 
%Contexts are fundamental for various non-distributive logics \cite{ Gehrke,CONRADIE2019923, CH, HCgame}, when the relational semantics  for the logics are considered.  %Bimbo and Dunn \cite{dunn} introduced generalized Kripke semantics %based on a context $\mathbb{K}$
%for the systems of **?**minimal non-distributive logic and minimal sub-structural logic, where the interpretants are concepts  of $\mathbb{K}$.
Gehrke \cite{Gehrke} %adopted the idea of generalized Kripke semantics and 
introduced a two-sorted approach to the relational semantics 
%based on the RS-polarity** 
for the implication-fusion fragment of various sub-structural logics.
The interpretation  is encoded by two relations: satisfaction ($\models$) and ``a part of" relation ($\succ$). Satisfaction is defined at a world (an object) and ``a part of" relation at a co-world (a property). Conradie  et al. \cite{CWFSPAPMTAWN, Conradie2017167,CONRADIE2019923}  investigated a two-sorted approach to the relational semantics for non-distributive modal logic. The definition of interpretation in \cite{CWFSPAPMTAWN} is the same as in \cite{Gehrke}, and the ``a  part of" relation ($\succ$) is renamed as co-satisfaction.  Later, Hartonas \cite{HCgame, CH} also studied two-sorted approaches to relational semantics for non-distributive modal systems. 

Along similar lines, a two-sorted approach to relational semantics is also defined for \textbf{PDBL} in this work -- only a formula is interpreted here as a semiconcept of a context. In particular, an object variable is interpreted as a left semiconcept, while a property variable is interpreted as a right semiconcept. The interpretation of any formula is then defined inductively with the help of relations of satisfaction ($\models$) and co-satisfaction ($\succ$). 
%turns out to be a logic for semiconcepts, as  intended.  
%Due to the representation result for pure double Boolean algebras \cite{BALBIANI2012260}, we are able to get  a **two-sorted approach to the relational semantics  for $\textbf{PDBL}$, where formulae are interpreted in $\mathfrak{H}(\mathbb{K})$. 

%To describe conceptual knowledge \cite{lpwille,WILLE1992493}, 
Using the concept lattice of a context,  a mathematical representation is given to conceptual knowledge; such a representation  is termed as a {\it conceptual knowledge system} \cite{lpwille,WILLE1992493}.
Wille adopted the idea from traditional philosophy that there are three basic notions of conceptual knowledge, namely, objects, attributes, and concepts. Moreover, these three are linked by four basic relations:
``an object has an attribute'', ``an object belongs to a concept'', ``an attribute abstracts from a concept'',
and ``a concept is a subconcept of  a concept''.  
%Now, for a context $\mathbb{K}:=(G, M, I)$, we note the following maps.
%\begin{itemize}
%	\item[-] $\zeta: G\rightarrow \{(\{g\}, \{g\}^{\prime})~:~g\in G\}(\subseteq \mathfrak{H}(\mathbb{K}))$, where $g\mapsto (\{g\}, \{g\}^{\prime})$ for all $g\in G$. 
%	\item[-] $\eta: M\rightarrow \{(\{m\}^{\prime}, \{m\})~:~m\in M\}(\subseteq \mathfrak{H}(\mathbb{K}))$, where $m \mapsto (\{m\}^{\prime}, \{m\})$ for all $m\in M$. 
%\end{itemize}
%  \noindent $\zeta$ and $\eta$ are  bijections. So the sets $\mathfrak{OS}:=\{(\{g\}, \{g\}^{\prime})~:~g\in G\}$ and  $\mathfrak{PS}:=\{(\{m\}^{\prime}, \{m\})~:~m\in M\}$ of semiconcepts  can be used to describe objects and  properties respectively, in the set $\mathfrak{H}(\mathbb{K})$.
%*Taking a cue from the above observation and hybrid modal logic \cite{blackburn2002moda},  object variables and  property variables are then interpretated in $\mathfrak{OS}$ and $\mathfrak{PS}$ , respectively.*
%*Taking a cue from the above observation and hybrid modal logic \cite{blackburn2002moda}, an object variable, and property variable is interpreted as a member of the set $\mathfrak{OS}$ and a member of the set $\mathfrak{PS}$ , respectively.*
%*A named model for \textbf{PDBL} is proposed for this purpose, taking   cue from  hybrid modal logic \cite{blackburn2002moda}. In a name model for  \textbf{PDBL}, an object variable $p\in \textbf{OV}$ interpreted as a member of the set $\mathfrak{OS}$ and a property variable $P\in \textbf{PV}$ interpreted as a member  of the set $\mathfrak{PS}$.* 
As concepts of a context are semiconcepts, it is expected that a logic for semiconcepts of a context should be one in which the three basic notions and four relations of conceptual knowledge can be expressed.
Indeed, one is able to demonstrate that 
using  \textbf{PDBL} and its relational semantics,  the basic notions and relations of conceptual knowledge are expressible. In particular,  ($\models$) represents the relation ``an object belongs to a concept'' and  ($\succ$)  represents the the relation ``an attribute abstracts from a concept''. 
Wille \cite{lpwille} has shown through an example that the negations $\neg, \lrcorner$  aid conceptual knowledge representation.
%improve the expressiveness of  the mathematical representations of conceptual knowledge. In \cite{lpwille} an example of a conceptual knowledge system demonstrates how negations aid conceptual knowledge representation.  
%Similar to the conceptual knowledge system, 
Employing sequents of  \textbf{PDBL}  comprising negations, we are also able to capture  conceptual knowledge involving negation. An example is given in this regard as well.
%The negations $\neg, \lrcorner$ are two fundamental operations for   Wille's  mathematical model of conceptual knowledge \cite{lpwille,WILLE1992493}. 
%The satisfaction relation $\models\subseteq G\times \mathfrak{H}(\mathbb{K})$ encodes the relation `an object belongs to a semiconcept', while the co-satisfaction relation $\succ\subseteq M\times \mathfrak{H}(\mathbb{K})$ encodes the relation `an attribute abstracts from a semiconcept'.

%the sound and completeness for  is  with respect to class of pdBas with operator,   The  Lindenbaum-Tarskialgebra algebra of $\textbf{MPDBL}$ forms a pdBas with operators(Definition \ref{DBA with operators}). As result of the above observation, 

%Soundness and completeness for $\mathbf{PDBL}$ are established for this semantics.  
%As $\mathfrak{B}(\mathbb{K})\subseteq\mathfrak{H}(\mathbb{K})$, the satisfaction relation also encodes the object concept relation of the conceptual knowledge, while co-satisfaction relation encodes attribute concept relation of the conceptual knowledg
%\cite{howlader2021dbalogic}.

%\vskip 2pt

In \cite{howlader2021dbalogic}, we have expanded the notion of a context to that of a {\it Kripke context} which links two Kripke frames by a relation:
\begin{definition}
	\label{K-cntx}
	{\rm \cite{howlader2021dbalogic} A {\it Kripke context} based on a context $\mathbb{K}:=(G,M,I)$ is a triple $\mathbb{KC}:=((G,R),(M,S),I)$, where $R,S$ are relations on $G$ and $M$ respectively.
		
		%\noindent  The Kripke context $\mathbb{KC}$, where $R$ and $S$ are reflexive and  transitive is called {\it reflexive transitive Kripke context}. Moreover, if $R$ and $S$ are symmetric then $\mathbb{KC}$ is called {\it reflexive transitive and symmetric Kripke context}.
	}
\end{definition}

\noindent The motivation of a Kripke context lies in the intersection of the frameworks  of rough set theory and FCA \cite{RCA, saquer2001concept, hul}. %Moreover,  Kripke context unifies within a single framework,  the notions of a context of FCA and approximation space of  rough set theory.
Two basic notions of rough set theory are those of approximation  spaces and approximation operators   \cite{pawlak1982rough,pawlak2012rough}. A pair $(W, E)$ is a {\it Pawlakian approximation space}, where $W$ is a set and $E$ is an equivalence relation on $W$.  $(W,E)$ is called a {\it generalised approximation space} \cite{yao1996generalization}, when 
%it is generalised to a pair $(W, E)$ where 
$E$ is any binary relation on $W$. %A {\it Pawlakian approximation space} is a pair $(W, E)$, where $W$ is a set and $E$  is an equivalence relation on $W$. This is generalised to a pair $(W, E)$ with $E$   any binary relation on $W$, and called a {\it generalised approximation space} \cite{yao1996generalization}.
%In this article, we are interested in the work of  Saquer and Deogun \cite{saquer2001concept}.
%An approximation space in a generalized rough set model \cite{yao1996generalization}  is a pair $(W, E)$, where $W$ is a set and $E$  is a binary relation on $W$. If $E$ is equivalence relation, $(W, E)$ is called  approximation space in  rough set model. 
For  $x\in W$,  $E(x):=\{y\in W~:~ xRy\}$.
The {\it lower} and {\it upper approximations} of any $A (\subseteq W)$ are defined respectively as $\underline{A}_{E}:=\{x\in W~:~E(x)\subseteq A\}$, and $\overline{A}^{E}:=\{x\in W~:~E(x)\cap A\neq\emptyset\}$.
%\blr
%\item $\underline{A}_{E}:=\{x\in W~:~E(x)\subseteq A\}$.
%\item $\overline{A}^{E}:=\{x\in W~:~E(x)\cap A\neq\emptyset\}$.
%\elr
%\noindent 
%If the relation is clear from the context, we shall omit the subscript and denote $\underline{A}_{E}$ by $\underline{A}$, $\overline{A}^{E}$ by $\overline{A}$.  
%\begin{proposition}
%	\label{pra}
%	{\rm 
	%		\noindent \blr  \item[{\bf I.}] For an approximation space $(W,E)$, the following are hold.
	%		\item $\overline{A}=(\underline{(A^{c})})^{c}, \underline{A}=(\overline{(A^{c})})^{c}$.
	%		\item $\underline{W}= W$.
	%		\item $\underline{A\cap B}=\underline{A}\cap\underline{B}, \overline{A\cup B}=\overline{A}\cup\overline{B}$.
	%		\item $A\subseteq B$ implies that $\underline{A} \subseteq \underline{B}, \overline{A} \subseteq \overline{B}$.
	%		
	%		\item[{\bf II.}] Moreover if $E$ is a reflexive and transitive relation then following hold.
	%		\item $\underline{A}\subseteq A$ and $A\subseteq \overline{A}$.
	%		\item $\underline{(\underline{A})}=\underline{A}$ and $\overline{(\overline{A})}=\overline{A}$.
	%		\elr}
%\end{proposition}
%In a rough set model $A\subset W$ is called definable if it is union of some equivalence classes. For example $\underline{A}$ and $\overline{A}$ are definable.
Kent  \cite{RCA}, Saquer et al. \cite{saquer2001concept} and  Hu et al. \cite{hul} incorporated the idea of approximation space in FCA, and discussed approximations {\it of concepts}, albeit from different  perspectives.
The basic motivation is that in some cases, the objects and properties that define a context are indistinguishable in terms of certain attributes.  For example,  two diseases may have symptoms that are indistinguishable. In this work, we  propose lower and upper approximations  {\it  of semiconcepts}. The base approximation space that we use is the same as the one in  \cite{saquer2001concept}. 
%**Moreover, it is shown that a modal extension of $\textbf{PDBL}$ (described briefly below)  is able to express the approximation operators for semiconcepts.
%but the approximation operators are different.  define approximation operators for semiconcepts, while  Kent, Saquer et al., and Hu et al. define approximation operators for concepts.
%*For a long time, the relationship between rough set and modal logic has been investigated.

Many researchers \cite{mbmc,yao1996generalization,lcj,mb,mbmak} have worked on various modal logics and their rough set semantics, in which modalities are interpreted as approximation operators. This extensive literature prompted us to look for appropriate  modal systems with semantics based on the Kripke context which could, moreover, express the approximations of semiconcepts. 
For a Kripke context $\mathbb{KC}$, the notion of {\it complex algebra} $\underline{\mathfrak{H}}^{+}(\mathbb{KC})$ of semiconcepts is defined in \cite{howlader2021dbalogic}.
In order to understand the equational theory of the complex algebras, {\it pure double Boolean algebras with operators} and {\it topological pure double Boolean algebras} 
%(Definition \ref{DBA with operators}) 
are proposed in \cite{howlader2021dbalogic}. 
%For a Kripke context $\mathbb{KC}$ the set $\mathfrak{H}(\mathbb{K})$ with the semiconcept approximation operators from a **??**pure double Boolean algebra with operators is a consequence of Theorem \ref{complex algebra}. 
%Using these algebraic structures, o
%In this work, one is able to formulate modal extensions of $\textbf{PDBL}$  with language $\mathfrak{L}_{1}$ that is obtained 
Here, the language $\mathfrak{L}$ of $\textbf{PDBL}$ is extended to $\mathfrak{L}_{1}$ by adding 
unary modal connectives $\square$ and $\blacksquare$, and   we get the modal system  $\textbf{MPDBL}$. $ \lozenge$ and $\blacklozenge$ are introduced as duals of $\square$ and $\blacksquare$ respectively. %Based on $\mathfrak{L}_{1}$, we get modal logics 
$\textbf{MPDBL}$ corresponds to the class of pure double Boolean algebras with operators. 
Further, the  logic $\textbf{MPDBL}\Sigma$ is  defined, where $\Sigma$ is any set of sequents in $\textbf{MPDBL}$. Taking an appropriate $\Sigma$, we get the logic for topological pure double Boolean algebras.
As a natural next course, the relational semantics of $\textbf{PDBL}$ is extended to that for the modal systems. 
%**two-sorted approach  to the relational semantics for the modal systems is also proposed.
%, based on the Kripke context \cite{howlader2021dbalogic}.
Formulae in the language $\mathfrak{L}_{1}$ are interpreted as semiconcepts of the underlying context $\mathbb{K}$ of a Kripke context $\mathbb{KC}$.    It is shown that formulae of the form $\square\alpha, \blacklozenge\alpha$ translate into the lower  approximations of a left semiconcept and  right semiconcept respectively, while $\blacksquare\alpha, \lozenge\alpha$ translate into the upper approximations of a left semiconcept and  right semiconcept respectively.
%where **$ \lozenge$ and $\blacklozenge$ are the dual of **$\square$ and $\blacksquare$, respectively.** 
Moreover, valid sequents  involving modal operators  express  properties of  lower and upper approximations of semiconcepts. The deductive systems are shown to be sound and complete for the respective semantics.
%it is shown that $\mathfrak{H}^{+}(\mathbb{K})$ forms a pdBa with operators and  if the relations $R$ and $S$ are reflexive and transitive then $\mathfrak{H}^{+}(\mathbb{K})$ is  topological pdBa. 

%In \cite{howlader2021dbalogic}  it is shown that every pdBa with operators is embedded into a complex algebra of some Kripke context, while every topological pdBa is embedded into a complex algebra into a complex algebra of some reflexive transitive Kripke context.
%Utilizing the above  representation theorems of pdBa with operators and topological pdBa, the following sound and completeness theorem for the modal system is proved.  $\textbf{MPDBL}$ is sound and complete with respect to the class  $\mathcal{KC}$ of all Kripke context. $\textbf{MPDBL4}$ is sound and complete with respect to the class $\mathcal{KC}_{RT}$ reflexive and transitive Kripke context, while sound and completeness of $\textbf{MPDBL5}$ is established with respect to class $\mathcal{KC}_{RST}$ of reflexive symmetric and transitive Kripke context. When rough set theory is concerned, the Kripke context $\mathbb{KC}:=((G, R), (M, S), I)\in \mathcal{KC}_{RST}$ can be understood as a basic classification skill about objects and properties of an agent based on the context $\mathbb{K}:= (G, M, I)$.

%Structure of the paper: 
Section \ref{preli} gives the preliminaries required for this work. Results related to pure double Boolean algebras are given in Section \ref{pdBas}. Section \ref{Appropefca} outlines work on approximation operators in rough set theory and approximations of concepts. Pure double Boolean algebras with operators,  topological pure double Boolean algebras and Kripke contexts  and related results are
given in Section \ref{dbawo}.  %The representation results for the class of algebras are studied  in Section \ref{RtdBaos}.
In Section \ref{logic}, the logics corresponding to the algebras
are studied. $\textbf{PDBL}$ for the class of all pure double Boolean algebras is discussed in Section \ref{ppdBl}. In Section \ref{logicmpdbl}, the modal systems are given. Section \ref{semisemantics} presents the relational semantics  for the logics, with the semantics for $\textbf{PDBL}$ discussed in Section \ref{gsempdbl}, and that for the modal systems given in Section \ref{modalsemantic}. A study of $\textbf{PDBL}$ and conceptual knowledge is done in Section \ref{sematics to meaning1}. Section \ref{rough set approx} describes the approximations of semiconcepts and demonstrates how  fundamental properties of the approximations are captured by using one of  the proposed modal systems, \textbf{MPDBL5}.  Section \ref{conclusion} concludes the  paper.

In our presentation, the symbols $\Longleftarrow,\Longleftrightarrow$, and, or and not will be used with
the usual meanings in the metalanguage. Throughout, $P(X)$ denotes the power set
of any set $X$, and the complement of $A \subseteq X$ in a set $X$ is denoted $A^{c}$.

\section{Preliminaries}
\label{preli}
%**to be checked**\\
%Basic notions and existing results related to dBas and approximations of rough set theory are presented in the following subsections. 
Our key references for this section are \cite{ganter2012formal,wille,saquer2001concept,BALBIANI2012260}.
\subsection{ Pure double Boolean algebra}
\label{pdBas}
Let us first give the definition of a double Boolean algebra. 
\begin{definition}
	\label{DBA}
	{\rm \cite{wille}  An  algebra $ \textbf{D}:= (D,\sqcup,\sqcap, \neg,\lrcorner,\top,\bot)$,  satisfying the following properties for any $x,y,z \in D$, is called a {\it double Boolean algebra} (dBa). \\
		$\begin{array}{ll}
		(1a)  (x \sqcap x ) \sqcap  y = x \sqcap  y  &
		(1b)  (x \sqcup x)\sqcup  y = x \sqcup y \\
		(2a) x\sqcap y = y\sqcap  x  &
		(2b)  x \sqcup   y = y\sqcup   x  \\
		(3a)  x \sqcap ( y \sqcap  z) = (x \sqcap  y) \sqcap  z  &
		(3b)  x \sqcup (y \sqcup  z) = (x \sqcup  y)\sqcup  z \\
		(4a) \neg (x \sqcap  x) = \neg  x  &
		(4b)  \lrcorner(x \sqcup   x )= \lrcorner x \\
		(5a)  x  \sqcap (x \sqcup y)=x \sqcap  x  &
		(5b)  x \sqcup  (x \sqcap y) = x \sqcup   x \\
		(6a) x \sqcap  (y \vee z ) = (x\sqcap  y)\vee (x \sqcap  z) &
		(6b)  x \sqcup  (y \wedge z) = (x \sqcup  y) \wedge  (x \sqcup  z) \\
		(7a)  x \sqcap (x\vee y)= x \sqcap  x  &
		(7b)  x\sqcup  (x \wedge  y) =x \sqcup   x \\
		(8a)  \neg \neg (x \sqcap  y)= x \sqcap  y &
		(8b)  \lrcorner\lrcorner(x \sqcup  y) = x\sqcup  y \\
		(9a)  x  \sqcap \neg  x= \bot &
		(9b)  x \sqcup \lrcorner x = \top  \\
		(10a) \neg \bot = \top \sqcap   \top  &
		(10b) \lrcorner\top =\bot \sqcup  \bot \\
		(11a)  \neg \top = \bot  &
		(11b) \lrcorner\bot =\top \\
		(12)  (x \sqcap  x) \sqcup (x \sqcap x) = (x \sqcup x) \sqcap (x \sqcup x), &
		\end{array}$

		\noindent where $ x\vee y := \neg(\neg x \sqcap\neg y)$, and 
		$ x \wedge y :=\lrcorner(\lrcorner x \sqcup \lrcorner y)$.
		\vskip 3pt 
		\noindent	$\textbf{D}$ is  {\it pure} if for all $x\in D$, either $x\sqcap x=x$ or $x\sqcup x=x$. 
				\vskip 3pt 

		\noindent	A binary relation $\sqsubseteq$ on $\textbf{D}$ is defined as follows:  
		$x\sqsubseteq y ~\mbox{if and only if}~  x\sqcap y=x\sqcap x~\mbox{and}~x\sqcup y=y\sqcup y,$ for any $x,y \in D$.  $\sqsubseteq$ is a quasi-order.
				\vskip 3pt 
		\noindent \textbf{D} is {\it contextual} if the quasi-order is a partial order.}
\end{definition}
\noindent The abbreviations pdBa and cdBa stand for pure dBa and contextual dBa, respectively. 
\begin{proposition} 
	\label{order pure}
	{\rm \cite{BALBIANI2012260} In every pdBa $\textbf{D}$, $\sqsubseteq$ is a partial order.} 
\end{proposition}

\begin{corollary}
\label{jh}
    {\rm Every pdBa $\textbf{D}$ is a cdBa.}
\end{corollary}
%A dBa $\textbf{D}$ is called {\it contextual} if $\sqsubseteq$ is a partial order. A contextual dBa is also  known as a regular dBa \cite{breckner2019topological}.

%Protoconcepts of a context $\mathbb{K}$ form a contextual dBa and semiconcepts of  $\mathbb{K}$ form a pdBa.
%\begin{notation} {\rm 
%For any dBa $\textbf{ D}:= (D,\sqcup,\sqcap,\neg,\lrcorner,\top,\bot)$, 
In the following, let $ \textbf{\textit{D}}:=(D,\sqcup,\sqcap,\neg,\lrcorner,\top,\bot)$ be a dBa.
Let us give some notations that shall be used:
$D_{\sqcap}:=\{x\in D~:~x\sqcap x=x\}$,  $D_{\sqcup}:=\{x\in D~:~x\sqcup x=x\}$, $D_{p}:=D_{\sqcap}\cup D_{\sqcup}$.
For $x\in D$, $x_{\sqcap}:=x\sqcap x$ and $x_{\sqcup}:=x\sqcup x$.

\begin{proposition} 
	\label{pro1}
	{\rm \cite{vormbrock} \noindent 
		%Let $\textbf{ D}= (D,\sqcup,\sqcap,\neg,\lrcorner,\top,\bot)$ be a dBa and $ \sqsubseteq $ is the quasi-order on $ D $. Then
		\begin{enumerate}
			\item $\textbf{D}_{\sqcap}:=(D_{\sqcap},\sqcap,\vee,\neg,\bot,\neg\bot)$ is a {\rm Boolean algebra} whose order relation is the restriction of $\sqsubseteq$ to $D_{\sqcap}$ and is denoted by $\sqsubseteq_{\sqcap}$.
			\item $\textbf{D}_{\sqcup}:=(D_{\sqcup},\sqcup,\wedge,\lrcorner,\top,\lrcorner\top)$ is a {\rm Boolean algebra} whose order relation is the restriction of $\sqsubseteq$ to $D_{\sqcup}$ and it is denoted by $\sqsubseteq_{\sqcup}$.
			\item $x\sqsubseteq y$ if and only if $x\sqcap x\sqsubseteq y\sqcap y$ and $x\sqcup x\sqsubseteq y\sqcup y$ for $x,y\in D$, that is, $x_{\sqcap}\sqsubseteq_{\sqcap} y_{\sqcap}$ and $x_{\sqcup}\sqsubseteq y_{\sqcup}$.
	\end{enumerate}}
\end{proposition}

\begin{proposition}
	\label{pro1.5}
	{\rm \cite{kwuida2007prime}
		Let 
		%$\textbf{ D}= (D,\sqcup,\sqcap,\neg,\lrcorner,\top,\bot)$ be a {\rm dBa} and 
		$x,y,a\in D$. Then the following hold.
		\begin{enumerate}
			\item $x\sqcap\bot=\bot$ and $x\sqcup\bot=x\sqcup x$ that is $\bot\sqsubseteq x$.
			\item $x\sqcup \top=\top$ and $x\sqcap \top=x\sqcap x$ that is $x\sqsubseteq \top$.
			\item $x=y$ implies that $x\sqsubseteq y$ and $y\sqsubseteq x$.
			\item $x\sqsubseteq y$ and $y \sqsubseteq x$ if and only if $x\sqcap x=y\sqcap y$ and $x\sqcup x = y\sqcup y$.
			\item $x\sqcap y\sqsubseteq x,y\sqsubseteq x\sqcup y, x\sqcap y\sqsubseteq y,x\sqsubseteq x\sqcup y$.
			\item $x\sqsubseteq y$ implies $x\sqcap a\sqsubseteq y\sqcap a$ and $x\sqcup a\sqsubseteq y\sqcup a$. 
	\end{enumerate}}
\end{proposition}

\begin{proposition}
	\label{pro2}
	{\rm \cite{howlader3} 
	%Let $ \textbf{\textit{D}}:=(D,\sqcup,\sqcap,\neg,\lrcorner,\top,\bot)$ be a dBa. Then 
	For any $ x,y\in D$, the following hold.
		\begin{enumerate}
			\item $\neg x\sqcap \neg x=\neg x$ and $\lrcorner x\sqcup \lrcorner x=\lrcorner x$ that is $\neg x=(\neg x)_{\sqcap}\in D_{\sqcap}$ and $\lrcorner x=(\lrcorner x)_{\sqcup}\in D_{\sqcup}$.
			\item $x\sqsubseteq y$  if and only if $ \neg y\sqsubseteq \neg x $ and $ \lrcorner y\sqsubseteq \lrcorner x $.
			\item $\neg\neg x=x\sqcap x$ and $\lrcorner\lrcorner x=x\sqcup x$.
			\item $x\vee y\in D_{\sqcap}, x\wedge y\in D_{\sqcup}$.
			\item $\neg\neg\bot=\bot$, and $\lrcorner\lrcorner \top=\top$.
	\end{enumerate}}
\end{proposition}

\begin{lemma}
	\label{botom neg and ordered}
	{\rm \cite{howlader2021dbalogic} 
	%Let $\textbf{D}:=(D,\sqcup,\sqcap,\neg,\lrcorner,\top,\bot)$ be a dBa.  
	For all $a,b\in D$, the following hold.
		\begin{enumerate}
			\item If $a\sqcap b=\bot$ then $a\sqcap a\sqsubseteq \neg b$.
			\item If $a\sqcap a\sqsubseteq \neg b$ then $a\sqcap b\sqsubseteq \bot$.
			\item If $a\sqcup b=\top$ then $\lrcorner b\sqsubseteq a\sqcup a$.
			\item If $\lrcorner b\sqsubseteq a\sqcup a$ then $\top\sqsubseteq a\sqcup b$.
		\end{enumerate}
		In particular, if $\textbf{D}$ is a cdBa then $a\sqcap b=\bot$ if and only if $a\sqcap a\sqsubseteq \neg b$, and  $a\sqcup b=\top$ if and only if $\lrcorner b\sqsubseteq a\sqcup a$.}
\end{lemma}

\begin{definition}
	{\rm A subset $F$ of $D$ is a  {\it filter} in  $\textbf{D}$ if and only if  $x\sqcap y\in F$ for all $x,y \in F$, and for all $z\in D$ and $ x \in F, x\sqsubseteq z$ implies that $z\in F$. An {\it ideal} in a dBa is defined dually.\\
		A filter $F$ (ideal $I$) is  {\it proper} if and only if  $F\neq D$ ($I\neq D$).  A proper filter $F$ (ideal $I$) is called {\it primary} if and only if $x\in F~ \mbox{or} ~\neg x \in F ~(x\in I ~\mbox{or}~ \lrcorner x\in I)$, for all $x\in D$.\\
		The set of  primary filters is denoted by $\mathcal{F}_{pr}(\textbf{D})$;  the set of all primary ideals is denoted by $\mathcal{I}_{pr}(\textbf{D})$.\\
		A {\it base} $F_{0}$ for a filter  $F$ is a subset of $D$ such that $F=\{x\in D~:~ z\sqsubseteq x~\mbox{for some}~ z\in F_{0}\}$. A {\it base} for an ideal is defined similarly. \\
		For a subset $X$ of $D$, $F(X)$ and $I(X)$ denote the filter and ideal generated by $X$ respectively.}
\end{definition}

\noindent To prove representation theorems for dBas, the following are introduced in   \cite{wille}. \\
$\mathcal{F}_{p}(\textbf{D}):=\{F \subseteq D : F ~\mbox{is a  filter of}~\textbf{D}~\mbox{and}~F\cap D_{\sqcap} ~\mbox{is a prime filter in}~ \textbf{D}_{\sqcap}\}$. \\
$\mathcal{I}_{p}(\textbf{D}):=\{I \subseteq D : I ~\mbox{is an ideal of}~\textbf{D}~\mbox{and}~I\cap D_{\sqcup}~\mbox{ is a prime ideal in}~ \textbf{D}_{\sqcup}\}$. 

\noindent In fact, we have
%\noindent In \cite{howlader3}, it has been proved that the sets $\mathcal{F}_{p}(\textbf{D})$ and $\mathcal{I}_{p}(\textbf{D})$ are the same as the sets of primary filters and primary ideals, respectively.
\begin{proposition}
	\label{compar-ideal}
	{\rm \cite{howlader3} %Let \textbf{D} be a dBa. Then 
		$\mathcal{F}_{p}(\textbf{D})=\mathcal{F}_{pr}(\textbf{D})$ and $\mathcal{I}_{p}(\textbf{D})=\mathcal{I}_{pr}(\textbf{D})$.}
\end{proposition}
\begin{lemma}
	\label{lema1}
	{\rm \cite{wille} %Let $\textbf{D}$ be  a dBa. 
		\begin{enumerate}
			\item For any filter $F$ of $\textbf{D}$, $F\cap D_{\sqcap}$ and $F\cap D_{\sqcup}$ are filters of the Boolean algebras $\textbf{D}_{\sqcap}$, $\textbf{D}_{\sqcup}$ respectively.
			\item Each filter $F_{0}$ of the Boolean algebra $\textbf{D}_{\sqcap}$ is the base of some filter $F$ of $\textbf{D}$ such that $F_{0}=F\cap D_{\sqcap}$. Moreover if $F_{0}$ is prime, $F\in\mathcal{F}_{p}(\textbf{D})$. 
		\end{enumerate}
		Similar results can be proved for ideals of dBas. }	
\end{lemma}
%For a context $\mathbb{K}:=(G, M, I)$ and sets  $A \subseteq G,B,\subseteq M$, recall the sets $A^\prime, B^\prime$  and   the operations on semiconcepts of $\mathbb{K}$ defined in Section \ref{intro}. 

%\begin{lemma}
%	\label{proty-prime}
%	{\rm 	\noindent   For any $A, X\subseteq G$ and  $B, Y\subseteq M$,
		%Let $\mathbb{K}:=(G, M, I)$ be a context and $A, X\subseteq G$ and $B, Y\subseteq M$. Then the following hold.
%		\begin{enumerate}
%			\item $A\subseteq A^{\prime\prime}$ and $B\subseteq B^{\prime\prime}$.
%			\item $A\subseteq X$ implies that $X^{\prime}\subseteq A^{\prime}$,  $B\subseteq Y$ implies that $Y^{\prime}\subseteq B^{\prime}$.
%			\item *For $(A, B)\in \mathfrak{H}(\mathbb{K})$, $A^{\prime\prime}=B^{\prime}$ and $A^{\prime}=B^{\prime\prime}$.*
%	\end{enumerate}}
%\end{lemma}
\noindent Recall the operations on the  the set $\mathfrak{H}(\mathbb{K})$ of all semiconcepts given in the Introduction.
\begin{theorem}
	\label{protconcept algebra}
	{\rm \cite{wille} \noindent 
		%Let $\mathbb{K}$ be a context. 
		$\underline{\mathfrak{H}}(\mathbb{K}):= (\mathfrak{H}(\mathbb{K}), \sqcap,\sqcup,\neg,\lrcorner,\top,\bot)$ is a  pdBa.
		%\item $\underline{\mathfrak{H}}(\mathbb{K}):=(\mathfrak{H}(\mathbb{K}), \sqcap,\sqcup,\neg,\lrcorner,\top,\bot)$ is a  pdBa. Moreover, $\underline{\mathfrak{H}}(\mathbb{K})=\underline{\mathfrak{P}}(\mathbb{K})_{p}$.
		
	}
\end{theorem} 

%\noindent Let us recall the left and right semiconcept define in Section \ref{intro}
\noindent For $\underline{\mathfrak{H}}(\mathbb{K})$, consider the sets $\mathfrak{H}(\mathbb{K})_{\sqcap}$ and $ \mathfrak{H}(\mathbb{K})_{\sqcup}$. Note that the elements of the set $\mathfrak{H}(\mathbb{K})_{\sqcap}$ are of the form $(A, A^{\prime})$ and the elements of the set $\mathfrak{H}(\mathbb{K})_{\sqcup}$ are of the form $(B^{\prime}, B)$.
\begin{theorem}
\label{cs}
\noindent
{\rm \cite{wille} \begin{enumerate}
\item $\mathfrak{H}(\mathbb{K}):=\mathfrak{H}(\mathbb{K})_{\sqcap}\cup \mathfrak{H}(\mathbb{K})_{\sqcup}$.
\item $\mathcal{B}(\mathbb{K}):=\mathfrak{H}(\mathbb{K})_{\sqcap}\cap \mathfrak{H}(\mathbb{K})_{\sqcup}$.
\end{enumerate}}
\end{theorem}

\noindent Theorem \ref{cs} (1) thus  implies that a semiconcept is either a left semiconcept or a right semiconcept. Next we move to representation theorem for pdBas.
The notations and results listed below are required.  
%, the following result has been proved. To recall them, first we recall some notations:   
Let $\textbf{D}$ be a dBa. For any $x \in D$, $F_{x}:=\{F\in\mathcal{F}_{p}(\textbf{D})~:~x\in F\}$  and $I_{x}:=\{I\in\mathcal{I}_{p}(\textbf{D})~:~x\in I\}$.
\begin{lemma}
	\label{complement of Fx}
	{\rm \cite{wille,howlader2020} Let  $x\in D$. Then the following hold.
		\begin{enumerate}
			\item $(F_{x})^{c}=F_{\neg x}$ and $(I_{x})^{c}=I_{\lrcorner x}$.
			\item $F_{x\sqcap y}=F_{x}\cap F_{y}$ and $I_{x\sqcup y}=I_{x}\cap I_{y}$.
	\end{enumerate}}
\end{lemma}
\noindent To prove the representation theorem, Wille uses the {\it standard context} corresponding to the dBa $\textbf{D}$, defined as $\mathbb{K}(\textbf{D}):=(\mathcal{F}_{p}(\textbf{D}),\mathcal{I}_{p}(\textbf{D}),\Delta)$, where for all $F\in \mathcal{F}_{p}(\textbf{D}) $ and $I\in \mathcal{I}_{p}(\textbf{D})$, $F\Delta I$  if and only if $F\cap I=\emptyset$. Then we have 
\begin{lemma}
	\label{derivation}
	{\rm \cite{wille} For all $x\in \textbf{D}$, $F_{x}^{\prime}=I_{x_{\sqcap\sqcup}}$ and $I_{x}^{\prime}=F_{x_{\sqcup\sqcap}}$.}
\end{lemma}

%\begin{lemma}
%\label{complement of Fx}
%{\rm \cite{wille, howlader2020} Let $\textbf{D}$ be a {\rm dBa} and $x\in D$. Then the following hold.
%\begin{enumerate}
%    \item $(F_{x})^{c}=F_{\neg x}$ and $(I_{x})^{c}=I_{\lrcorner x}$.
%   \item $F_{x\sqcap y}=F_{x}\cap F_{y}$ and $I_{x\sqcup y}=I_{x}\cap I_{y}$.
%\end{enumerate}}
%\end{lemma}
%\begin{lemma}
%\label{derivation}
%{\rm \cite{wille} Let $\textbf{D}$ be a dBa and $\mathbb{K}(\textbf{D}):=(\mathcal{F}_{p}(\textbf{D}),\mathcal{I}_{p}(\textbf{D}),\Delta)$ then for all $x\in \textbf{D}$, $F_{x}^{\prime}=I_{x_{\sqcap\sqcup}}$ and $I_{x}^{\prime}=F_{x_{\sqcup\sqcap}}$.}
%\end{lemma}

%\begin{theorem} 
%\label{protoembedding}
%{\rm \cite{wille} Let $\textbf{D}$ be a {\rm dBa}. Then the map $h:\textbf{D}\rightarrow \underline{\mathfrak{P}}(\mathbb{K}(\textbf{D})) $ defined by $h(x)=(F_{x},I_{x})$ for all $x\in \textbf{D}$ is a quasi-injective homomorphism.}
%\end{theorem} 

%\noindent In \cite{BALBIANI2012260}, Balbiani shows that this map is an embedding into $\mathfrak{H}(\mathbb{K}(\textbf{D}))$ for pdBa $\textbf{D}$. 

\begin{theorem} 
	\label{semiconceptembedding}
	{\rm \cite{BALBIANI2012260} Let $\textbf{D}$ be a { \rm pdBa}. Then the map $h:\textbf{D}\rightarrow \underline{\mathfrak{H}}(\mathbb{K}(\textbf{D})) $ defined by $h(x)=(F_{x},I_{x})$ for all $x\in \textbf{D}$ is  an injective homomorphism.}
\end{theorem}

\subsection{Rough set theory and approximations}	
\label{Appropefca}

 In rough set theory, 
 %**why this ref** \cite{pawlak1998rough},  knowledge is the ability to classify objects  and is modelled  as a pair  $K:=(W, E)$ consisting of a set $W$ called the universe of discourse, and an equivalence relation $E$ on $W$. $K$ is called a  Pawlakian approximation space, as mentioned in the Introduction. % in rough set theory \cite{pawlak1982rough, pawlak1998rough}. 
 a Pawlakian approximation space $K:=(W, E)$ can be understood to represent a  basic classification skill  of an ``intelligent" agent \cite{pawlak2012rough}. $K$ models ``knowledge", which is this ability to classify objects by agents.
 %The set of equivalence classes $W/E$ is also called knowledge. 
 %**The two notions of knowledge do not contradict each other, as mathematically  they are two mutually interchangeable concepts. 
 In the induced quotient set $W/E$ of equivalence classes due to the equivalence relation $E$, each equivalence class is called a {\it basic category}  of the knowledge given by $K$.  A {\it category}, also called a {\it definable set}, is a union of basic categories.  
 %A category $A$ is also called a {\it definable set} in the Pawlakian approximation space.  
 It is not always the case that a subset $A$ of $W$ is a category. Then the task in rough set theory is to approximate the set $A$ with respect to $K$. This is done by using the lower and upper approximation operators defined in the Introduction.
In the knowledge $K$, the lower approximation $\underline{A}_E$ of the set $A$ is interpreted as the set of all elements of $W$ that can be classified as elements of $A$ with certainty, while 
 $A$'s upper approximation  $\overline{A}^E$ is the set of elements of $W$ that can be possibly classified as elements of $A$. We will omit the subscript and superscript, and denote $\underline{A}_{E}$ by $\underline{A}$, $\overline{A}^{E}$ by $\overline{A}$, if the relation is evident from the context. Recall that  $(W,E)$ is termed a generalised approximation space, when $E$ is any binary relation on $W$.

\begin{proposition}
	\label{pra}
	{\rm \cite{yao1996generalization,pawlak2012rough}
 Let $(W,E)$ be a generalised approximation space, and $A,B \subseteq W$.
		\noindent \blr  \item[{\bf I.}] 	
		%For a generalised approximation space $(W,E)$, 
		The following  hold.
		\item $\overline{A}=(\underline{(A^{c})})^{c}, \underline{A}=(\overline{(A^{c})})^{c}$.
		\item $\underline{W}= W$.
		\item $\underline{A\cap B}=\underline{A}\cap\underline{B}, \overline{A\cup B}=\overline{A}\cup\overline{B}$.
		\item $A\subseteq B$ implies that $\underline{A} \subseteq \underline{B}, \overline{A} \subseteq \overline{B}$.
		
		\item[{\bf II.}] If $E$ is a reflexive and transitive relation, the following hold.
		\item $\underline{A}\subseteq A$ and $A\subseteq \overline{A}$.
		\item $\underline{(\underline{A})}=\underline{A}$ and $\overline{(\overline{A})}=\overline{A}$.
			
		\item[{\bf III.}]
		If  $(W,E)$ is a Pawlakian approximation space,  the following hold.
		\item $\underline{ A} = A$ implies A is a category.
		\item  $\overline{A} = A$ implies A is a category.
		\item  $A\cap B$ is a category, if A,B are categories.
	\elr}
\end{proposition}

Motivated by rough set theory, several authors have introduced approximation operators in FCA.  Let  $\mathbb{K}:=(G, M, I)$ be a context.
 %As mentioned in the Introduction, Kent \cite{RCA} introduced the notion of approximation space in FCA and defined lower and upper approximations of contexts and concepts. He
 Kent \cite{RCA} considered an approximation space $(G, E)$ and defined lower and upper approximation contexts of the relation $I$. Using these contexts, he defined lower and upper approximations of concepts of the context $\mathbb{K}$. 
 
  %used an equivalence relation $E$ on the set of objects $G$ provided by an expert. He defines lower approximation and upper approximation of a context $\mathbb{K}:=(G, M, I)$ using the approximation space $(G, E)$. 
%  **The upper approximation for a concept of $\mathbb{K}$ is defined by the concept of upper approximation of $\mathbb{K}$, whereas the lower approximation for a concept of $\mathbb{K}$  is defined by the concept of lower approximation of $\mathbb{K}$.  
\vskip 2pt 
\noindent On the other hand, Hu et.al.  \cite{hul} defined relations $J_{1}, J_{2}$   on $G$  and $M$ respectively, as follows.
 \bla
\item For $g_{1},g_{2}\in G$, $g_{1}J_{1} g_{2}$ if and only if $I(g_{1})\subseteq I(g_{2})$.
\item For $m_{1}, m_{2}\in M$, $m_{1}J_{2}m_{2}$ if and only if  $I^{-1}(m_{1})\subseteq I^{-1}(m_{2})$.
\ela

\noindent The   relations $J_{1}, J_{2}$ are partial orders. 
%relations \cite{hul}. %Now consider $[g]=:\{g_{1}\in G~:~ gJ_{1}g_{1}\}$ for all $g\in G$ and  $[m]=:\{m_{1}\in M~:~ mJ_{1}m_{1}\}$ for all $m\in M$.
 %In \cite{hul}, i
 It is shown that $UI:=\{(I(g), I(g)^{\prime}) ~:~g\in G\}$ is the set of join irreducible elements of $\mathcal{B}(\mathbb{K})$ and $MI:=\{(I^{-1}(m)^{\prime}, I^{-1}(m))~:~ m\in M\}$ is the set of meet irreducible elements of $\mathcal{B}(\mathbb{K})$. For $A\subseteq G$ ($B\subseteq M$), the lower and upper approximations  of $A$ ($B$) are defined in terms of members of $UI$ ($MI$). 
 %and for, the lower and upper approximations of $B$ are defined in terms of members of $MI$.
\vskip 2pt 
%For a context  $\mathbb{K}:=(G,M, I)$,  
\noindent In \cite{saquer2001concept},   relations $E_{1}, E_{2}$ on $G$  and  $M$ respectively were defined by Saquer et al.   as follows.
\bla
\item For $g_{1},g_{2}\in G$, $g_{1}E_{1} g_{2}$ if and only if $I(g_{1})=I(g_{2})$.
\item For $m_{1}, m_{2}\in M$, $m_{1}E_{2}m_{2}$ if and only if  $I^{-1}(m_{1})= I^{-1}(m_{2})$.
\ela

 \noindent $E_{1}, E_{2}$  are equivalence relations. $A\subseteq G$ and $B\subseteq M$ are called  {\it feasible}, if $A^{\prime\prime}=A$ and $B^{\prime\prime}=B$.
The {\it concept approximations} of $A,B$  are then defined:

%Let $\mathbb{K}:=(G,M, I)$ be a context and recall the approximation spaces $(G, E_{1})$ and $(M, E_{2})$ mentioned in Section \ref{intro}. In \cite{saquer2001concept}, $A\subseteq G$ and $B\subseteq M$ are called {\it feasible} if $A^{\prime\prime}=A$ and $B^{\prime\prime}=B$.  Then the concept approximation(s) of $A$ are defined as follows.

\begin{itemize}
	\item[-]  If  $A$ is feasible, the concept approximation of $A$ is $(A, A^{\prime})$.
	\item[-] If $A$ is not feasible, it is treated as a rough set of the approximation space $(G, E_{1})$, and its concept approximations are constructed using its lower approximation $\underline{A}_{E_{1}}$ and upper approximation $\overline{A}^{E_{1}}$, respectively. The pair $((\underline{A}_{E_{1}})^{\prime\prime}, (\underline{A}_{E_{1}})^{\prime})$ is the  lower concept approximation of $A$, whereas  $((\overline{A}^{E_{1}})^{\prime\prime}, (\overline{A}^{E_{1}})^{\prime})$ is  the upper concept approximation of $A$.

	%	If $A$ is not feasible,  $A$ is considered as s rough set of  the approximation space $(G, E_{1})$, and its  concept approximations are defined with the help of its lower approximation $\underline{A}_{E_{1}}$ and  upper approximation $\overline{A}^{E_{1}}$. The  {\it lower concept approximation}  of $A$ is the pair $((\underline{A}_{E_{1}})^{\prime\prime}, (\underline{A}_{E_{1}})^{\prime})$, while its {\it upper concept approximation} is $((\overline{A}^{E_{1}})^{\prime\prime}, (\overline{A}^{E_{1}})^{\prime})$. 
	
\end{itemize}
%For $B\subseteq M$:
%the concept approximations  are defined as follows.
\begin{itemize}
	\item[-] if $B$ is feasible, the concept approximation of $B$ is $(B^{\prime}, B)$.
	\item[-] if $B$ is not feasible, 
	the lower and upper concept approximations of $B$ are defined as    $((\overline{B}^{E_{2}})^{\prime}, (\overline{B}^{E_{2}})^{\prime\prime})$ and $((\underline{B}_{E_{2}})^{\prime}, (\underline{B}_{E_{2}})^{\prime\prime})$ respectively.
\end{itemize}
If a pair $(A, B)$ is not a concept of the context $\mathbb{K}$, it is called a {\it non-definable} \cite{saquer2001concept} concept. If the extent of a concept approximates $A$ and the intent approximates $B$, it is said to approximate such a pair $(A, B)$.
 The following are the four possible scenarios for $(A,B)$: 
%A pair $(A, B)$ is called a {\it non-definable} concept, if it is not a concept of the context $\mathbb{K}$. A concept is said to approximate such a pair $(A, B)$, if its  extent approximates A and intent approximates B.  The  four possible cases for $A,B$ are considered:  
(i) both $A$ and $B$ are feasible, 
(ii) $A$ is feasible and $B$ is not, 
(iii)  $B$ is feasible and $A$ is not, and 
(iv) both $A$ and $B$ are not feasible.
If both $A$ and $B$ are feasible and $A^{\prime}=B$, the pair $(A, B)$ is a concept  itself, and no approximations are required. In  the other cases,  the lower (upper) approximation of $(A, B)$ is obtained in terms of the meet (join) of the lower (upper) concept approximations of its individual components.
For example, 
%we discuss the case 4.  Let $(A, B)$ be a non-definable concept, where both 
in case (iii) where   $B$ is feasible and $A$ is not, the 
%approximations are defined as follows.
%\vskip 3pt
%\noindent The 
lower  approximation
%of $(A, B)$ is defined by 
$\underline{(A, B)}:=((\underline{A}_{E_{1}})^{\prime\prime}\cap B^{\prime}, ((\underline{A}_{E_{1}})^{\prime\prime}\cap B^{\prime})^{\prime} )$, while 
%\noindent The 
the upper  approximation 
%of $(A, B)$ is defined by 
$\overline{(A, B)}: =(((\overline{A}^{E_{1}})^{\prime}\cap B)^{\prime}, (\overline{A}^{E_{1}})^{\prime}\cap B )$.

\subsection{Pure double Boolean algebras with operators and Kripke contexts}
\label{dbawo}
Let us recall the definitions of pdBa with operators and topological pdBa \cite{howlader2021dbalogic} and its representation results \cite{howlader2021dbalogic}, which are used in Sections \ref{logicmpdbl} and \ref{modalsemantic}.
%The algebraic analogues of the modal systems \textbf{MPDBL} and \textbf{MPDBL4} suggested in Section \ref{logic} are proposed in this section as pdBa with operators and topological pdBa.
%For the algebras, Jonsson-Tarski type representation results are proven, which are used in Section \ref{modalsemantic}. 

%This section proposes two algebraic structures, pdBa with operators and topological pdBa, which are algebraic counterparts of the modal systems \textbf{MPDBL} and \textbf{MPDBL4} proposed in Section \ref{logic}.  J{\'o}nsson-Tarski type representation results are proved for the  algebras,
\begin{definition}
	\label{DBA with operators}
	{\rm \cite{howlader2021dbalogic} A  structure $\mathfrak{O}:=(D,\sqcup,\sqcap,\neg,\lrcorner,\top,\bot, \textbf{I},\textbf{C})$ is a {\it pdBa with operators} (pdBao) provided\\
		1.  $(D,\sqcup,\sqcap,\neg,\lrcorner,\top,\bot)$ is a pdBa and\\ 
		2. $\textbf{I},\textbf{C}$ are monotonic operators on $D$ satisfying the following for any $x,y\in D$.
		\vskip 2pt	
		$\begin{array}{ll}
		1a~ \textbf{I}(x \sqcap y) = \textbf{I}(x) \sqcap \textbf{I}(y) &
		1b~ \textbf{C}(x \sqcup y) = \textbf{C}(x)\sqcup \textbf{C}(y)\\ 
		2a~ \textbf{I}(\neg\bot)= \neg\bot  &
	    2b~ \textbf{C}(\lrcorner\top)=\lrcorner\top \\
		3a~ \textbf{I}(x\sqcap x)=\textbf{I}(x) &
		3b~ \textbf{C}(x\sqcup x)=\textbf{C}(x)
			
		\end{array}$
		\vskip 2pt
		
		\noindent Moreover, a pdBao  %$\mathfrak{O}:=(D,\sqcap,\sqcup,\neg,\lrcorner,\top,\bot,\textbf{I},\textbf{C})$
		is called a {\it topological pdBa} (tpdBa) if the following hold.
		\vskip 2pt
		$\begin{array}{ll}
		4a~  \textbf{I}(x) \sqsubseteq x &  
		4b~ x\sqsubseteq \textbf{C}(x ) \\
		5a~ \textbf{I}\textbf{I}(x)= \textbf{I}(x)&
		5b~ \textbf{C}\textbf{C}(x)=\textbf{C}(x) 	
		\end{array}$
		%\noindent If the underlying dBa is pure,  the dBao is called a {\it pdBao}.\\
	
	\noindent The duals of $\textbf{I}$ and $\textbf{C}$ with respect to $\neg,\lrcorner$ are defined as
		$\textbf{I}^{\delta}(a):=\neg \textbf{I}(\neg a)$ and $\textbf{C}^{\delta}(a):=\lrcorner \textbf{C}(\lrcorner a)$ for all $a\in D$.	}
\end{definition}

 \noindent Some essential features of the operators $\textbf{I}^{\delta},\textbf{C}^{\delta}$ for a tpdBa are:
\begin{lemma}
	\label{tdBadual}
	{\rm \cite{howlader2021dbalogic} Let $\mathfrak{D}$ be a tpdBa. Then for all $a\in D$, %$\textbf{I}^{\delta}\textbf{I}^{\delta}(a)=\textbf{I}^{\delta}(a)$ and $\textbf{C}^{\delta}\textbf{C}^{\delta}(a)=\textbf{C}^{\delta}(a)$.
		\begin{enumerate}
			\item $\textbf{I}^{\delta}\textbf{I}^{\delta}(a)=\textbf{I}^{\delta}(a)$ and $\textbf{C}^{\delta}\textbf{C}^{\delta}(a)=\textbf{C}^{\delta}(a)$.
			\item  $a\sqcap a\sqsubseteq \textbf{I}^{\delta}(a)$ and $\textbf{C}^{\delta}(a)\sqsubseteq a\sqcup a$.
		\end{enumerate} 
	}
\end{lemma}
Recall Definition \ref{K-cntx} of a Kripke context from Section \ref{intro}.
In \cite{howlader2021dbalogic}, we show that for a Kripke context $\mathbb{KC}:=((G,R),(M,S),I)$, 
%if $(G,R)$ is a Pawlakian approximation space,   $-_{R}:\mathcal{P}(G)\rightarrow \mathcal{P}(G) $ is defined as $-_{R}(A):=\underline{A}_{R}$ for every $A\in \mathcal{P}(G)$ (Proposition \ref{pra}), is an interior operator.
%Similarly, if $(M,S)$ is a Pawlakian approximation space,  $-_{S}:\mathcal{P}(M)\rightarrow \mathcal{P}(M) $ is defined as $-_{M}(B):=\underline{B}_{S}$ for all $B\in \mathcal{P}(M)$, is an interior operator.
%From Theorem \ref{protosemialgebra2}, we get the isomorphism $f:\mathcal{P}(G)\rightarrow  \mathfrak{H}(\mathbb{K})_{\sqcap}$ given by $f(A):=(A, A^{\prime})$ for every $A\in \mathcal{P}(G)$ and the anti-isomorphism $g:\mathcal{P}(M)\rightarrow  \mathfrak{H}(\mathbb{K})_{\sqcup}$ given by $g(B):=(B^{\prime}, B)$ for all $B\in \mathcal{P}(M)$. Now taking the compositions of $f,-_{R}$ and $g,-_{S}$,
we can define two unary operators $f_R$ and $f_S$ on $ \mathfrak{H}(\mathbb{K})$  as follows. 

%Utilizing  $f,-_{R}$ on the one hand, one gets the unary operator $f_{R}$ below, which acts as an interior-type operator on $ \mathfrak{P}(\mathbb{K})$. On the other hand, using   $g$ and  $-_{S}$, we get  the unary operator $f_{S}$ which acts as a closure-type operator on $ \mathfrak{P}(\mathbb{K})$.
%a couple of operators 
%we can lift the interior type operator $-_{R}$ to an interior type operator $i$ on $\mathfrak{P}(\mathbb{K})_{\sqcap}$ as follows, $i((A, A^{\prime}))=(\underline{A}_{R},(\underline{A}_{R})^{\prime})$. Dually using the anti-isomorphism$g$, we can lift the interior type operator $-_{S}$
%\noindent to a closure type operator $c$ on $\mathfrak{P}(\mathbb{K})_{\sqcup}$ as follows $c(B^{\prime}, B)=((\underline{B}_{S})^{\prime},\underline{B}_{S})$ for all $(B^{\prime}, B)\in \mathfrak{P}(\mathbb{K})_{\sqcup}$. Moreover, the interior type operator $i$ and the closure type operator $c$ can be extend to an  interior type operator $f_{R}$ and a closure type operator $f_{S}$ on $\mathfrak{P}(\mathbb{K})$. This motivate us the following studies.**\\ 
%Let $\mathbb{KC}:=((G,R),(M,S),I)$ be a Kripke context based on the context $\mathbb{K}:=(G,M,I)$. 
%$\mathfrak{H}(\mathbb{K})$.
%Based on the relations $R$ and $S$, 
%the unary operators $f_{R}, f_{S}$ 
%on the set $ \mathfrak{P}(\mathbb{K})$ are defined 
%as follows. 
For any $(A,B)\in \mathfrak{H}(\mathbb{K})$,
%$\mathfrak{H}(\mathbb{K})$ as follows.
%\blr
%\item $f_{R}:\mathfrak{H}(\mathbb{K})\rightarrow\mathfrak{H}(\mathbb{K})$ defined by $f_{R}((A,B)):=(\underline{A}_{R},(\underline{A}_{R})^{\prime})$, for all $(A,B)\in \mathfrak{H}(\mathbb{K})$.
%\item $f_{S}:\mathfrak{H}(\mathbb{K})\rightarrow\mathfrak{H}(\mathbb{K})$ defined by $f_{S}((A,B)):=((\underline{B}_{S})^{\prime},\underline{B}_{S})$, for all $(A,B)\in \mathfrak{H}(\mathbb{K})$.
%\elr
\bit
\item 
%$f_{R}: \mathfrak{P}(\mathbb{K}) \rightarrow\mathfrak{P}(\mathbb{K})$ defined by 
$f_{R}((A,B)):=(\underline{A}_{R},(\underline{A}_{R})^{\prime})$,
\item 
%$f_{S}:\mathfrak{P}(\mathbb{K})\rightarrow\mathfrak{P}(\mathbb{K})$ defined by 
$f_{S}((A,B)):=((\underline{B}_{S})^{\prime},\underline{B}_{S})$.
\eit

\noindent $f_{R}, f_{S}$ are well-defined, as $(\underline{A}_{R}, (\underline{A}_{R})^{\prime})$ and $((\underline{B}_{S})^{\prime},\underline{B}_{S})$ are both   semiconcepts  of $\mathbb{K}$. This implies that the set $\mathfrak{H}(\mathbb{K})$ of semiconcepts  is closed under the operators $f_{R}, f_{S}$. We have

\begin{definition}
	\label{full-complexalg}
	{\rm \cite{howlader2021dbalogic} Let $\mathbb{KC}:=((G,R),(M,S),I)$ be a Kripke context. The {\it  complex algebra} of $\mathbb{KC}$, $\underline{\mathfrak{H}}^{+}(\mathbb{KC}):=(\mathfrak{H}(\mathbb{K}),\sqcup,\sqcap,\neg,$ $\lrcorner,\top,\bot,f_{R},f_{S})$, is the expansion  of the algebra $\underline{\mathfrak{H}}(\mathbb{K})$ of semiconcepts with the operators $f_{R}$ and $f_{S}$. 
		%The full complex algebra of $\mathbb{KC}$ is denoted by . \\
		%Any subalgebra of  $\underline{\mathfrak{P}}^{+}(\mathbb{KC})$ is called a {\it complex algebra} of $\mathbb{KC}$. 
	}
\end{definition}

Let $f_{R}^{\delta}, f_{S}^{\delta}$ denote the operators on $ \mathfrak{P}(\mathbb{K})$ that are {\it dual} to $f_{R}, f_{S}$ respectively. In other words, for each $x:=(A,B)\in \mathfrak{H}(\mathbb{K})$, $f_{R}^{\delta}(x):=\neg f_{R}(\neg x)=\neg f_{R}((A^{c}, A^{c\prime}))=\neg (\underline{A}^{c}_{R},(\underline{A}^{c}_{R})^{\prime})=((\underline{A}^{c}_{R})^{c},(\underline{A}^{c}_{R})^{c\prime})=(\overline{A}^{R},(\overline{A}^{R})^{\prime})$, by Proposition \ref{pra}(i).\\
Similarly  $f_{S}^{\delta}(x):=\lrcorner f_{S}(\lrcorner x)=((\overline{B}^{S})^{\prime}, \overline{B}^{S})$. 
Again, note that $f_{R}^{\delta}(x) = (\overline{A}^{R}, (\overline{A}^{R})^{\prime})$ and $f_{S}^{\delta}(x) = ((\overline{B}^{S})^{\prime},\overline{B}^{S})$ are   semiconcepts of $\mathbb{K}$.

%In the following let $\mathbb{KC}:=((G,R),(M,S),I)$ be any Kripke context. 

%Next theorem  shows that $ \mathfrak{H}(\mathbb{K})$ generates a pdBao when $f_R$ and $f_{S}$ are taking account. 

%As intended, the sets of semiconcepts of a context provide examples of  pdBaos. 
\begin{theorem}
	\label{complex algebra}
	{\rm \cite{howlader2021dbalogic} Let $\mathbb{KC}:=((G,R),(M,S),I)$ be a Kripke context based on the context $\mathbb{K}:=(G, M, I)$, $\underline{\mathfrak{H}}^{+}(\mathbb{KC}):=(\mathfrak{H}(\mathbb{K}),\sqcup,\sqcap,\neg,\lrcorner,\top,\bot,f_{R},f_{S})$ is a pdBao.
		
	}
\end{theorem}
Different kinds of Kripke contexts have been defined, depending on the properties of the relations giving the Kripke contexts.
\begin{definition}
	{\rm \cite{howlader2021dbalogic} \noindent Let  $\mathbb{KC}:=((G,R),(M,S),I)$ be a Kripke context. 
		\begin{enumerate} 
			\item $\mathbb{KC}$ is {\it reflexive  from the left}, if $R$ is reflexive. 
			\item $\mathbb{KC}$ is {\it reflexive from the right}, if $S$ is reflexive. 
			\item $\mathbb{KC}$ is {\it reflexive}, if it is reflexive from both left and right.
			%Let $\mathbb{KC}=((G,R),(M,S),I)$ be a Kripke context.
			%		\begin{enumerate}
			%			\item $\mathbb{KC}$ is reflexive from the left if $R$ is reflexive and is reflexive from the right if $S$ is reflexive. $\mathbb{KC}$ is reflexive if it is reflexive from both left and right.
			%			\item $\mathbb{KC}$ is symmetric from the left if $R$ is symmetric and is symmetric from the right if $S$ is symmetric. $\mathbb{KC}$ is symmetric if it is symmetric from both left and right.
			%			\item $\mathbb{KC}$ is transitive from left if $R$ is transitive and is transitive from right if $S$ is transitive. $\mathbb{KC}$ is transitive if it is transitive from both left and right.			
			%\item $\mathbb{KC}$ is right-unbounded form left if $R$ is right-unbounded and is right-unbounded form right if $S$ is right-unbounded. $\mathbb{KC}$ is right-unbounded if it is right-unbounded from both left and right.			
			%\item $\mathbb{KC}$ is no branching to the right form left if $R$ is no branching to the right and is no branching to the right form right if $S$ is no branching to the right. $\mathbb{KC}$ is no branching to the right if it is no branching to the right from both left and right.			
		\end{enumerate}
		The cases for {\it symmetry} and {\it transitivity} of $\mathbb{KC}$ are similarly defined.}
\end{definition}
%\noindent For examples of Kripke context, we refer to our work \cite{howlader2021dbalogic}. We get a class of examples of tpdBas  from the sets of semiconcepts of  the underlying contexts of a symmetric  and transitive Kripke context, just as we expected. 

%Observe that  the Kripke context in Example \ref{rt-s-kcxt} is symmetric from the right.
%Again, as intended, we obtain a class of examples of topological dBas from the sets of protoconcepts and semiconcepts of contexts. 

\begin{theorem}
	\label{pure complex algebra]}
	{\rm \cite{howlader2021dbalogic} Let $\mathbb{KC}:=((G,R),(M,S),I)$ be a reflexive and transitive Kripke context. Then 
		$\underline{\mathfrak{H}}^{+}(\mathbb{KC})$  is a tpdBa.
	}
\end{theorem}

\noindent The following are introduced in order to prove representation theorems for  pdBaos in \cite{howlader2021dbalogic}. 
%\subsubsection{Representation theorems for pdBaos and tpdBas}
%\label{RtdBaos}
%The representation theorems for dBao are  proved using  Wille's representation theorems for dBa. 
For every pdBao $\mathfrak{O}:=(D,\sqcup,\sqcap,\neg,\lrcorner,\top,\bot, \textbf{I}, \textbf{C})$, we construct a Kripke context based on the standard context  $\mathbb{K}(\textbf{D}):=(\mathcal{F}_{p}(\textbf{D}),\mathcal{I}_{p}(\textbf{D}), \Delta)$ corresponding to the underlying pdBa \textbf{D}. For that,  relations $R$ and $S$ are defined  on $\mathcal{F}_{p}(\textbf{D})$ and  $\mathcal{I}_{p}(\textbf{D})$ respectively
%. The relations $R$ and $S$ are defined 
as follows.
\begin{itemize}
	\item[-] For all $u,u_{1}\in \mathcal{F}_{p}(\textbf{D})$, $uRu_{1}$ if and only if $\textbf{I}^{\delta}(a)\in u$ for all $a\in u_{1}$.
	\item[-] For all $v,v_{1}\in \mathcal{I}_{p}(\textbf{D})$, $vSv_{1}$ if and only if $\textbf{C}^{\delta}(a)\in v$ for all $a\in v_{1}$.
\end{itemize}
%To prove representation theorems for dBao, we will use 
The following results are required to get (Representation) Theorem \ref{rtdBao}.
%\begin{lemma}
%	\label{req}
%	{\rm 
		%For any \textbf{D} dBa, 		
		%\begin{enumerate}
		%\item 
%		If $F$ is a primary filter (ideal) of a dBa \textbf{D}, then for any $x\in D$,  exactly one of the elements $x$ and $\neg x$ belongs to $F$.
		%\item If $I$ is a primary ideal then for all $x\in D$, exactly one of the elements $x$ and $\lrcorner x$ belongs to $I$.
		%\end{enumerate}
%	}
%\end{lemma}
%\begin{proof}
%	Proof follows from the definition of a primary filter (ideal).  
%\end{proof}

\begin{lemma}
	\label{canonical relations}
	{\rm \cite{howlader2021dbalogic} Let $\mathfrak{O}:=(D,\sqcup,\sqcap,\neg,\lrcorner,\top,\bot, \textbf{I}, \textbf{C})$ be a pdBao. The following hold.
		\begin{enumerate}
			\item  For all $u,u_{1}\in \mathcal{F}_{p}(\textbf{D})$, $uRu_{1}$ if and only if for all $a\in D$, $\textbf{I}a\in u$ implies that $a\in u_{1}$.
			\item  For all $v,v_{1}\in \mathcal{I}_{p}(\textbf{D})$, $vSv_{1}$ if and only if for all $a\in D$, $\textbf{C}a\in v$ implies that $a\in v_{1}$.
	\end{enumerate}}
\end{lemma}
\begin{lemma}
	\label{canonical box and dimon} 
	{\rm \cite{howlader2021dbalogic} Let $\mathfrak{O}$ be a pdBao and $\mathbb{KC}(\mathfrak{O}):=((\mathcal{F}_{p}(\textbf{D}),R), (\mathcal{I}_{p}(\textbf{D}),S),\Delta)$. For all $a\in D$: 
		\begin{enumerate}
			\item  $\overline{F_{a}}^{R}=F_{\textbf{I}^{\delta}(a)}$ and $\underline{F_{a}}_{R}=F_{\textbf{I}(a)}$.
			\item  $\overline{I_{a}}^{S}=I_{\textbf{C}^{\delta}(a)}$ and $\underline{I_{a}}_{S}=I_{\textbf{C}(a)}$.
	\end{enumerate}}
\end{lemma}

The Kripke context $\mathbb{KC}(\mathfrak{O})$ of Lemma \ref{canonical box and dimon} is used to obtain the representation theorem.
%we show that $\mathfrak{O}$ is isomorphic to complex algebra of the Kripke context $\mathbb{KC}(\mathfrak{O})=((\mathcal{F}_{p}(\textbf{D}),R),(\mathcal{I}_{p}(\textbf{D}), S), \Delta)$ that is we prove the following representation theorem.
\begin{theorem}[Representation theorem]
	\label{rtdBao}
	{\rm \cite{howlader2021dbalogic} A pdBao $\mathfrak{O}:=(D,\sqcup,\sqcap,\neg,\lrcorner,\top,\bot, \textbf{I}, \textbf{C})$ is embeddable into the  complex algebra $\underline{\mathfrak{H}}^{+}(\mathbb{KC}(\mathfrak{O})):=(\mathfrak{H}(\mathbb{K}(\mathfrak{O})),\sqcup,\sqcap,\neg,\lrcorner,\top,\bot,f_{R},f_{S})$ of the Kripke context $\mathbb{KC}(\mathfrak{O})$. The map  $h:D\rightarrow \mathfrak{H}(\mathbb{K}(\textbf{D}))$ defined by $h(x):=(F_{x},I_{x})$ for all $x \in D$,  is the required embedding.
		%\item If  $\mathfrak{O}$ is a contextual dBao then the quasi-embedding $h$ is an embedding.
		%	\item  $\mathfrak{O}_{p}$ is embeddable into the largest pure subalgebra $\underline{\mathfrak{H}}^{+}(\mathbb{KC}(\mathfrak{O}))$ of $\underline{\mathfrak{P}}^{+}(\mathbb{KC}(\mathfrak{O}))$.
		%\end{enumerate}
	}
\end{theorem}
For a tpdBa $\mathfrak{O}$, we now have
\begin{theorem}
	\label{rttdBa}
	{\rm \cite{howlader2021dbalogic} 
		%Let $\mathfrak{O}$ be a topological dBa. Then  
		$\mathbb{KC}(\mathfrak{O}):=((\mathcal{F}_{p}(\textbf{D}), R), (\mathcal{I}_{p}(\textbf{D}), S), \Delta)$  is a reflexive  and transitive Kripke context.}
\end{theorem}
%\begin{proof}
	%Let $\mathfrak{O}$ be a topological dBa. We show that $\mathbb{KC}(\mathfrak{O}):=((\mathcal{F}_{p}(\textbf{D}), R), (\mathcal{I}_{p}(\textbf{D}), S), \Delta)$  is a reflexive  and transitive Kripke context.  
%	To show  $R$ is reflexive, let $F\in \mathcal{F}_{p}(\textbf{D})$ and $\textbf{I}a\in F$ for some $a\in D$. By Definition \ref{tdBa}(4a), $\textbf{I}a\sqsubseteq a$, which implies that $a\in F$, as $F$ is a filter. So $FRF$ by Lemma \ref{canonical relations}.\\
%	To show $R$ is transitive, let $F,F_{1}, F_{2}\in \mathcal{F}_{p}(\mathfrak{O})$ such that $FRF_{1}$ and $F_{1}RF_{2}$. We show that $FRF_{2}$. Let $a\in F_{2}$. Then $\textbf{I}^{\delta}(a)\in F_{1}$, as $F_{1}RF_{2}$, which implies that $\textbf{I}^{\delta}\textbf{I}^{\delta}(a)\in F$, as $FRF_{1}$. So $\textbf{I}^{\delta}(a)=\textbf{I}^{\delta}\textbf{I}^{\delta}(a)\in F$, using Lemma \ref{tdBadual}. Thus $FRF_{2}$. \\
	%Therefore $R$ is reflexive and transitive.\\
%	Similarly, one can show that $S$ is reflexive and transitive.   
%\end{proof}
We get the representation results for tpdBas in terms of reflexive and transitive Kripke contexts by combining Theorems \ref{rtdBao} and \ref{rttdBa}. 

%Combining Theorem \ref{rtdBao} and Theorem \ref{rttdBa}, we get the representation results for topological dBas in terms of reflexive  and transitive Kripke contexts.
\begin{theorem}
	\label{rtdBa}
	{\rm \cite{howlader2021dbalogic} A  tpdBa $\mathfrak{O}$ is embeddable into the  complex algebra $\underline{\mathfrak{H}}^{+}(\mathbb{KC}(\mathfrak{O}))$ of the reflexive and transitive Kripke context $\mathbb{KC}(\mathfrak{O})$. 
		%\\$\mathfrak{O}_{p}$ is embeddable into the complex algebra   $\underline{\mathfrak{H}}^{+}(\mathbb{KC}(\mathfrak{O}))$ of $\mathbb{KC}(\mathfrak{O})$. Moreover,
		%	\begin{enumerate}
		%			\item If  $\mathfrak{O}$ is a topological contextual dBa then  $\mathfrak{O}$ is embeddable into  $\underline{\mathfrak{P}}^{+}(\mathbb{KC}(\mathfrak{O}))$. 
		%		\item If   $\mathfrak{O}$ is a topological pdBa  then  $\mathfrak{O}$  is embeddable into the complex algebra  $\underline{\mathfrak{H}}^{+}(\mathbb{KC}(\mathfrak{O}))$ of $\mathbb{KC}(\mathfrak{O})$.
		%\end{enumerate}
	}
\end{theorem}

\section{The logic PDBL and its modal  extensions}
\label{logic}
%We next  formulate the logic \textbf{PDBL}
%for  contextual dBas. The logic \textbf{MPDBL} for the class of pdBaos, and its extension  \textbf{MPDBL4} for topological pdBas and \textbf{MPDBL5} are both  defined with  \textbf{PDBL} as their base. In Section \ref{semisemantics}, it is shown that, apart from the algebraic semantics, the logics can be imparted a introduced a two sided approach to the relational semantics based on the context, due  to the representation theorems for the respective classes of algebras obtained in Sections \ref{preli}.
We now formulate the logic \textbf{PDBL}  for pdBas, followed by the modal extensions
%\textbf{PDBL} is the base for the logic
 \textbf{MPDBL}  for the class of pdBaos and  
 %and its extensions
  \textbf{MPDBL4} for tpdBas. %and \textbf{MPDBL5}. 
%Due to the representation theorems for the respective classes of algebras obtained in Sections 
%Due to the representation theorems obtained in Section  \ref{preli} for the respective classes of algebras, it is shown in Section \ref{semisemantics} that, in addition to the algebraic semantics, a two-sorted approach to relational semantics based on contexts can be given to the logics. 
\subsection{PDBL}
\label{ppdBl}
Recall that a pdBa is obtained from a dBa $\textbf{D}$, by adding the property: for all $a\in D$, $a=a\sqcap a~\mbox{or}~a=a\sqcup a$   in $\textbf{D}$ (Definition \ref{DBA}). 
%What this property say?  This property  say that for any  $a\in D$,  the **conjunction** disjunction  $((a=a\sqcap a)~\mbox{or}~(a=a\sqcup a))$ is valid in the class of pdBa . Now o
By Proposition \ref{pro1.5}(5), for all $a\in D$, $a\sqcap a\sqsubseteq a$ and $a\sqsubseteq a\sqcup a$ are anyway valid in the class of pdBas. So   $((a=a\sqcap a)~\mbox{or}~(a=a\sqcup a))$ is valid in the class of pdBas if and only if $((a\sqsubseteq a\sqcap a) ~\mbox{or}~ (a\sqcup a)\sqsubseteq a))$ is valid in the class. 
%As we want to develop a calculus for the class of pdBas, it is natural to expect that all the valid statement in the class of pdBa is also provable in the calculus.
Thus \textbf{PDBL} will need to be a system in which  the last disjunction is provable. To facilitate this,
%$((a\sqsubseteq a\sqcap a) ~\mbox{or}~ (a\sqcup a)\sqsubseteq a))$ is derivable. As this is a disjunction of statements, 
one makes use of hypersequent calculus. 

% The thing to note here is that the conjunction  $((a\sqsubseteq a\sqcap a) ~\mbox{or}~ (a\sqcup a)\sqsubseteq a))$ consists of  more than one inequalities. The calculus \textbf{PDBL} cannot be develop by using the  sequent calculus, as it allows only a single sequent as a succedent  of a proof. Here, we need a  calculus  so that it allows a succedent not to be only a single sequent but also more than one sequent. Such a generalization of the sequent calculus is hypersequent calculus.  
Hypersequents were  introduced by Pottinger \cite{pottinger1983uniform} and independently studied by Avron \cite{avron1987constructive}. 
For a language \textbf{L}, a hypersequent \cite{avron1996method,PaoliF} is a  finite sequence of sequents of the form:  
\[\Gamma_{1}\vdash\Delta_{1}\mid\Gamma_{2}\vdash\Delta_{2}\mid\ldots\mid\Gamma_{n}\vdash\Delta_{n},\] where $\Gamma_{i}$ and $\Delta_{i}$ are finite sequences of formulae in \textbf{L}. $\Gamma_{i}\vdash\Delta_{i}$ are called the {\it components} of the hypersequent.  
%The components of a hypersequent are ordinary sequents. 
A hypersequent is called {\it empty}  when $n=0$.  It is called  {\it single-conclusioned} or an {\it s-hypersequent}, if $\Delta_{i}$ consists of a single formula for all $i$. 
For \textbf{PDBL}, a special case of s-hypersequents is considered -- 
%. In particular, we are interested in s-hypersequent, 
$\Gamma_{i}$ also consists of a single formula for all $i$. The intended interpretation of an s-hypersequent, that is, the definition of its satisfaction   by a valuation (given formally in Definition \ref{satis-hyper-sequent} below), is  in terms of the satisfaction  of {\it any of} its  components by the valuation. This fact enables us to deal with the disjunction of statements giving the pdBa axiom. 

The language $\mathfrak{L}$ of   \textbf{PDBL}  consists of a countably infinite set $\textbf{OV}:=\{ p,q,r,\ldots\}$ of object variables, a countably infinite set $\textbf{PV}:=\{ P,Q,R,\ldots\}$ of property variables, propositional constants  $\bot,\top$, and logical connectives $\sqcup,\sqcap,\neg,\lrcorner$. The set $\mathfrak{F}$ of formulae is given by the following  scheme:  
\[\top~|~\bot~|~p~|~P~|~\alpha\sqcup\beta~|~ \alpha\sqcap\beta~|~\neg\alpha~|~\lrcorner\beta\] where $p\in\textbf{OV}$ and $P\in \textbf{PV}$.
$\vee$ and $\wedge$ are definable connectives:
$\alpha\vee\beta:=\neg(\neg\alpha\sqcap\neg\beta)$ and $\alpha\wedge\beta:=\lrcorner(\lrcorner\alpha\sqcup\lrcorner\beta)$  for all $\alpha,\beta\in\mathfrak{F} $.
A {\it sequent} in {\rm \textbf{PDBL}} is a pair of formulae denoted by $\alpha\vdash\beta$ for $\alpha,\beta\in \mathfrak{F}$. 
%For $\alpha,\beta\in \mathfrak{F}$, 
If $\alpha\vdash\beta$ and  $\beta\vdash\alpha$, we use the abbreviation $\alpha\dashv\vdash\beta$.  
%The language of  \textbf{PDBL} consists of  propositional constants $\bot,\top$, a set \textbf{PV} of propositional variables, and logical connectives $\sqcup,\sqcap,\neg,-$. The set  $\mathfrak{F}$ of all formulae of \textbf{PDBL} is  given by the scheme:
%$\bot \mid  \top \mid  p\mid  a \sqcup \b \mid  a \sqcap b\mid  \neg a\mid  - a,$
%where $p \in$ \textbf{PV}.
%: \textbf{PV}:=\{p,q,r,......\}. \\
%Two constant symbols: $\bot ~\textpipe~\top$.\\
%Logical connectives: $\sqcup~ \textpipe ~\sqcap~\textpipe~\neg~\textpipe~- ~$.\\
%Formula of the {\rm \textbf{DBL}  } denoted by small Greek latter that is $\alpha, \beta, \gamma$ etc. and  define inductively that is : All propositional variable and constant symbols are formula. If $\alpha$
%and $\beta$ are formula then $\alpha\sqcup\beta$, $\alpha\sqcap\beta$,$\neg\alpha$,$-\beta$ are formula and that's it. The set all formula denoted by $\mathfrak{F}$.\\
%$\vee$ and $\wedge$ are definable connectives:  
%  $\alpha\vee\beta:=\neg(\neg\alpha\sqcap\neg\beta)$ and $\alpha\wedge\beta:=-(-\alpha\sqcup-\beta)$, for $\alpha,\beta\in\mathfrak{F}$.\\
%  A {\it sequent} in \textbf{PDBL} with formulae $\alpha,\beta \in \mathfrak{F}$ is denoted in the usual manner as $\alpha\vdash\beta$. $\alpha\dashv\vdash\beta$ will be used as abbreviation for   ($\alpha\vdash\beta$ and  $\beta \vdash \alpha$).   
\begin{definition}
	{\rm An s-hypersequent in \textbf{PDBL} is a   finite sequence of sequents of the form:\[\alpha_{1}\vdash\beta_{1}\mid\alpha_{2}\vdash\beta_{2}\mid\ldots\mid\alpha_{n}\vdash\beta_{n},\] where for all $i\in\{1,2,\ldots, n\}$, $\alpha_{i},\beta_{i}\in \mathfrak{F}$.
	}
\end{definition}
We use $B, C, D,\ldots$ as  meta variables for s-hypersequents. Note that a sequent $\alpha\vdash\beta$ is a special case of an s-hypersequent $B\mid\alpha\vdash\beta\mid D$,  where  $B$ and $D$ are empty s-hypersequents.
\vskip 3pt

\noindent The {\it axioms}  of $\textbf{PDBL}$ are given by the following schema. $\alpha,\beta,\gamma\in \mathfrak{F}$, $p\in \textbf{OV}$, $P\in\textbf{PV}$.

1 $\alpha\vdash\alpha$.
\vskip 2pt

\noindent {\it Axioms for $\sqcap$ and $\sqcup$}: 

$\begin{array}{ll}

    2a~ \alpha\sqcap\beta\vdash \alpha &
	2b~ \alpha\vdash\alpha\sqcup\beta \\
	
	3a~  \alpha\sqcap\beta\vdash \beta &
	3b~ \beta\vdash\alpha\sqcup\beta \\
	
	4a~ \alpha\sqcap\beta\vdash (\alpha\sqcap\beta)\sqcap(\alpha\sqcap\beta) &
	4b~ (\alpha\sqcup\beta)\sqcup(\alpha\sqcup\beta)\vdash\alpha\sqcup \beta 
	\end{array}$
\vskip 2pt
	
\noindent {\it Axioms for  $\neg$ and $\lrcorner$}:

$\begin{array}{ll}

   5a~\neg(\alpha\sqcap \alpha)\vdash\neg \alpha &
	5b~\lrcorner \alpha\vdash\lrcorner(\alpha\sqcup \alpha)\\
		6a~\alpha\sqcap\neg \alpha\vdash \bot &
	6b~\top\vdash \alpha\sqcup \lrcorner \alpha\\
	7a~\neg\neg(\alpha\sqcap \beta)\dashv\vdash(\alpha\sqcap \beta)&
	7b~\lrcorner\lrcorner(\alpha\sqcup \beta)\dashv\vdash(\alpha\sqcup \beta)
\end{array}$
\vskip 2pt

\noindent {\it Generalization of the law of absorption}:

$\begin{array}{ll}

8a~\alpha\sqcap \alpha\vdash \alpha\sqcap(\alpha\sqcup \beta)&
	8b~\alpha\sqcup (\alpha\sqcap \beta)\vdash \alpha\sqcup \alpha\\
    9a~\alpha\sqcap \alpha\vdash \alpha\sqcap(\alpha\vee \beta)&
	9b~\alpha\sqcup(\alpha\wedge \beta)\vdash \alpha\sqcup \alpha\\
	\end{array}$
\vskip 2pt
	
\noindent  {\it Laws of distribution}:

	$\begin{array}{ll}
	
	10a~\alpha\sqcap(\beta\vee \gamma)\dashv\vdash(\alpha\sqcap \beta)\vee(\alpha\sqcap \gamma) &
	10b~\alpha\sqcup (\beta\wedge \gamma)\dashv\vdash(\alpha\sqcup \beta)\wedge(\alpha\sqcup \gamma)\\
	
\end{array}$
\vskip 2pt

\noindent  {\it Axioms for} $\bot,\top$:
 
$\begin{array}{ll}
	
	11a~  \bot\vdash \alpha &
	11b~ \alpha\vdash \top\\
	12a~\neg\top\vdash\bot &
	12b~\top\vdash\lrcorner\bot\\
	13a~\neg \bot\dashv\vdash \top\sqcap\top &
	13b~\lrcorner\top\dashv\vdash\bot\sqcup\bot\\

\end{array}$	
\vskip 2pt

\noindent {\it The compatibility 	 axiom}:

14  $(\alpha\sqcup \alpha)\sqcap (\alpha\sqcup \alpha)\dashv\vdash (\alpha\sqcap \alpha)\sqcup (\alpha\sqcap \alpha)$

\noindent  {\it Special axioms for} variables:
$\begin{array}{ll}

%1a~  \bot\vdash \alpha &
%1b~ \alpha\vdash \top\\

%2a~ \alpha\sqcap\beta\vdash \alpha &
%2b~ \alpha\vdash\alpha\sqcup\beta \\

%3a~  \alpha\sqcap\beta\vdash \beta &
%3b~ \beta\vdash\alpha\sqcup\beta \\

%4a~ \alpha\sqcap\beta\vdash (\alpha\sqcap\beta)\sqcap(\alpha\sqcap\beta) &
%4b~ (\alpha\sqcup\beta)\sqcup(\alpha\sqcup\beta)\vdash\alpha\sqcup \beta \\

%5a~\alpha\sqcap \alpha\vdash \alpha\sqcap(\alpha\sqcup \beta)&
%5b~\alpha\sqcup (\alpha\sqcap \beta)\vdash \alpha\sqcup \alpha\\

%6a~\neg(\alpha\sqcap \alpha)\vdash\neg \alpha &
%6b~\lrcorner \alpha\vdash\lrcorner(\alpha\sqcup \alpha)\\

%7a~\alpha\sqcap\neg \alpha\vdash \bot &
%7b~\top\vdash \alpha\sqcup \lrcorner \alpha\\

%8a~\neg \bot\dashv\vdash \top\sqcap\top &
%8b~\lrcorner\top\dashv\vdash\bot\sqcup\bot\\

%9a~\alpha\sqcap \alpha\vdash \alpha\sqcap(\alpha\vee \beta)&
%9b~\alpha\sqcup(\alpha\wedge \beta)\vdash \alpha\sqcup \alpha\\

%10a~\alpha\sqcap(\beta\vee \gamma)\dashv\vdash(\alpha\sqcap \beta)\vee(\alpha\sqcap \gamma) &
%10b~\alpha\sqcup (\beta\wedge \gamma)\dashv\vdash(\alpha\sqcup \beta)\wedge(\alpha\sqcup \gamma)\\

%11a~\neg\neg(\alpha\sqcap \beta)\dashv\vdash(\alpha\sqcap \beta)&
%11b~\lrcorner\lrcorner(\alpha\sqcup \beta)\dashv\vdash(\alpha\sqcup \beta)\\

%12a~\neg\top\vdash\bot &
%12b~\top\vdash\lrcorner\bot\\

%13~ \alpha\vdash\alpha &
%14~ (\alpha\sqcup \alpha)\sqcap (\alpha\sqcup \alpha)\dashv\vdash (\alpha\sqcap \alpha)\sqcup (\alpha\sqcap \alpha)\\
15a~ p\sqcap p\dashv\vdash p&
15b~ P\sqcup P\dashv\vdash P
\end{array}$

\vskip 3pt
\noindent 
%One thing to mention is that except for axioms 15a and 15b, all other \textbf{PDBL}
 Axioms 1-14 are lookalikes   of the axioms defining \textbf{CDBL} \cite{howlader2021dbalogic}, the logic for cdBas -- this is expected because a pdBa is also a cdBa. %The axioms 15a and 15b occur for the syntactical difference between   \textbf{PDBL} and  \textbf{CDBL}.
 However, it may be pointed out that \textbf{PDBL} and \textbf{CDBL} are different syntactically, the fact also begin highlighted by the \textbf{PDBL} axioms 15a and 15b.
\vskip 3pt
\noindent {\it Rules of inference of {\bf PDBL}} are as follows. $B,C,D,E, F, G, H, X$ are possibly empty s-hypersequents.
\vskip 3pt
\noindent {\it  For $\sqcap$ and $\sqcup$}:
\begin{center}
	$	\begin{array}{ll}
	\infer[(R1)]{B\mid\alpha\sqcap \gamma\vdash \beta\sqcap \gamma\mid  C}{B\mid  \alpha\vdash \beta\mid  C} &
	\infer[(R1)^{\prime}]{B\mid  \gamma\sqcap\alpha \vdash \gamma\sqcap \beta\mid  C}{B\mid  \alpha\vdash \beta\mid  C}\\
	\infer[(R2)]{B\mid  \alpha\sqcup \gamma\vdash \beta\sqcup \gamma\mid  C}{B\mid  \alpha\vdash \beta\mid  C}&
	\infer[(R2)^{\prime}]{B\mid  \gamma\sqcup\alpha \vdash\gamma \sqcup \beta\mid  C}{B\mid  \alpha\vdash \beta\mid  C}
	\end{array}$
	\end{center}
	
	\noindent {\it  For $\neg, \lrcorner$}:
	\begin{center}
	   $\begin{array}{ll}
	\infer[(R3)]{B\mid  \neg\beta\vdash \neg\alpha\mid  C}{B\mid  \alpha\vdash \beta\mid  C}&
	\infer[(R3)^{\prime}]{B\mid  \lrcorner\beta\vdash \lrcorner\alpha\mid  C}{B\mid  \alpha\vdash \beta\mid  C}\end{array}$
	\end{center}
	
	\noindent	{\it Transitivity}:
	\begin{center}
	$\infer[(R4)]{B\mid  D\mid  \alpha\vdash\gamma\mid  C\mid  E}{B\mid  \alpha\vdash\beta\mid  C & D\mid  \beta\vdash\gamma\mid  E}$
	\end{center}
	\begin{center}
	$\infer[(Sp)]{\alpha\vdash\alpha\sqcap\alpha\mid  \alpha\sqcup\alpha\vdash\alpha}{}$
	\end{center}
	%\infer[(R9)]{p\vdash \alpha\sqcap\alpha~|~p\vdash\neg\alpha}{}&
	% \infer[(R10)]{\alpha\sqcup \alpha\vdash P~|~\lrcorner\alpha\vdash P}{}\\
	% \infer[(R7)]{ D\mid B\mid \alpha\vdash\bot\mid E\mid C}{D\mid \alpha\vdash\alpha\sqcap\alpha\mid E&  B\mid \alpha\vdash p\mid C} &
	% \infer[(R8)]{D\mid B\mid\top\vdash\alpha\mid \alpha\vdash P\mid E\mid C}{ D\mid \alpha\sqcup\alpha\vdash\alpha\mid E & B\mid P\vdash\alpha\mid C} 
	
	\begin{center}
	    {\small $\infer[\hspace*{-4pt}(R5)]{B\mid  D\mid  F\mid  H\mid  \alpha\vdash\beta\mid  C\mid  E\mid  G\mid  X}{B\mid  \alpha\sqcap\beta\vdash\alpha\sqcap\alpha\mid  C & D\mid  \alpha\sqcap\alpha\vdash\alpha\sqcap\beta\mid  E & F\mid  \alpha\sqcup\beta\vdash\beta\sqcup\beta\mid  G & H\mid  \beta\sqcup\beta\vdash\alpha\sqcup\beta\mid  X}$}
\end{center}

\vskip 2pt
\noindent We shall see in the sequel that $(Sp)$ corresponds to the defining axiom for pdBas, while $(R5)$ captures the order relation of the  pdBas.
\vskip 3pt
%\noindent Except for (Sp), all other \textbf{PDBL} rules are look alike as that of  \textbf{CDBL} \cite{howlader2021dbalogic} and the only difference between the \textbf{PDBL} and \textbf{CDBL} rules are syntactical and  the \textbf{PDBL} version allows for extra side components. Moreover, \textbf{PDBL} has some  
{\it External rules of inference}:
\begin{center}
	%**Use arrays!!**
	$\begin{array}{ll}
	\infer[\mbox{(External contraction-EC)}]{B\mid D\mid  C}{B\mid  D\mid  D\mid  C}&
	\infer[\mbox{(External exchange-EE)}]{B\mid  E\mid  D\mid  C}{B\mid  D\mid  E\mid  C}\\
	\infer[\mbox{(External weakening-EW)}]{B\mid C}{B}
	\end{array}$
\end{center}
\noindent  Derivability is defined in the standard manner: an s-hypersequent S is {\it derivable} (or
{\it provable}) in \textbf{PDBL}, if there exists a finite sequence of s-hypersequents $S_{1},. . ., S_{m}$
such that $S_{m}$ is the s-hypersequent S and for all $k \in\{1,. . .,m\}$ either $S_{k}$ is an axiom or $S_{k}$ is obtained by applying  rules of  \textbf{PDBL} to elements from $\{S_{1} ,. . . , S_{k-1}\}$. 
%The following results are easy to obtain.
Let us give a few examples of derived rules and sequents.

\begin{proposition}
	\label{drive rule1}
	{\rm  The following rules are derivable in  \textbf{PDBL}.
		\begin{multicols}{2}
			\item $\infer[(R6)]{B\mid  D\mid  \alpha\sqcap\alpha\vdash\beta\sqcap\gamma\mid  C\mid  E}{B\mid  \alpha\vdash\beta\mid  C & D\mid  \alpha\vdash \gamma\mid  E}$ 
			\item $\infer[(R7)]{B\mid  D\mid  \beta\sqcup \gamma\vdash\alpha\sqcup\alpha\mid  C\mid  E}{B\mid  \beta\vdash\alpha\mid  C & D\mid  \gamma\vdash \alpha\mid  E}$
		\end{multicols}
		%\noindent where B,C,D,E are possibly empty s-hypersequents.
	}
\end{proposition}
\begin{proof}
	$(R6)$ is derived using $(R1), (R1)^{\prime}$ and $(R4)$, while for $(R7)$ one uses $(R2), (R2)^{\prime}$ and $(R4)$.  
	%**end proof box??**$\hbox{}$
	%We will give proof for $(R7)$ using $(R1), (R1^{\prime})$ and $(R4)$. Let $\alpha,\beta,\gamma\in\mathfrak{F}$
	%\begin{prooftree}
	%	\AxiomC{$\alpha\vdash\beta$}
	%	\UnaryInfC{$\alpha\sqcap\alpha\vdash\beta\sqcap\alpha$}
	%	\AxiomC{$\alpha\vdash\gamma$}
	%	\UnaryInfC{$\beta\sqcap\alpha\vdash\beta\sqcap\gamma$}
	%	\BinaryInfC{$\alpha\sqcap\alpha\vdash\beta\sqcap\gamma$}
	%	\end{prooftree}
	%	 Proof of $(R8)$ is dual to the proof $(R7)$, using $(R2), (R2^{\prime})$ and $(R4)$
\end{proof}

\begin{theorem} 
	\label{thempdbl}
	{\rm 
		\label{other axiom of DBA}
		For $\alpha,\beta,\gamma\in\mathfrak{F}$, the following are provable in  \textbf{PDBL}. 
		
		$	\begin{array}{ll}

		1a~ (\alpha\sqcap \beta)\dashv\vdash(\beta\sqcap \alpha).&
		1b~ \alpha\sqcup \beta\dashv\vdash \beta\sqcup \alpha.\\
		
		2a~ \alpha\sqcap(\beta\sqcap \gamma)\dashv\vdash(\alpha\sqcap \beta)\sqcap \gamma.&
		2b~ \alpha\sqcup(\beta\sqcup \gamma)\dashv\vdash(\alpha\sqcup \beta)\sqcup \gamma.\\
		
		3a~ (\alpha\sqcap \alpha)\sqcap \beta\dashv\vdash (\alpha\sqcap \beta).&
		3b~ (\alpha\sqcup \alpha)\sqcup \beta\dashv\vdash \alpha\sqcup \beta.\\
		
		4a~ \neg \alpha\vdash\neg (\alpha\sqcap \alpha).&
		4b~ \lrcorner(\alpha\sqcup \alpha)\vdash\lrcorner \alpha.\\
		
		5a~ \alpha\sqcap(\alpha\sqcup \beta)\vdash (\alpha\sqcap \alpha).&
		5b~ \alpha\sqcup \alpha\vdash \alpha\sqcup (\alpha\sqcap \beta).\\
		
		6a~ \alpha\sqcap(\alpha\vee \beta)\vdash \alpha\sqcap \alpha.&
		6b~ \alpha\sqcup \alpha\vdash \alpha\sqcup (\alpha\wedge \beta).\\
		
		7a~ \bot \vdash \alpha\sqcap\neg \alpha.&
		7b~ \alpha\sqcup \lrcorner \alpha\vdash \top .\\
		
		8a~ \bot\vdash \neg \top. &
		8b~ \lrcorner \bot \vdash \top.\\
		
		9a~ \neg\neg(\alpha\sqcap\alpha)\dashv\vdash \alpha\sqcap\alpha. &
		9b~ \lrcorner\lrcorner(\alpha\sqcup\alpha)\dashv\vdash \alpha\sqcup\alpha.\\
		%9~ **(\alpha\sqcap \alpha)\sqcup (\alpha\sqcap \alpha)\vdash (\alpha\sqcup \alpha)\sqcap (\alpha\sqcup \alpha).
		10a~\neg\neg\alpha\dashv\vdash \alpha\sqcap\alpha. &
		10b~ \lrcorner\lrcorner\alpha\dashv\vdash \alpha\sqcup\alpha.\\
		11a~ \beta\vdash\alpha\vee\beta\mid\beta\sqcup\beta\vdash\beta. &
		11b~ \beta\vdash\beta\sqcap\beta\mid\alpha\wedge\beta\vdash\beta.
		\end{array}$
	}
\end{theorem}

\begin{proof}
	The proofs of 1-9 are obtained in a similar way as the proofs given in  Theorem 25  in \cite{howlader2021dbalogic} in case of   \textbf{CDBL}. Nevertheless, we include the proofs in the appendix to make the paper self-contained. 
	%The proofs for axioms 10a and 11a are given here.
	%The proofs of 10b and 11b are based on axioms and rules that are dual to those of 10a and 11a.
	\vskip 5pt
	\noindent  $10a$ is obtained using axiom 5a, Theorem \ref{thempdbl}(4a, 9a) and rules (R3), (R4). We get $10b$  by using the dual axiom and rules. The proof of $11a$ is given below. $11b$ is again obtained by using duals of axioms and rules that prove $11a$.
	\vskip 5pt
	\noindent Proof of  $11a$:

	 $\infer{\beta\vdash\alpha\vee\beta\mid\beta\sqcup\beta\vdash\beta~(R4)}{\infer{(Sp)~\beta\vdash\beta\sqcap\beta\mid\beta\sqcup\beta\vdash\beta~~~~\beta\sqcap\beta\vdash\alpha\vee\beta~(\mbox{using definition of} \vee \mbox{in previous step})}{\infer{\beta\sqcap\beta\vdash\neg(\neg\alpha\sqcap\neg\beta)~(R4)}{\infer{(\mbox{10a,~this~theorem)}~\beta\sqcap\beta\vdash\neg\neg\beta~~~~\neg\neg\beta\vdash\neg(\neg\alpha\sqcap\neg\beta)~~(R3~\mbox{on~previous~step})}{\infer{\neg\alpha\sqcap\neg\beta\vdash\neg\beta~(R4)}{\infer{(R1)~\neg\alpha\sqcap\neg\beta\vdash\neg\bot\sqcap\neg\beta~~~~\neg\bot\sqcap\neg\beta\vdash\neg\beta~(\mbox{axiom 3a})}{\infer{\neg\alpha\vdash\neg \bot~(R3)}{\infer{\bot\vdash\alpha~(\mbox{axiom 11a})}{}}}}}}}}$
	\end{proof}

\noindent

\begin{definition}
	\label{valution}
	{\rm 	 Let $\textbf{D}:=(D, \sqcap,\sqcup, \neg, \lrcorner, \top_{D}, \bot_{D})$ be a  pdBa. A {\it valuation} $v:\textbf{OV}\cup\textbf{PV}\cup \{\top, \bot\}\rightarrow D$ on $\textbf{D}$ is a map such that $v(p)\in D_{\sqcap}$, for all $p\in \textbf{OV}$,  $v(P)\in D_{\sqcup}$, for all $P\in \textbf{PV}$, $v(\top):=\top_{D}$ and  $v(\bot):=\bot_{D}$. $v$ is extended to the set $\mathfrak{F}$ of formulae by the following. }

	$\begin{array}{ll}
	
	1.~ v(\alpha\sqcup\beta):=v(\alpha)\sqcup v(\beta).&
	2.~ v(\alpha\sqcap\beta):=v(\alpha)\sqcap v(\beta).\\
	3.~ v(\neg \alpha):=\neg v(\alpha).&
	4.~ v(\lrcorner \alpha):=\lrcorner v(\alpha).
	%5.~ v(\top):=\top_{D}.&
	%6.~ v(\bot):=\bot_{D}.
	\end{array}$

\end{definition}
%\begin{definition}
%	\label{satis-sequent}
%	{\rm   A sequent $\alpha\vdash\beta$ is  {\it true} in  $\textbf{D}$ if and only if  for  all valuations  $v$ in $\textbf{D}$, $v$ satisfies $\alpha\vdash\beta$. A sequent $\alpha\vdash\beta$ is {\it valid} in the class of all  pdBas if and only if it is true in  every  pdBa.
%} \end{definition}

\begin{definition}
	\label{satis-hyper-sequent}
	{\rm  A sequent $\alpha\vdash\beta$ is said to be  {\it satisfied} by a valuation $v$ on   a pdBa $\textbf{D}$ if and only if  $v(\alpha)\sqsubseteq v(\beta)$. An s-hypersequent $B$ is said to be  {\it satisfied} by a valuation $v$ on   a pdBa $\textbf{D}$ if and only if  $v$ satisfies one of the components of the s-hypersequent $B$.
		An s-hypersequent $B$ is  {\it true} in  $\textbf{D}$ if and only if  for  all valuations  $v$ on $\textbf{D}$, $v$ satisfies the s-hypersequent $B$. An s-hypersequent $B$  is {\it valid} in the class of all  pdBas if and only if it is true in  every  pdBa.
} \end{definition}

\begin{theorem}[Soundness]
	\label{sound1}
	{\rm If an s-hypersequent $G$ is provable in  \textbf{PDBL}  then it is  valid in the class of all pdBas.
}	 \end{theorem}
\begin{proof}
	The proof that  all the axioms of  \textbf{PDBL} are valid in the class of all pdBas is straightforward and can be obtained using Proposition \ref{pro1.5} and Definition \ref{DBA}. As examples, we give proofs for 15a and 15b. Let $p\in\textbf{OV}$, $P\in\textbf{PV}$, $\textbf{D}$ be a pdBa and $v$ be a valuation on $\textbf{D}$. As $v(p)\in D_{\sqcap}$, $v(p)\sqcap v(p)=v(p)$. So $v$ satisfies $p\sqcap p\dashv\vdash p$. Similarly as $v(P)\in D_{\sqcup}$, $v(P)\sqcup v(P)=v(P)$, implying  that $v$  satisfies $P\sqcup P\dashv\vdash P$.

	\noindent	The validity of the inference rules  then needs to be verified. The case for external rules is straightforward.
	%One then  needs to verify that the rules of inference preserve validity. The case for the external rules is easy. 
	%We argue for the other rules.
	Using Proposition  \ref{pro1.5}, one can show that $(R1), (R2),(R1)^{\prime}$ and $(R2)^{\prime}$ preserve validity. The cases for $(R3)$ and $(R3)^{\prime}$  follow from Proposition \ref{pro2}. For $(R5)$, one uses Proposition \ref{order pure} and Definition \ref{DBA}.\\
	%We first argue for the rule $(Sp)$.	\\
	%To show $(R4)$ preserve validity, let the s-hypersequents $B\mid  \alpha\vdash\beta\mid  C$ and $D\mid  \beta\vdash\gamma\mid  E$ are valid in the class of all pdBa. Let $\textbf{D}$ be a pdBa and $v$ be a valuation on $\textbf{D}$. Then $v$ satisfies one components from each s-hypersequents. Now there are two possibilities that.
	%1. $v$ satisfies $\alpha\vdash\beta$ and $\beta\vdash\gamma$\\ 
	%	2. $v$ satisfies at least one component of the s-hypersequents $B,C,D,E$.\\
	%	Now if 1. is true then $v(\alpha)\sqsubseteq v(\beta)$ and $v(\beta)\sqsubseteq v(\gamma)$. So $v(\alpha)\sqsubseteq v(\gamma)$, as $\sqsubseteq$ is a transitive. 
	%	If 2 is true then observe that each component of the s-hypersequents $B,C,D,E$ is also a component of the s-hypersequent $B\mid  D\mid  \alpha\vdash\gamma\mid  C\mid  E~$. So $v$ satisfies a components of $B\mid  D\mid  \alpha\vdash\gamma\mid  C\mid  E~$, which implies that in both  the case $v$ satisfies the s-hypersequent  $B\mid  D\mid  \alpha\vdash\gamma\mid  C\mid  E~$. Hence $(R4)$ preserved validity.\\
	%To show $(Sp)$  preserves validity, 
	$(Sp)$: Let $\textbf{D}$ be a pdBa and $v$  a valuation on $\textbf{D}$. Let $\alpha\in \mathfrak{F}$. Then either $v(\alpha)=v(\alpha)\sqcap v(\alpha)=v(\alpha\sqcap\alpha)$ or $v(\alpha)=v(\alpha)\sqcup v(\alpha)=v(\alpha\sqcup\alpha)$. In the former case $v$ satisfies $\alpha\vdash\alpha\sqcap\alpha$ and in the latter, $v$ satisfies $\alpha\sqcup\alpha\vdash\alpha$.  	
	%So $(Sp)$ is valid.\\
	%To show $(EW)$  preserve validity, let $\textbf{D}$ be a pdBa and $v$ be a valuation on $\textbf{D}$. Let  $v$ satisfies the s-hypersequent $B$. Then $v$ satisfies the s-hypersequent $B|C$ as each components of $B$ is also  components of $B|C$. Hence $(EW)$ preserved validity. Similarly we can show that $(EC)$ and $(EP)$ also preserved validity.\\
	%Proof for other rule of inference is straight froward.
	%\noindent 	To show $(EW)$  preserve validity, let $\textbf{D}$ be a pdBa and $v$ be a valuation on $\textbf{D}$. Let  $v$ satisfies the s-hypersequent $B$. Then $v$ satisfies the s-hypersequent $B|C$ as each components of $B$ is also a component of $B|C$. Hence $(EW)$ preserved validity. 
	%\noindent Similarly we can show that $(EC)$ and $(EP)$ also preserved validity.
\end{proof}

The completeness theorem is proved using the Lindenbaum-Tarski algebra of \textbf{PDBL}, which is constructed in the usual fashion as follows. 
%The algebra is formed in the following manner. 
%As usual, the  completeness theorem is proved using the Lindenbaum-Tarski algebra of  \textbf{PDBL}, and
%the algebra is constructed in the standard way as follows.
A relation $\equiv_{\vdash}$ is defined  on $\mathfrak{F}$ by:  $\alpha\equiv_{\vdash}\beta$ if and only if $\alpha\dashv\vdash\beta$, for $\alpha,\beta\in \mathfrak{F}$. $\equiv_{\vdash}$ is    a congruence relation on $\mathfrak{F}$ with respect to  $\sqcup$, $\sqcap$, $\neg$, $\lrcorner$. The quotient set $\mathfrak{F}/\equiv_{\vdash}$  
%induced by the relation $\equiv_{\vdash}$ and 
with operations induced by the logical connectives, give the Lindenbaum-Tarski algebra $\mathcal{L}(\mathfrak{F}):= (\mathfrak{F}/\equiv_{\vdash},\sqcup,\sqcap,\neg,\lrcorner,[\top],[\bot])$. The axioms in {\bf PDBL} and Theorem \ref{other axiom of DBA} ensure that $\mathcal{L}(\mathfrak{F})$  is a dBa such that $[p]_{\sqcap}=[p]\sqcap [p]=[p]$ and $[P]_{\sqcup}=[P]\sqcup [P]=[P]$ for all $p\in \textbf{OV}$ and $P\in \textbf{PV}$, respectively.
One then obtains

\begin{proposition}
	\label{ldbm-dba}
	{\rm For any formulae $\alpha$ and $\beta$, the following are equivalent.
		\begin{enumerate}
			%	\item A s-hypersequent $G:\alpha_{1}\vdash\beta_{1}~|~\alpha_{2}\vdash\beta_{2}~|~\alpha_{3}\vdash\beta_{3}~|.........~|\alpha_{n}\vdash\beta_{n}$ is provable in  \textbf{PDBL}.
			\item  $[\alpha] \sqsubseteq [\beta]$ in $\mathcal{L}(\mathfrak{F})$.
			\item $\alpha\vdash \beta$ is provable in \textbf{PDBL}.
		\end{enumerate}
	}
\end{proposition}
\noindent The proof of   Proposition \ref{ldbm-dba} is obtained in a similar way as that of Proposition 27 in \cite{howlader2021dbalogic}. 
\vskip 3pt
 For $\mathcal{L}(\mathfrak{F})$, the corresponding standard context is  $\mathbb{K}(\mathcal{L}(\mathfrak{F})):=(\mathcal{F}_{p}(\mathcal{L}(\mathfrak{F})),\mathcal{I}_{p}(\mathcal{L}(\mathfrak{F})),\Delta)$. Let us note
\begin{lemma}
	\label{lemmavalu}
	\noindent {\rm\begin{enumerate}
			\item For  $p\in \textbf{OV}$, $F_{[p]}^{\prime}=I_{[p]}$ and $P\in\textbf{PV}$, $I_{[P]}^{\prime}=F_{[P]}$.
			\item For $\phi, \sigma\in \mathfrak{F}$, $F_{[\phi]\sqcap [\sigma]}^{\prime}=I_{[\phi]\sqcap [\sigma]}$, $I_{[\phi]\sqcup [\sigma]}^{\prime}=F_{[\phi]\sqcup [\sigma]}$, $F_{\neg[\phi]}^{\prime}=I_{\neg[\phi]}$ and $I_{\lrcorner[\phi]}^{\prime}=F_{\lrcorner[\phi]}$.
	\end{enumerate}  }
\end{lemma}
\begin{proof}
	1.	Let $p\in \textbf{OV}$. By Lemma \ref{derivation}, $F_{[p]}^{\prime}=I_{[p]_{\sqcap\sqcup}}=I_{[p]_{\sqcup}}$, as $[p]_{\sqcap}=[p]$ in $\mathcal{L}(\mathfrak{F})$. By Lemma \ref{complement of Fx}, $I_{[p]\sqcup [p]}=I_{[p]}$, implying $F_{[p]}^{\prime}= I_{[p]}$. Similarly, we can show that for $P\in\textbf{PV}$, $I_{[P]}^{\prime}=F_{[P]}$.
	\vskip4pt
	\noindent 2. In the following equations, we use Lemmas \ref{derivation} and  \ref{complement of Fx}.
	
	\noindent  $F^{\prime}_{[\phi]\sqcap [\sigma]}= I_{([\phi]\sqcap[\sigma])\sqcup ([\phi]\sqcap[\sigma])}=I_{([\phi]\sqcap[\sigma])}\cap I_{([\phi]\sqcap[\sigma])}=I_{([\phi]\sqcap[\sigma])}$.
	\vskip4pt
	\noindent  $I^{\prime}_{[\phi]\sqcup [\sigma]}= F_{([\phi]\sqcup[\sigma])\sqcap ([\phi]\sqcup[\sigma])}=F_{([\phi]\sqcup[\sigma])}\cap F_{([\phi]\sqcup[\sigma])}=F_{([\phi]\sqcup[\sigma])}$.
	\vskip4pt
	
	\noindent   $F^{\prime}_{\neg [\phi]}= I_{\neg[\phi]\sqcap \neg[\phi]}=I_{\neg [\phi]}$, as $\neg[\phi]\sqcap \neg[\phi]=\neg [\phi]$ and $I^{\prime}_{\lrcorner [\phi]}= F_{\lrcorner[\phi]\sqcup \lrcorner[\phi]}=F_{\lrcorner [\phi]}$, as $\lrcorner[\phi]\sqcup \lrcorner[\phi]=\lrcorner [\phi]$
\end{proof}
Using  Lemma \ref{lemmavalu}, we obtain a well-defined map:
\begin{definition}
	\label{v0}
	{\rm The map $v_{0}:\textbf{OV}\cup\textbf{PV}\cup \{\top, \bot\}\rightarrow \mathfrak{H}(\mathbb{K}(\mathcal{L}(\mathfrak{F})))$ is defined by $v_{0}(x):=(F_{ [x]},I_{[x]})$, for all $x\in \textbf{OV}\cup\textbf{PV}\cup \{\top,\bot\}$. }
\end{definition}

%\noindent  *From  Lemma \ref{lemmavalu} it follows that $v_{0}$ is a well-defined map.* 

\begin{lemma}
	\label{cnnvaluation}
	{\rm $v_{0}$ is a valuation on $\underline{\mathfrak{H}}(\mathbb{K}(\mathcal{L}(\mathfrak{F})))$. Moreover, for each formula $\alpha\in \mathfrak{F}$,  $v_{0}(\alpha)=(F_{[\alpha]}, I_{[\alpha]})$.}
\end{lemma}
\begin{proof} 
	Let $p\in \textbf{OV}$, $v_{0}(p)\sqcap v_{0}(p)=(F_{[p]}, I_{[p]})\sqcap (F_{[p]},I_{[p]})=(F_{[p]}\cap F_{[p]}, (F_{[p]}\cap F_{[p]})^{\prime})=(F_{[p]},F_{[p]}^{\prime})= (F_{[p]},I_{[p]})=v_{0}([p])$, by Lemma \ref{lemmavalu}(1). Similarly, for $P\in \textbf{PV}$,  $v_{0}(P)\sqcup v_{0}(P)= v_{0}(P)$. By Definition \ref{v0}, $v_{0}(\top)=(F_{[\top]}, I_{[\top]})$. So $v_{0}(\top)= (\mathcal{F}_{p}(\mathcal{L}(\mathfrak{F})),\emptyset)=\top_{\underline{\mathfrak{H}}(\mathbb{K}(\mathcal{L}(\mathfrak{F})))}$. Similarly,  we can show that $v_{0}(\bot)= \bot_{\underline{\mathfrak{H}}(\mathbb{K}(\mathcal{L}(\mathfrak{F})))}$. So $v_{0}$ is a valuation on $\underline{\mathfrak{H}}(\mathbb{K}(\mathcal{L}(\mathfrak{F})))$.

\noindent The second part of the lemma is proved by induction on the number of connectives in $\alpha$, and  using Definition \ref{valution} and Lemma \ref{lemmavalu}(2), one gets the result in each case.	
\end{proof}

\begin{proposition}
	\label{ldbm-dba1}
	{\rm The following are equivalent.
		\begin{enumerate}
			\item An s-hypersequent $G:\alpha_{1}\vdash\beta_{1}~|~\alpha_{2}\vdash\beta_{2}~|~\alpha_{3}\vdash\beta_{3}~|\ldots~|\alpha_{n}\vdash\beta_{n}$ is provable in  \textbf{PDBL}.
			\item $\alpha_{i}\vdash \beta_{i}$ is provable in \textbf{PDBL} for some $i\in\{1,2,\ldots,n\}$.
		\end{enumerate}
	}
\end{proposition} 
\begin{proof}
	$1 \Longrightarrow 2:$	Let $G:= \alpha_{1}\vdash\beta_{1}|\alpha_{2}\vdash\beta_{2}|\ldots|\alpha_{n}\vdash\beta_{n}$ be provable in $\textbf{PDBL}$. If possible, let us assume that  for all $i\in\{1,2,\ldots,n\}$, $\alpha_{i}\vdash \beta_{i}$ is not provable in $\textbf{PDBL}$.  By Proposition \ref{ldbm-dba},  for all $i\in\{1,2,\ldots,n\}$, $[\alpha_{i}]\not\sqsubseteq [\beta_{i}]$ in $\mathcal{L}(\mathfrak{F})$. By  Proposition \ref{pro1},  either $[\alpha_{i}]\sqcap[\alpha_{i}]\not\sqsubseteq_{\sqcap}[\beta_{i}]\sqcap[\beta_{i}]$  or $[\alpha_{i}]\sqcup[\alpha_{i}]\not\sqsubseteq_{\sqcap}[\beta_{i}]\sqcup[\beta_{i}]$  for all $i\in\{1,2,\ldots,n\}$.  Now, we consider the valuation $v_{0}$. By  Lemma \ref{cnnvaluation}, $v_{0}(\alpha_{i})=(F_{[\alpha_{i}]}, I_{[\alpha_{i}]})$ for all $i\in \{1,2,\ldots, n\}$.  Let $[\alpha_{i}]\sqcap[\alpha_{i}]\not\sqsubseteq_{\sqcap}[\beta_{i}]\sqcap[\beta_{i}]$ for some $i$. Then there exists a prime filter $F_{0i}$ in  $\mathcal{L}(\mathfrak{F})_{\sqcap}$ such that $[\alpha_{i}]\sqcap[\alpha_{i}]\in F_{0i}$ and $[\beta_{i}]\sqcap[\beta_{i}]\notin F_{0i}$. Therefore by Lemma \ref{lema1} there exists a  filter $F_{i}$ in $\mathcal{L}(\mathfrak{F})$ such that $F_{i}\cap\mathcal{L}(\mathfrak{F})_{\sqcap}=F_{0i}$. As $F_{0i}$ is a prime filter in  $\mathcal{L}(\mathfrak{F})_{\sqcap}$, $F_{i}\in\mathcal{F}_{p}(\mathcal{L}(\mathfrak{F}))$ such that $[\alpha_{i}]\sqcap[\alpha_{i}]\in F_{i}$ and $[\beta_{i}]\sqcap[\beta_{i}]\notin F_{i}$. So $[\alpha_{i}]\in F_{i}$, as $[\alpha_{i}]\sqcap[\alpha_{i}]\sqsubseteq[\alpha_{i}]$ and $[\beta_{i}]\notin F_{i}$ otherwise $[\beta_{i}]\sqcap[\beta_{i}]\in F_{i}$. Therefore  $F_{[\alpha_{i}]} \cancel{\subseteq} F_{[\beta_{i}]}$.

	If $[\alpha_{i}]\sqcup[\alpha_{i}]\not\sqsubseteq_{\sqcap}[\beta_{i}]\sqcup[\beta_{i}]$, then dually, we can show that there exists $I_{i}\in \mathcal{I}_{p}(\mathcal{L}(\mathfrak{F}))$ such that $[\alpha_{i}]\notin I_{i} $ and $[\beta_{i}]\in I_{i}$, which implies that $I_{[\beta_{i}]} \cancel{\subseteq} I_{[\alpha_{i}]}$.
	
	So in either case, $v_{0}(\alpha_{i})\cancel{\sqsubseteq} v_{0}(\beta_{i})$ for all $i\in \{1,2,3,\ldots, n\}$,
	which is not possible, as $G$ is valid in the class all pdBas by Theorem \ref{sound1}.
	So there exists $i\in \{1,2,3,\ldots, n\}$ such that $\alpha_{i}\vdash \beta_{i}$ is provable in $\textbf{PDBL}$.

	$2 \Longrightarrow 1:$ Let  $\alpha_{i}\vdash\beta_{i}$ be provable for some $i\in \{1, 2,\ldots, n\}$. 
	%In the following proofs, we 
	The s-hypersequent $G$ is then proved by repeatedly applying the external rules EW and EE. 
\end{proof}
\begin{lemma}
	{\rm $\mathcal{L}(\mathfrak{F})$ is a pdBa.}
\end{lemma}
\begin{proof}
	By Propositions \ref{ldbm-dba1}, \ref{ldbm-dba} and (Sp), it follows that $[\alpha]\sqsubseteq [\alpha\sqcap\alpha]$ or $[\alpha\sqcup\alpha]\sqsubseteq [\alpha]$ in $\mathcal{L}(\mathfrak{F})$. By Proposition \ref{pro1.5}(5), $[\alpha]\sqcap[\alpha]\sqsubseteq [\alpha]$ and $[\alpha]\sqsubseteq [\alpha]\sqcup[\alpha]$. So $[\alpha]=[\alpha]\sqcap [\alpha]=[\alpha\sqcup\alpha]$ or $[\alpha]=[\alpha]\sqcup [\alpha]=[\alpha\sqcup\alpha]$, which implies that $\mathcal{L}(\mathfrak{F})$ is a pdBa.
\end{proof}

\begin{theorem}[Completeness]
	\label{algcomdbl}
	{\rm If a hyper-sequent $G:\alpha_{1}\vdash\beta_{1}|\alpha_{2}\vdash\beta_{2}|\ldots|\alpha_{n}\vdash\beta_{n}$ is valid in the class of all pdBas then $G$ is provable in  \textbf{PDBL}.}
\end{theorem}

\begin{proof}
	Suppose if possible,  $G$ is not provable. Then $\alpha_{i}\vdash\beta_{i}$ are not provable for all $i\in\{1,2,\ldots, n\}$ by Proposition \ref{ldbm-dba1}. Rest of the proof is similar to the proof of $1 \Longrightarrow 2$ part of Proposition \ref{ldbm-dba1}.
\end{proof}

\subsection{MPDBL}
\label{logicmpdbl}
%\noindent As in the case of \textbf{PDBL} , all modal axioms are look alike to those in  \textbf{MCDBL}, \textbf{MCDBL4} \cite{howlader2021dbalogic} with the only difference being syntactical.
%The difference between the modal system corresponding to \textbf{PDBL} and the modal system corresponding to \textbf{CDBL} for the rule of inferences is the same as the difference between  the rules of \textbf{PDBL} and the rules of \textbf{CDBL}. 

The language $\mathfrak{L}_{1}$ of \textbf{MPDBL} adds two  unary modal connectives $\square$ and $\blacksquare$ to the language $\mathfrak{L}$ of \textbf{PDBL}. The formulae are given by the following scheme. 
\[\top\mid  \bot\mid  p\mid P\mid \alpha\sqcup\beta\mid   \alpha\sqcap\beta\mid  \neg\alpha\mid  \lrcorner\alpha\mid   \square\alpha\mid  \blacksquare\alpha,\] 
where $p\in \textbf{OV}$ and $P\in \textbf{PV}$. The set of formulae is denoted by $\mathfrak{F}_{1}$.
%**Change notation to F1 **
The axiom schema for \textbf{MPDBL} consists of all the axioms of  \textbf{PDBL} and the following.
\vskip 3pt
$\begin{array}{ll}
16a~\square \alpha\sqcap\square\beta\dashv\vdash\square(\alpha\sqcap\beta)& 
16b~ \blacksquare\alpha\sqcup\blacksquare\beta\dashv\vdash\blacksquare(\alpha\sqcup\beta)\\
17a~ \square(\neg\bot)\dashv\vdash\neg\bot&
17b~\blacksquare(\lrcorner\top)\dashv\vdash\lrcorner\top\\
18a~\square(\alpha\sqcap\alpha)\dashv\vdash\square(\alpha)&
18b~ \blacksquare(\alpha\sqcup\alpha)\dashv\vdash\blacksquare(\alpha)
\end{array}$
\vskip 3pt
\noindent {\it Rules of inference}: All the rules of \textbf{PDBL} and the following.

$\infer[(R8)]
{B\mid  \square\alpha\vdash\square\beta\mid  C}{B\mid  \alpha\vdash\beta \mid  C}$~~~
$\infer[(R9)]
{B\mid  \blacksquare\alpha\vdash\blacksquare\beta\mid  C}{B\mid  \alpha\vdash\beta\mid  C}$

\noindent Definable modal operators are $\lozenge, \blacklozenge$, given by $\lozenge\alpha:=\neg\square\neg\alpha$  and $\blacklozenge\alpha:=\lrcorner\blacksquare\lrcorner\alpha$.
%Provability of an s-hypersequent $G$ is defined in standard way as it is defined in case of \textbf{PDBL}. Then the following hold.
It is immediate that 
\begin{theorem}
	\label{provability}
	{\rm If an s-hypersequent $G:=\alpha_{1}\vdash\beta_{1}|\alpha_{2}\vdash\beta_{2}|\ldots|\alpha_{n}\vdash\beta_{n}$ is provable in \textbf{PDBL} then $G$ is also provable in \textbf{MPDBL}.}
\end{theorem}
\begin{proof}
As all the axioms of \textbf{PDBL} are also axioms of \textbf{MPDBL} and the rules of inference of \textbf{PDBL} are also rules of inference of \textbf{MPDBL}, a proof for the s-hypersequent $G$ in \textbf{PDBL} is also a proof for  $G$ in \textbf{MPDBL}.
	%Let $G$ be provable in \textbf{PDBL} and $G_{1}$, $G_{2}$, $\cdots$, $G_{n}$ be a proof of $G$ in \textbf{PDBL}. Then for each $i$, either $G_{i}$ is an axiom or $G_{i}$ is obtained by applying a rule of inference of \textbf{PDBL} to some subset of $\{G_{1},G_{2}, \cdots, G_{n}\}$. As all axioms and rules of inference of \textbf{PDBL}  are  axioms and rules of inference of \textbf{MPDBL},  $G_{1}$, $G_{2}$, $\cdots$, $G_{n}$ is also a proof of $G$ in \textbf{MPDBL}.
\end{proof}
%A valuation $v$ on a pdBao $\mathfrak{O}:=(D,\sqcup,\sqcap,\neg,\lrcorner,\top_{D},\bot_{D},\textbf{I},\textbf{C})$, is a map  from $\textbf{OV}\cup\textbf{PV}$ to $D$   that satisfies the conditions in Definition \ref{valution}  and the following  for the modal operators:
As we observed in the case of \textbf{PDBL}, the modal axioms defining \textbf{MPDBL} are also lookalikes of those defining \textbf{MCDBL} \cite{howlader2021dbalogic}.
\begin{definition}
	\label{algvaludbao}
	{\rm A valuation $v$ on a pdBao $\mathfrak{O}:=(D,\sqcup,\sqcap,\neg,\lrcorner,\top_{D},\bot_{D},\textbf{I},\textbf{C})$, is a map  from $\textbf{OV}\cup\textbf{PV}\cup \{\top, \bot\}$ to $D$   that satisfies the conditions in Definition \ref{valution}  and the following  for the modal operators:
		%Let $\mathfrak{O}:=(D,\sqcup,\sqcap,\neg,\lrcorner,\top_{D},\bot_{D},\textbf{I},\textbf{C})$ be a pdBao. A {\it valuation} $v$ in $\mathfrak{O}$ is a map from **\mathfrak{F}_{1} to $D$ satisfying 
		%the following.
		
		%		\hspace{2.5cm} $v(\top):=\top_{D}$ and $v(\bot):=\bot_{D}$.
		%		
		%		\hspace{1cm}$v(\alpha\sqcap\beta):=v(\alpha)\sqcap v(\beta)$ and $v(\alpha\sqcup\beta):=v(\alpha)\sqcup v(\beta)$.
		%		
		%		\hspace{2cm} $v(\neg\alpha):=\neg v(\alpha)$ and $v(\lrcorner\alpha):=\lrcorner v(\alpha)$.
		
		%\hspace{2cm} 
		$v(\square\alpha):=\textbf{I}(v(\alpha))$ and $v(\blacksquare\alpha):=\textbf{C}(v(\alpha))$.}
\end{definition}
\noindent The satisfaction, truth, and validity of  s-hypersequents are defined in the same way as before. 

\subsubsection{$MPDBL\Sigma$}
\label{mpdblsigma}
%$\textbf{MPDBL4}$ is obtained as a special case of the logic $\textbf{MPDBL}\Sigma$ that is defined as follows. 
%where $\Sigma$ is  any set of   sequents  in \textbf{MPDBL}. 
%and $V_{\Sigma}$ denote the class of those  pdBaos in which the sequents of $\Sigma$ are valid. 

We give a scheme of logics that can be obtained using sequents of \textbf{MPDBL}.
\begin{definition}
	{\rm Let $\Sigma$ be  any set of   sequents  in \textbf{MPDBL}. $\textbf{MPDBL}\Sigma$ is the logic obtained from \textbf{MPDBL} by adding all the sequents in $\Sigma$ as axioms.}
\end{definition}
\noindent %It is worth noting that 
If $\Sigma=\emptyset$, $\textbf{MPDBL}\Sigma$ is the same as $\textbf{MPDBL}$.
At the end of this section, %the reader will find 
we give the  set $\Sigma$ %needed to 
defining $\textbf{MPDBL4}$.
Let us go over some of the properties of $\textbf{MPDBL}\Sigma$ for any $\Sigma$, which would be applicable to both $\textbf{MPDBL}$ and $\textbf{MPDBL4}$. 
%Note that if $\Sigma=\emptyset$ then $\textbf{MPDBL}\Sigma$ is the same as $\textbf{MPDBL}$.  The  set $\Sigma$ required to define $\textbf{MPDBL4}$ will be given at the end of this section. Let us briefly discuss some features of $\textbf{MPDBL}\Sigma$ for any $\Sigma$ -- these would  then apply to both $\textbf{MPDBL}$ and $\textbf{MPDBL4}$.
\vskip 2pt
\noindent The class of pdBaos in which the sequents of $\Sigma$ are valid is denoted by $V_{\Sigma}$.
%As a result of axioms 15a, 16a, 17a, 15a, 16b 17b, and rules $(R9), (R10)$, one can deduce that if an s-hypersequents 
%Let $V_{\Sigma}$ denote the class of those  pdBaos in which the sequents of $\Sigma$ are valid. As a consequence of axioms 15a, 16a, 17a, 15a, 16b  17b and  rules  $(R9)$,  $(R10)$, one can conclude that  if an s-hypersequents
% $G:\alpha_{1}\vdash\beta_{1}\mid\alpha_{2}\vdash\beta_{2}\mid\ldots\mid\alpha_{n}\vdash\beta_{n}$ is provable in $\textbf{MPDBL}\Sigma$ then it is valid in the class $V_{\Sigma}$.
\begin{theorem}[Soundness]
	\label{soundmdl}
	{\rm If an s-hypersequent $G:\alpha_{1}\vdash\beta_{1}\mid\alpha_{2}\vdash\beta_{2}\mid\ldots\mid\alpha_{n}\vdash\beta_{n}$ is provable in $\textbf{MPDBL}\Sigma$ then it is valid in the class $V_{\Sigma}$.}
\end{theorem}
\begin{proof}
	To complete the proof it is sufficient to show that all axioms of $\textbf{MPDBL}\Sigma$  are valid in $V_{\Sigma}$ and rules of inference preserve validity. The propositional cases follow from Theorem \ref{sound1}. Here, we check  the validity of axioms 16a, 17a, 18a, 16b, 17b and 18b in $V_{\Sigma}$, and show that rules of inference $(R9)$ and $(R10)$ preserve  validity.   
	Let  $\mathfrak{O}$ be a pdBao belonging to the class $V_{\Sigma}$  and $v$ be a valuation on $\mathfrak{O}$. Let $\alpha,\beta\in \mathfrak{F}$. Then $v(\square(\alpha\sqcap\beta))=\textbf{I}(v(\alpha\sqcap\beta))=\textbf{I}(v(\alpha)\sqcap v(\beta))= \textbf{I}(v(\alpha))\sqcap\textbf{I}(v(\beta))=v(\square\alpha)\sqcap v(\square\beta)=v(\square\alpha\sqcap\square\beta)$. Now $v(\square(\alpha\sqcap\alpha))=\textbf{I}(v(\alpha\sqcap\alpha))=\textbf{I}(v(\alpha)\sqcap v(\alpha))=\textbf{I}(v(\alpha))=v(\square\alpha)$. Therefore the axioms $16a$ and $18a$ are true in $\mathfrak{O}$, which implies that $16a$ and $18a$ are  valid in the class $V_{\Sigma}$. Dually, we can show that the axioms $16b$ and $18b$ are valid in the class $V_{\Sigma}$. Now $v(\square(\neg \bot))=\textbf{I}(v(\neg\bot))=\textbf{I}(\neg v(\bot))=\textbf{I}(\neg\bot_{D})=\neg\bot_{D}$ and $v(\blacksquare(\lrcorner\top) )= \textbf{C}(v(\lrcorner\top))=\textbf{C}(\lrcorner v(\top))= \textbf{C}(\lrcorner\top_{D})=\lrcorner\top_{D}$. Therefore axioms $17a$ and $17b$ are true in $\mathfrak{O}$, which implies that $17a$ and $17b$ are valid in the class $V_{\Sigma}$.
	
	To show that $(R9)$ preserves validity, let $B~|~\alpha\vdash\beta~|~C$ be valid in $V_{\Sigma}$. Let $\mathfrak{O}\in V_{\Sigma}$ and $v$ be a valuation on $\mathfrak{O}$. Now $B~|~\alpha\vdash\beta~|~C$ is true in $\mathfrak{O}$, as $B~|~\alpha\vdash\beta~|~C$ is valid in $V_{\Sigma}$. So $v$ satisfies  the s-hypersequent $B~|~\alpha\vdash\beta~|~C$. Now there are two possibilities. (1) $v$ satisfies a component from B or C. (2) $v$ satisfies the component $\alpha\vdash\beta$. Now if (1) holds then $v$ also satisfies the s-hypersequent $B~|~\square\alpha\vdash\square\beta~|~C$, and if (2) holds then $v(\alpha)\sqsubseteq v(\beta)$ in $\mathfrak{O}$, which implies that $\textbf{I}(v(\alpha))\sqsubseteq \textbf{I}(v(\beta))$, by the monotonicity of $\textbf{I}$. So $v(\square\alpha)\sqsubseteq v(\square\beta)$, which implies that $v$ satisfies the s-hypersequent $B~|~\square\alpha\vdash\square\beta~|~C$. So (R9) preserves validity in $V_{\Sigma}$.
	
	\noindent Showing (R10) preserves validity is similar to the above.
\end{proof}

As before, the Lindenbaum-Tarski algebra $\mathcal{L}_{\Sigma}(\mathfrak{F}_{1})$ is obtained
for $\textbf{MPDBL}\Sigma$; the modal operators in  $\mathfrak{L}_{1}$ have induced new unary operators. 
%As before, one has the Lindenbaum-Tarski algebra $\mathcal{L}_{\Sigma}(\mathfrak{F}_{1})$   for $\textbf{MPDBL}\Sigma$; it has additional unary operators induced by the modal operators in $\mathfrak{L}_{1}$.
More precisely,  
$\mathcal{L}_{\Sigma}(\mathfrak{F}_{1}):=(\mathfrak{F}_{1}/\equiv_{\vdash}, \sqcup,\sqcap,\neg,\lrcorner,[\top],[\bot],f_{\square},f_{\blacksquare})$, where $f_{\square},f_{\blacksquare}$ are defined as: 
$f_{\square}([\alpha]):=[\square\alpha]$,
$f_{\blacksquare}([\alpha]):=[\blacksquare\alpha].$

\begin{note}
	\label{dual-ldbalg-ope}
	{\rm  $f_{\square}^{\delta}([\alpha])=\neg f_{\square}(\neg [\alpha])=\neg[\square\neg\alpha]=[\neg\square\neg\alpha]=[\lozenge\alpha]$ and, we denote $f_{\lozenge}([\alpha]):=f_{\square}^{\delta}([\alpha])$.
		
		\noindent Similar to the above $f_{\blacksquare}^{\delta}([\alpha])=[\blacklozenge\alpha]$ and, we denote $f_{\blacklozenge}([\alpha]):=f_{\blacksquare}^{\delta}([\alpha])$. }
\end{note}

\vskip 3pt 
\noindent Propositions \ref{ldbm-dba} and  \ref{ldbm-dba1} extend to this case also:
\begin{proposition}
	\label{ldbm-mdba}
	{\rm For any formulae $\alpha$ and $\beta$ the following are equivalent.
		\begin{enumerate}
			%	\item A s-hypersequent $G:\alpha_{1}\vdash\beta_{1}~|~\alpha_{2}\vdash\beta_{2}~|~\alpha_{3}\vdash\beta_{3}~|.........~|\alpha_{n}\vdash\beta_{n}$ is provable in  \textbf{PDBL}.
			\item  $[\alpha] \sqsubseteq [\beta]$ in $\mathcal{L}_{\Sigma}(\mathfrak{F}_{1})$.
			\item $\alpha\vdash \beta$ is provable in $\textbf{MPDBL}\Sigma$.
		\end{enumerate}
	}
\end{proposition}

\begin{proposition}
	\label{ldbm-mdba1}
	{\rm The following are equivalent.
		\begin{enumerate}
			\item An s-hypersequent $G:\alpha_{1}\vdash\beta_{1}~|~\alpha_{2}\vdash\beta_{2}~|~\alpha_{3}\vdash\beta_{3}~|\ldots~|\alpha_{n}\vdash\beta_{n}$ is provable in  $\textbf{MPDBL}\Sigma$.
			\item $\alpha_{i}\vdash \beta_{i}$ is provable in $\textbf{MPDBL}\Sigma$ for some $i\in\{1,2,\ldots,n\}$.
		\end{enumerate}
	}
\end{proposition} 
\noindent The operators $f_{\square}, f_{\blacksquare}$ are monotonic, according to  Proposition \ref{ldbm-mdba} and rules $(R9)$, $(R10)$: 

%Using this proposition and rules $(R9)$, $(R10)$,	one shows that the operators $f_{\square}, f_{\blacksquare}$ are monotonic:
%\end{definition}
%\begin{proposition}
%	\label{mldbm}
%	{\rm For any $\alpha,\beta \in \mathfrak{F}_{1}$, the following are equivalent.
%		\begin{enumerate}
%			\item $\alpha\vdash\beta$ is provable in $\textbf{MPDBL}\Sigma$.
%			\item  $[\alpha]\sqsubseteq [\beta]$ in $\mathcal{L}_{\Sigma}(\mathfrak{F}_{1})$.
%	\end{enumerate}}
%\end{proposition}
%\begin{proof}
%	The proof is similar to that of  Proposition \ref{ldbm}.
%\end{proof}
%Moreover, the operators $f_{\square}, f_{\blacksquare}$ are monotonic -- this is  proved using the proposition and rules $(R13)$, $(R14)$:
\begin{lemma}
	\label{monotonicity of operators}
	{\rm For $\alpha,\beta \in \mathfrak{F}_{1}$, $[\alpha]\sqsubseteq [\beta]$ in $\mathcal{L}_{\Sigma}( \mathfrak{F}_{1})$ implies that $f_{\square}([\alpha])\sqsubseteq f_{\square}([\beta])$ and $f_{\blacksquare}([\alpha])\sqsubseteq f_{\blacksquare}([\beta])$. }
\end{lemma}
\begin{proof}
	%	The proof is obtained by applying Proposition  \ref{mldbm} and rules **?** $(R9)$, $(R10)$.
	Let  $[\alpha]\sqsubseteq [\beta]$. $\alpha\vdash\beta$  by   Proposition  \ref{ldbm-mdba}. So $\square\alpha\vdash\square\beta$ and $\blacksquare\alpha\vdash\blacksquare\beta$  by $(R9)$ and $(R10)$. Again using  Proposition \ref{ldbm-mdba}, $[\square\alpha]\sqsubseteq [\square\beta]$ and $[\blacksquare\alpha]\sqsubseteq [\blacksquare\beta]$. Hence  $f_{\square}([\alpha])\sqsubseteq f_{\square}([\beta])$ and $f_{\blacksquare}([\alpha])\sqsubseteq f_{\blacksquare}([\beta])$. 
\end{proof}
\noindent $(\mathfrak{F}_{1}/\equiv_{\vdash}, \sqcup,\sqcap,\neg,\lrcorner,[\top],[\bot])$ is a pdBa;   Lemma \ref{monotonicity of operators} along with axioms  16a, 16b, 17a, 17b and Proposition \ref{ldbm-mdba1} gives
\begin{theorem}
	\label{lbam}
	{\rm   $\mathcal{L}_{\Sigma}(\mathfrak{F}_{1}) \in V_\Sigma$.} 
\end{theorem}
\begin{proof}
	By Propositions \ref{ldbm-mdba} and \ref{ldbm-mdba1}, $(\mathfrak{F}_{1}/\equiv_{\vdash}, \sqcup,\sqcap,\neg,\lrcorner,[\top],[\bot])$ is a pdBa.   Lemma \ref{monotonicity of operators} implies that the operators $f_{\square}, f_{\blacksquare}$ are monotonic. 
	Now let $[\alpha],[\beta]\in\mathfrak{F}_{1}/\equiv_{\vdash} $. Then  $f_{\square}([\alpha]\sqcap[\beta])=f_{\square}([\alpha\sqcap\beta])=[\square(\alpha\sqcap\beta)]=[\square\alpha\sqcap\square\beta]=[\square\alpha]\sqcap [\square\beta]$ by axiom 16a, which implies that $f_{\square}([\alpha]\sqcap[\beta])=[\square\alpha]\sqcap[\square\beta]=f_{\square}([\alpha])\sqcap f_{\square}([\beta])$. Dually, we can show that $f_{\blacksquare}([\alpha]\sqcup [\beta])= f_{\blacksquare}([\alpha])\sqcup f_{\blacksquare}([\beta])$. The others axioms can be  proved directly using  axioms  17a, 17b, 18a, and 18b, which implies that $\mathcal{L}_{\Sigma}(\mathfrak{F}_{1})$ is a pdBao. 
\end{proof}
\noindent Now, we define $v_{+}:\textbf{OV}\cup \textbf{PV}\cup \{\top, \bot\}\rightarrow \mathfrak{F}_{1}/\equiv_{\vdash}$, by $v_{+}(x):=[x]$ for all $x\in \textbf{OV}\cup \textbf{PV}\cup \{\top, \bot\}$.
% and extend $v_{1}$ to the set $\mathfrak{F}_{1}$ by the following. For $\alpha, \beta\in \mathfrak{F}_{1}$. $v_{+}(\alpha\sqcap\beta)=v_{+}(\alpha)\sqcap v_{+}(\beta)$, $v_{+}(\alpha\sqcup\beta)=v_{+}(\alpha)\sqcup v_{+}(\beta)$, $v_{+}(\neg\alpha)=\neg v_{+}(\alpha)$, $v_{+}(\lrcorner\alpha)=\lrcorner v_{+}(\alpha)$, $v_{+}(\square\alpha)=f_{\square}([\alpha])$, $v_{+}(\blacksquare\alpha)=f_{\blacksquare}([\alpha])$ $v_{+}(\top)=[\top]$ and $v_{+}(\bot)=[\bot]$. Then, we can show that $v_{+}$ is a valuation on $\mathcal{L}_{\Sigma}(\mathfrak{F}_{1})$. Moreover, we have the following.
\begin{lemma}
	{\rm $v_{+}$ is a valuation on $\mathcal{L}_{\Sigma}(\mathfrak{F}_{1})$. Moreover, for all $\alpha\in \mathfrak{F}_{1}$, $v_{+}(\alpha)=[\alpha]$.}
\end{lemma}
\begin{proof}
	For any $p\in \textbf{OV}$, $v_{+}(p):=[p]\in \mathcal{L}_{\Sigma}(\mathfrak{F}_{1})_{\sqcap}$, using axiom 15a. Similarly, using axiom 15b, for any $P\in \textbf{PV}$, $v_{+}(P):=[P]\in \mathcal{L}_{\Sigma}(\mathfrak{F}_{1})_{\sqcup}$. By definition of $v_{+}$, $v_{+}(\top)=[\top]=\top_{\mathcal{L}_{\Sigma}(\mathfrak{F}_{1})}$. Similarly, $v_{+}(\bot)=\bot_{\mathcal{L}_{\Sigma}(\mathfrak{F}_{1})}$. So $v_{+}$ is a valuation on $\mathcal{L}_{\Sigma}(\mathfrak{F}_{1})$.

	The second part of the lemma is proved by
	%To prove the rest of this lemma, we use
	 mathematical induction on the number of connectives in $\alpha$; using the definition of $v_{+}$ and the induction hypothesis, the cases are obtained.
\end{proof}
\begin{theorem}[Completeness]
	\label{algeb compe mdbl}
	{\rm If an s-hypersequent $G:\alpha_{1}\vdash\beta_{1}\mid\alpha_{2}\vdash\beta_{2}\mid\ldots\mid\alpha_{n}\vdash\beta_{n}$ is valid in the class $V_{\Sigma}$ then it is provable in $\textbf{MPDBL}\Sigma$.}
\end{theorem}
\begin{proof}
	Let $G:\alpha_{1}\vdash\beta_{1}\mid\alpha_{2}\vdash\beta_{2}\mid\ldots\mid\alpha_{n}\vdash\beta_{n}$ be not provable in  $\textbf{MPDBL}\Sigma$. By   Proposition \ref{ldbm-mdba1}, for all $i\in \{1,2,\ldots, n\}$, $\alpha_{i}\vdash\beta_{i}$ is not provable in  $\textbf{MPDBL}\Sigma$.  By Proposition \ref{ldbm-mdba},  for all $i\in \{1,2,\ldots, n\}$, $[\alpha_{i}]\cancel{\sqsubseteq}[\beta_{i}]$, which implies that for all $i\in \{1,2,\ldots, n\}$, $v_{+}(\alpha_{i})\cancel{\sqsubseteq} v_{+}(\beta_{i})$. So $G$ is not true in $\mathcal{L}_{\Sigma}(\mathfrak{F}_{1})$, which is not possible, as $G$ is valid in $V_{\Sigma}$. So $G$ is provable in $\textbf{MPDBL}\Sigma$.
\end{proof}

\noindent   \textbf{MPDBL4} is defined as the logic $\textbf{MPDBL}\Sigma$ where  $\Sigma$ contains the following:

\begin{center}
	$\begin{array}{l l}
	19a~ \square\alpha\vdash\alpha &
	19b~ \alpha\vdash\blacksquare\alpha\\
	20a~ \square\square\alpha\dashv\vdash\square\alpha & 
	20b~ \blacksquare\blacksquare\alpha\dashv\vdash\blacksquare\alpha
	\end{array}$
\end{center}
We have thus obtained  
\begin{theorem}[Soundness and Completeness]
	\label{sundandcompdbl}
	{\rm \noindent 
		\begin{enumerate}
			\item An s-hypersequent $G:\alpha_{1}\vdash\beta_{1}\mid\alpha_{2}\vdash\beta_{2}\mid\ldots\mid\alpha_{n}\vdash\beta_{n}$ is provable in \textbf{MPDBL} if and only if $G:\alpha_{1}\vdash\beta_{1}\mid\alpha_{2}\vdash\beta_{2}\mid\ldots\mid\alpha_{n}\vdash\beta_{n}$ is valid in the class of all  pdBaos.
			\item An s-hypersequent $G:\alpha_{1}\vdash\beta_{1}\mid\alpha_{2}\vdash\beta_{2}\mid\ldots\mid\alpha_{n}\vdash\beta_{n}$ is provable in \textbf{MPDBL4} if and only if $G:\alpha_{1}\vdash\beta_{1}\mid\alpha_{2}\vdash\beta_{2}\mid\ldots\mid\alpha_{n}\vdash\beta_{n}$ is valid in the class of all tpdBas.
	\end{enumerate}}
\end{theorem}

\section{Relational semantics  for the logics}
\label{semisemantics}
In this section, a two-sorted approach to give a relational semantics for the logics is proposed. Instead of taking the Kripke frame as a relational structure, we work with a context $\mathbb{K}:=(G, M, R)$ or a Kripke context $\mathbb{KC}$ based on $\mathbb{K}$. We consider the pdBa $\mathfrak{H}(\mathbb{K})$ and as defined earlier, a valuation on $\underline{\mathfrak{H}}(\mathbb{K})$ is a map $v: \textbf{OV}\cup\textbf{PV}\cup \{\top, \bot\}\rightarrow \mathfrak{H}(\mathbb{K})$ such that $v(p)\in \mathfrak{H}(\mathbb{K})_{\sqcap}$, $v(P)\in \mathfrak{H}(\mathbb{K})_{\sqcup}$, $v(\top)=(G, \emptyset)$ and $v(\bot)=(\emptyset, M)$.
%(also called a polarity in literature \cite{dunn, Gehrke}), where $G$ and $M$ are two non empty sets and $R\subseteq G\times M$. The elements of $G$ are called objects (worlds) and the elements of $M$ are called properties (co-worlds). 

%we introduced a two side approach to the relational semantics based on context,    Our method differs from them by choosing interpretant.  We interpreted each formula as semiconcepts instead of concepts.  Choosing semiconcepts as an interpretant allows us to talk about negations of concepts, which are two fundamental operations of Will's mathematical model of conceptual Knowledge \cite{}. We also consider world(object) and co-world(property) as a first class citizens, as it is done with state  in hybrid logic \cite{}. 
%For a context $\mathbb{K}$ the objects in $G$ are called world, while the properties in  $M$ are called co-world. satisfactions($\models$) is defined at world(objects) and the co-satisfaction($\succ$) is defined at co-world(properties). The satisfaction($\models$) encode the relation between objects and concepts(an object belongs to a concepts). The co-satisfaction($\succ$) encod the relation between property and concepts(). This two relation object-concepts and property-concepts are one of  the two core assumption of Wille's  conceptual knowledge \cite{}.
\subsection{  Relational semantics for PDBL}
\label{gsempdbl}

First, we define a model for  \textbf{PDBL}.
%	 Now, we define a model for \textbf{PDBL} based on a context $\mathbb{K}$ and its semiconcepts $\mathfrak{H}(\mathbb{K})$. Our model consists of a context and a map, which assigns each object variable to a semiconcept generated by objects and assign each property variable to a semiconcept generated by properties. A propositional variable is assigned to a semiconcept. 
Interpretation is defined in the same way as in  \cite{Gehrke, CWFSPAPMTAWN}. Instead of saying ``a part of'' relation ($\succ$), we refer to it as ``co-satisfaction" as in \cite{CWFSPAPMTAWN}.
\begin{definition}
	
	\label{pdBamodel}
	{\rm A model for \textbf{PDBL} is a pair $\mathbb{M}:=(\mathbb{K},v)$, where $\mathbb{K}=(G,M,R)$ is a context and   $v$ is a valuation on $\underline{\mathfrak{H}}(\mathbb{K})$.
		%$:\textbf{OV}\cup \textbf{PV}\rightarrow\mathfrak{H}(\mathbb{K})$ such that $v(p)\in \mathfrak{H}(\mathbb{K})_{\sqcap}$ and $v(P)\in \mathfrak{H}(\mathbb{K})_{\sqcup}$, is called valuation on $\mathfrak{H}(\mathbb{K})$.
		%\begin{enumerate}
		%	\item $v(\alpha\sqcap\beta)=v(\alpha)\sqcap v(\beta)$ and $v(\alpha\sqcup\beta)=v(\alpha)\sqcup v(\beta)$.
		%	\item $v(\neg \alpha)=\neg v(\alpha)$ and $v(\lrcorner\alpha)=\lrcorner v(\alpha)$.
		%	\item $v(\top)=(G, \emptyset)$ and $v(\bot)=(\emptyset, M)$.
		%\item $v(\square\alpha)=f_{R}(v(\alpha))$ and $v(\textbf{C}\alpha)=f_{S}(v(\alpha))$.
		%\item  $v(\lozenge\alpha)=f_{R}^{\delta}(v(\alpha))$ and $v(\blacklozenge\alpha)=f_{S}^{\delta}(v(\alpha))$.
		%\end{enumerate}
	}
\end{definition}

\noindent Given a model $\mathbb{M}:=(\mathbb{K},v)$, the relations $\models$ of satisfaction  and $\succ$ of co-satisfaction of formulae with respect to $\mathbb{M}$ are defined recursively as follows. For any $g\in G$ and formula $\alpha$, $\mathbb{M}, g\models \alpha$ denotes $\alpha$ is satisfied at $g$ in $\mathbb{M}$, while for any $m\in M$, $\mathbb{M}, m\succ \alpha$ denotes $\alpha$ is co-satisfied at $m$ in $\mathbb{M}$.

%\noindent Now we define inductively the notion of a formula $\alpha$ being satisfied at a world $g\in G$ in a model $\mathbb{M}$ and we denote it by $\mathbb{M},g\models\alpha$. We also define the notion of a formula $\alpha$ being co-satisfied at a co-world $m\in M$ in a model $\mathbb{M}$ inductively and denote it by   $\mathbb{M},m\succ \alpha$.

\begin{definition}[Satisfaction and co-satisfaction]
	\label{satisfiction and co-satisfiction}
	{\rm For each $g\in G$ and for each $m\in M$,
		\begin{enumerate}
			%\item $\mathbb{M},g\models x ~\mbox{if and only if}~ g\in ext(v(x))$,  where $x\in\textbf{PV}.$
			\item $\mathbb{M},g\models p ~\mbox{if and only if}~ g\in ext(v(p))$,  where $p\in\textbf{OV}.$
			\item $\mathbb{M},g\models P ~\mbox{if and only if}~ g\in ext(v(P))$,  where $P\in \textbf{PV}.$
			\item $\mathbb{M},g\models \top$ for all $g\in G$ and $\mathbb{M},m\not\succ \top$ for all $m\in M$.
			\item $\mathbb{M},g\not\models\bot$ for all $g\in G$ and $\mathbb{M},m \succ \bot$ for all $m\in M$.
			\item $\mathbb{M},g\models \alpha\sqcap \beta$ if and only if $\mathbb{M},g\models \alpha$ and $\mathbb{M},g\models\beta$
			
			\item $\mathbb{M},g\models \neg \alpha$ if and only if $\mathbb{M},g\not\models\alpha $.
			% \item $\mathbb{M},m\succ x~\mbox{if and only if}~ m\in int(v(x))$, where $x\in\textbf{PV}.$
			\item $\mathbb{M},m\succ p~\mbox{if and only if}~ m\in int(v(p))$, where $p\in\textbf{OV}.$
			\item $\mathbb{M},m\succ P~\mbox{if and only if}~ m\in int(v(P))$, where $P \in \textbf{PV}.$
			\item $\mathbb{M},m\succ\alpha\sqcup \beta$ if and only if $\mathbb{M},m\succ\alpha$ and $\mathbb{M},m\succ\beta$.
			\item $\mathbb{M},m\succ\lrcorner\alpha$ if and only if $\mathbb{M},m\not\succ\alpha$.
			
			\item $ \mathbb{M},g\models \alpha \sqcup\beta $ if and only if for all $m\in M(\mathbb{M},m\succ \alpha\sqcup\beta\Longrightarrow gRm)$.
			
			\item $\mathbb{M},g\models \lrcorner \alpha$ if and only if  for all $m\in M(\mathbb{M},m\not\succ \alpha\Longrightarrow gRm)$.
			\item $\mathbb{M},m\succ\neg\alpha$ if and only if for all $g\in G, (\mathbb{M},g\not\models\alpha\Longrightarrow gRm )$.
			\item $\mathbb{M},m\succ \alpha\sqcap\beta$ if and only if for all $g\in G,  (\mathbb{M},g\models\alpha\sqcap\beta\Longrightarrow gRm) $.
	\end{enumerate}}
	
\end{definition}
%$\mathbb{M}=(\mathbb{K},v)$ is a model for {\rm \textbf{PDBL}} if $v$ is a map from propositional variable to $\mathfrak{S}(\mathbb{K})$. 

\begin{definition} 
	\label{set-of-world-co}
	{\rm For any model $\mathbb{M}:=(\mathbb{K}, v)$ and  $\alpha\in \mathfrak{F}$,  the maps $v_{1}: \mathfrak{F}\rightarrow \mathcal{P}(G)$ and $v_{2}: \mathfrak{F}\rightarrow \mathcal{P}(M)$ are defined as   $v_{1}(\alpha):=\{g\in G~|~ \mathbb{M},g\models \alpha\},$ and  $v_{2}(\alpha):=\{m\in M~ |~ \mathbb{M},m\succ \alpha\}$.  }
\end{definition}

\begin{proposition}
	\label{satisfiction-set}
	{\rm Let $\mathbb{M}:=(\mathbb{K}, v)$ be a model and $\alpha, \beta\in \mathfrak{F}$. %the following hold.
		\begin{enumerate}
			\item $v_{1}(\alpha\sqcap\beta)=v_{1}(\alpha)\cap v_{1}(\beta)$ and $v_{2}(\alpha\sqcap\beta)=(v_{1}(\alpha)\cap v_{1}(\beta))^{\prime}$.
			\item $v_{1}(\alpha\sqcup\beta)=(v_{2}(\alpha)\cap v_{2}(\beta))^{\prime}$ and $v_{2}(\alpha\sqcup\beta)=v_{2}(\alpha)\cap v_{2}(\beta)$.
			\item $v_{1}(\neg\alpha)= v_{1}(\alpha)^{c}$ and  $v_{2}(\neg\alpha)= (v_{1}(\alpha)^{c})^{\prime}$.
			\item $v_{1}(\lrcorner \alpha)= (v_{2}(\alpha)^{c})^{\prime}$ and $v_{2}(\lrcorner \alpha)= (v_{2}(\alpha)^{c}$.
			\item $v_{1}(\top)=G$ and $v_{2}(\top)=\emptyset$.
			\item $v_{1}(\bot)=\emptyset$ and $ v_{2}(\bot)=M$.
	\end{enumerate}}
\end{proposition}
\begin{proof}
We given the proofs of 1 and 3. 2 and 4 follow dually. 5 and 6 follow directly from Definition \ref{satisfiction and co-satisfiction}.
\begin{eqnarray*}
		1.~v_{1}(\alpha\sqcap\beta)
		&=&\{g\in G~|~\mathbb{M},g\models \alpha\sqcap\beta\}
		=\{g\in G~|~\mathbb{M},g\models \alpha~\mbox{and}~\mathbb{M},g\models \beta\}\\
		&=&\{g\in G~|~ g\in v_{1}(\alpha)~\mbox{and}~g\in v_{1}(\beta)\}
		= v_{1}(\alpha)\cap v_{1}(\beta).
	\end{eqnarray*}
	\begin{eqnarray*}	~~~v_{2}(\alpha\sqcap\beta)&=&\{m\in M~|~\mathbb{M},m\succ \alpha\sqcap\beta\}\\
		&=&\{m\in M~|~\mbox{for all}~ g\in G( \mathbb{M},g\models \alpha\sqcap\beta \Longrightarrow gRm)\}\\
		&=&\{m\in M~|~\mbox{for all}~ g\in G( \mathbb{M},g\models \alpha~\mbox{and}~\mathbb{M},g\models\beta \Longrightarrow gRm)\}\\
		&=&\{m\in M~|~\mbox{for all}~ g\in G( g\in v_{1}(\alpha)~\mbox{and}~g\in v_{1}(\beta) \Longrightarrow gRm)\}\\
		&=&\{m\in M~|~\mbox{for all}~ g\in G( g\in v_{1}(\alpha)\cap v_{1}(\beta) \Longrightarrow gRm)\}\\
		&=&  (v_{1}(\alpha)\cap v_{1}(\beta))^{\prime}.
	\end{eqnarray*}
	%\noindent 2. 
	%\begin{eqnarray*}
	%	v_{1}(\alpha\sqcup\beta)
	%	&=&\{g\in G~|~\mathbb{M},g\models \alpha\sqcup\beta\}\\
	%	&=&\{g\in G~|~\mbox{for all}~ m\in M( \mathbb{M},m\succ \alpha\sqcup\beta \Longrightarrow gRm)\}\\
	%	&=&\{g\in G~|~\mbox{for all}~ m\in M( \mathbb{M},m\succ \alpha~\mbox{and}~\mathbb{M},m\succ\beta \Longrightarrow gRm)\}\\
	%	&=&\{g\in G~|~\mbox{for all}~ m\in M( m\in v_{2}(\alpha)~\mbox{and}~m\in v_{2}(\beta) \Longrightarrow gRm)\}\\
	%	&=&\{g\in G~|~\mbox{for all}~ m\in M( m\in v_{2}(\alpha)\cap v_{2}(\beta) \Longrightarrow gRm)\}\\
	%	&=&  (v_{2}(\alpha)\cap v_{2}(\beta))^{\prime}
	%\end{eqnarray*}
	%\begin{eqnarray*}		**?**v_{2}(\alpha\sqcup\beta)
	%	&=&\{m\in M~|~\mathbb{M},m\succ \alpha\sqcup\beta\}\\
	%	&=&\{m\in M~|~\mathbb{M},m\succ \alpha~\mbox{and}~\mathbb{M},m\succ \beta\}\\
	%	&=&\{m\in M~|~ m\in v_{2}(\alpha)~\mbox{and}~m\in v_{2}(\beta)\}\\
	%	&=&  v_{2}(\alpha)\cap v_{2}(\beta).
	%\end{eqnarray*}
	%\noindent 3. 
	\begin{eqnarray*}
		3.~v_{1}(\neg\alpha)
		&=&\{g\in G~|~\mathbb{M},g\models \neg\alpha\}
		=\{g\in G~|~\mathbb{M},g\not\models \alpha\}\\
		&=&\{g\in G~|~ g\notin v_{1}(\alpha)\}
		=(v_{1}(\alpha))^{c}.
	\end{eqnarray*}
	\begin{eqnarray*}
		~~~v_{2}(\neg\alpha)
		&=&\{m\in M~|~\mathbb{M},m\succ\neg \alpha\}\\
		&=&\{m\in M~|~\mbox{for all}~ g\in G( \mathbb{M},g\not\models \alpha \Longrightarrow gRm)\}\\
		&=&\{m\in M~|~\mbox{for all}~ g\in G( g\notin v_{1}(\alpha) \Longrightarrow gRm)\}\\
		&=&\{m\in M~|~\mbox{for all}~ g\in G( g\in (v_{1}(\alpha))^{c} \Longrightarrow gRm)\}\\
		&=&(v_{1}(\alpha))^{c\prime}.
	\end{eqnarray*}
	%\begin{eqnarray*}
	%	v_{1}(\lrcorner\alpha)
	%	&=&\{g\in G~|~\mathbb{M},g\models\lrcorner \alpha\}\\
	%	&=&\{g\in G~|~\mbox{for all}~ m\in M( \mathbb{M},m\not\succ\alpha \Longrightarrow gRm)\}\\
	%	&=&\{g\in G~|~\mbox{for all}~ m\in M( m\notin v_{2}(\alpha) \Longrightarrow gRm)\}\\
	%	&=&\{g\in G~|~\mbox{for all}~ m\in M( m\in (v_{2}(\alpha))^{c} \Longrightarrow gRm)\}\\
	%	&=& (v_{2}(\alpha))^{c\prime}.
	%\end{eqnarray*}
	%\begin{eqnarray*}	v_{2}(\lrcorner\alpha)
	%	&=&\{m\in M~|~\mathbb{M},m\succ \lrcorner\alpha\}
	%	=\{m\in M~|~\mathbb{M},m\not\succ \alpha\}\\
	%	&=&\{m\in M~|~ m\notin v_{2}(\alpha)\}
	%	=(v_{2}(\alpha))^{c}.
	%\end{eqnarray*}
	%\noindent	5 and 6 directly follow from Definition \ref{satisfiction and co-satisfiction}.
\end{proof}

\begin{corollary}
	\label{evproto} {\rm For any formula $\alpha\in \mathfrak{F}$, $(v_{1}(\alpha),v_{2}(\alpha))\in \mathfrak{H}(\mathbb{K})$.}
\end{corollary}
\begin{proof}
	%We will prove this
	The  result is  obtained by using mathematical induction on the number of connectives that occur  in
	%in the formula 
	$\alpha$.  Let $\mathbb{M}:=(\mathbb{K},v)$ be a model, where $v:\textbf{OV}\cup\textbf{PV}\cup \{\top, \bot\}\rightarrow\mathfrak{H}(\mathbb{K})$ is the valuation.  
	
	\noindent For the base case let the number of connectives in $\alpha$, $n=0$. Then either  $\alpha$ belongs to $\textbf{OV}\cup\textbf{PV}$ or $\alpha$  is a propositional constant.
	
	\noindent Let $\alpha=p$. Then $v_{1}(p)=\{g\in G~|~ \mathbb{M},g\models p\}=ext(v(p))$ and $ v_{2}(p)=\{m\in M~ |~ \mathbb{M},m\succ p\}= int(v(p))$ by the definition of satisfaction and co-satisfaction. So $(v_{1}(p),v_{2}(p))=(ext(v(p)), int(v(p)))=v(p)\in \mathfrak{H}(\mathbb{K})$.
	
	\noindent If $\alpha=P$ then similar to the  above, $(v_{1}(\alpha),v_{2}(\alpha))=(ext(v(P)), int(v(P)))=v(P)\in \mathfrak{H}(\mathbb{K})$. 
	
	%\noindent Now if $\alpha=x$ then similar to the above $v(x)\in \mathfrak{H}(\mathbb{K})$. 
	
	\noindent For $\alpha=\top$, $v_{1}(\top)=G$ and $v_{2}(\top)=\emptyset$. So $(v_{1}(\top),v_{2}(\top))=(G,\emptyset)\in \mathfrak{H}(\mathbb{K})$, and if $\alpha=\bot$ then $v_{1}(\bot)=\emptyset$ and $v_{2}(\bot)=M$. So $(v_{1}(\bot),v_{2}(\bot))=(\emptyset, M)\in \mathfrak{H}(\mathbb{K})$. 
	
	\noindent  Now, we assume that the result is true for all formulae  with number of connectives less or equal to $n$. Let $\alpha$ be a formula with $n+1$ connectives. Then  $\alpha$ is of one of the forms $\beta\sqcap\delta$, $\beta\sqcup\delta$, $\neg\beta$, or $\lrcorner\beta$, where $\beta, \delta$ are formulae with number of connectives less or equal to $n$. By induction hypothesis $(v_{1}(\beta),v_{2}(\beta))$ and $(v_{1}(\delta),v_{2}(\delta))$ belong to $\mathfrak{H}(\mathbb{K})$.
	Therefore $(v_{1}(\beta),v_{2}(\beta))\sqcap (v_{1}(\delta),v_{2}(\delta))$, %=(v_{1}(\beta)\cap v_{1}(\delta),(v_{1}(\beta)\cap v_{1}(\delta))^{\prime})$, 
	$(v_{1}(\beta),v_{2}(\beta))\sqcup (v_{1}(\delta),v_{2}(\delta))$, %=((v_{2}(\beta)\cap v_{2}(\delta))^{\prime},v_{2}(\beta)\cap v_{2}(\delta))$,
	$\neg(v_{1}(\beta),v_{2}(\beta))$ and $\lrcorner (v_{1}(\beta),v_{2}(\beta))$
	all belong to $\mathfrak{H}(\mathbb{K})$, as $\mathfrak{H}(\mathbb{K})$ is closed under $\sqcap,\sqcup, \neg, \lrcorner$.
	
	\noindent  By Proposition \ref{satisfiction-set} it follows that
	
	\noindent  $(v_{1}(\beta\sqcap\delta),v_{2}(\beta\sqcap\delta))
	=(v_{1}(\beta)\cap v_{1}(\delta),(v_{1}(\beta)\cap v_{1}(\delta))^{\prime})=(v_{1}(\beta),v_{2}(\beta))\sqcap (v_{1}(\delta),v_{2}(\delta))$, 
	
	\noindent $(v_{1}(\beta\sqcup\delta),v_{2}(\beta\sqcup\delta))=((v_{2}(\beta)\cap v_{2}(\delta))^{\prime},v_{2}(\beta)\cap v_{2}(\delta))=(v_{1}(\beta),v_{2}(\beta))\sqcup (v_{1}(\delta),v_{2}(\delta))$,
	
	\noindent $(v_{1}(\neg\beta),v_{2}(\neg\beta))=(v_{1}(\beta)^{c},(v_{1}(\beta))^{c\prime})=\neg (v_{1}(\beta),v_{2}(\beta))$,
	
	\noindent  $(v_{1}(\lrcorner \beta),v_{2}(\lrcorner \beta))=((v_{2}(\beta))^{c\prime},v_{2}(\beta)^{c})=\lrcorner (v_{1}(\beta),v_{2}(\beta))$. 
	
\noindent So	$(v_{1}(\alpha), v_{2}(\alpha))\in \mathfrak{H}(\mathbb{K})$.
	% Since number of connectives in $\alpha$ and $\beta$ are less or equal to $n$,   $(v_{1}(\beta),v_{2}(\beta))$ and $(v_{1}(\delta),v_{2}(\delta))$ belongs to $\mathfrak{H}(\mathbb{K})$, by the induction hypotheses.  Now, we  prove case by case. \\
	%By Corollary \ref{PDBL-valution} it follows that $(v_{1}(\beta\sqcap\delta),v_{2}(\beta\sqcap\delta))
	%=(v_{1}(\beta),v_{2}(\beta))\sqcap (v_{1}(\delta),v_{2}(\delta))
	%\in \mathfrak{H}(\mathbb{K})$, $(v_{1}(\beta\sqcup\delta),v_{2}(\beta\sqcup\delta))=(v_{1}(\beta),v_{2}(\beta))\sqcup (v_{1}(\delta),v_{2}(\delta))\in \mathfrak{H}(\mathbb{K})$,
	% $(v_{1}(\neg\beta),v_{2}(\neg\beta))=\neg (v_{1}(\beta),v_{2}(\beta))\in \mathfrak{H}(\mathbb{K})$, and  $(v_{1}(\lrcorner \beta),v_{2}(\lrcorner \beta))=\lrcorner (v_{1}(\beta),v_{2}(\beta))\in \mathfrak{H}(\mathbb{K})$. %By mathematical induction  $(v_{1}(\alpha),v_{2}(\alpha))\in\mathfrak{H}(\mathbb{K})$.
	%$(b.)$ In case $\mathbb{M}=(\mathbb{K},v)$ is a model for \textbf{PDBL}, $v:\textbf{PV}\rightarrow\mathfrak{H}(\mathbb{K})$ and $\mathfrak{H}(\mathbb{K})$ is a subalgebra of $\mathfrak{R}(\mathbb{K})$ is closed under the above operations $\sqcap,\sqcup,\top,\bot$ and contains $\top,\bot$. So the same proof as in $(a)$ can be carried over to $\mathfrak{H}(\mathbb{K})$.
\end{proof}

\begin{corollary}
	\label{PDBL-valution}
	{\rm %Let $\mathbb{M}=(\mathbb{K}, v)$ be a model and $\alpha, \beta\in \mathfrak{F}$. Then the following hold. 
		For any formulae $\alpha, \beta\in \mathfrak{F}$,
		
		\noindent	\begin{enumerate}
			\item $(v_{1}(\alpha\sqcap\beta), v_{2}(\alpha\sqcap\beta))=(v_{1}(\alpha), v_{2}(\alpha))\sqcap (v_{1}(\beta),v_{2}(\beta))$.
			\item $(v_{1}(\alpha\sqcup\beta), v_{2}(\alpha\sqcup\beta))=(v_{1}(\alpha), v_{2}(\alpha))\sqcup (v_{1}(\beta),v_{2}(\beta))$.
			\item $\neg (v_{1}(\alpha), v_{2}(\alpha))=(v_{1}(\neg\alpha), v_{2}(\neg\alpha))$ and $\lrcorner (v_{1}(\alpha), v_{2}(\alpha))=(v_{1}(\lrcorner\alpha), v_{2}(\lrcorner\alpha))$.
			%\item $(v_{1}(\top), v_{2}(\top))=(G, \emptyset)$ and $(v_{1}(\bot), v_{2}(\bot))=(\emptyset, M)$.
			%\item $v(\square\alpha)=f_{R}(v(\alpha))$ and $v(\textbf{C}\alpha)=f_{S}(v(\alpha))$.
			%\item  $v(\lozenge\alpha)=f_{R}^{\delta}(v(\alpha))$ and $v(\blacklozenge\alpha)=f_{S}^{\delta}(v(\alpha))$.
	\end{enumerate}}
\end{corollary}
\begin{proof}
	The proof follows from Proposition \ref{satisfiction-set}, Corollary \ref{evproto} and the definitions of the operations $\sqcap, \sqcup, \neg, \lrcorner$ on $\mathfrak{H}(\mathbb{K})$ as given  in Section \ref{intro}.
	\end{proof}
%\noindent For a model $\mathbb{M}=(\mathbb{K}, v)$, we extend the valuation $v$  to the set of formula $\mathfrak{F}$ by  $v(\alpha)=(v_{1}(\alpha),v_{2}(\alpha))$ for all $\alpha\in \mathfrak{F}$ and,  retain the notation $v$ for the extension. 

\begin{corollary}
	\label{PDBL-valution1}
	{\rm   For any model $\mathbb{M}:=(\mathbb{K}, v)$ and  $\alpha\in \mathfrak{F}$, $v(\alpha)=(v_{1}(\alpha), v_{2}(\alpha))$. }
\end{corollary}
\begin{proof}
	Follows by induction on the number of connectives that occur in a formula, using  definition of extension of a valuation and Corollary \ref{PDBL-valution}. Note that $v(p)=(ext(v(p)), int(v(p)))=(v_{1}(p), v_{2}(p))$ and similarly, $v(P)=(v_{1}(P), v_{2}(P))$. $v(\top)=(G, \emptyset)=(v_{1}(\top),  v_{2}(\top))$ by Proposition \ref{satisfiction-set}. Similarly, $v(\bot)=(v_{1}(\bot), v_{2}(\bot))$.
\end{proof}
% \begin{proof}
%	Let $p\in \textbf{OV}$ and $P\in \textbf{PV}$. Then $v(p)=(v_{1}(p), v_{2}(p))=(\{g\}, \{g\}^{\prime})$ for some $g\in G$ and	$v(P)=(v_{1}(P), v_{2}(P))=(\{m\}^{\prime}, \{m\})$ for some $m\in M$. Rest of the proof follows from Proposition \ref{satisfiction-set}.
% \end{proof}

%\begin{proof}
%roof is similar to the proof of Lemma \ref{evproto}.
%\end{proof}
% \begin{lemma}
%\label{to prove sound}
%{\rm  Let $\mathbb{M}$ be a model for {\rm \textbf{PDBL}} and $\alpha$ be a  formula . Let $g\in G$ and $m\in M$. Then $\mathbb{M},g\models \alpha$ and $\mathbb{M},m\succ \alpha$ implies that $gRm$.}
%\end{lemma}
%\begin{proof}
%Let $\mathbb{M}$ be a model and $g\in G$ and $m\in M$ such that $\mathbb{M},g\models \alpha$ and $\mathbb{M},m\succ \alpha$. So $g\in v_{1}(\alpha)$ and $m\in v_{2}(\alpha)$  . By Lemma \ref{evproto}, $( v_{1}(\alpha),v_{2}(\alpha))\in \mathfrak{H}(\mathbb{K})$, which implies that $( v_{1}(\alpha))^{\prime\prime}=(v_{2}(\alpha))^{\prime}$, as semiconcepts are also protoconcepts. So $g\in ( v_{1}(\alpha))^{\prime\prime}=(v_{2}(\alpha))^{\prime}$, as $v_{1}(\alpha)\subseteq ( v_{1}(\alpha))^{\prime\prime}$, whence  $gRm$.
% \end{proof}
%  \textbf{Note:} As semiconcepts are also protoconcepts then above result is also true for the model of {\rm\textbf{PDBL}}. 

\begin{definition}
	\label{dfsatis}
	{\rm Let $\mathbb{M}$ be a model. For $\alpha,\beta\in \mathfrak{F}$  a sequent $\alpha\vdash\beta$ is said to be {\it satisfied} in the model $\mathbb{M}$ if and only if the following hold.
		\begin{enumerate}
			\item For all $g\in G$ $\mathbb{M},g\models\alpha$ implies that $\mathbb{M},g\models\beta$.
			\item For all $m\in M$ $\mathbb{M},m\succ \beta$ implies that $\mathbb{M},m\succ\alpha$.
	\end{enumerate}}
\end{definition}
\begin{proposition}
	\label{satisdbl}
	{\rm Let $\mathbb{M}:=(\mathbb{K}, v)$ be a model. For $\alpha,\beta\in \mathfrak{F}$  a sequent $\alpha\vdash\beta$ is  satisfied in the model $\mathbb{M}$  if and only if $v(\alpha)\sqsubseteq v(\beta)$ in $\mathfrak{H}(\mathbb{K})$, that is, $\alpha\vdash \beta$ is satisfied by $v$ on the pdBa  $\underline{\mathfrak{H}}(\mathbb{K})$.
		%A  sequent $\alpha\vdash\beta$ is {\it true} in  $\mathbb{K}$  if and only if every model $\mathbb{M}$ based on the context $\mathbb{K}$ satisfies the sequent $\alpha\vdash\beta$. A sequent  $\alpha\vdash\beta$ is {\it valid} in the class  $\mathcal{K}$ if and only if it is true  in  $\mathbb{K}$ for all $\mathbb{K}\in \mathcal{K}$.
	}
\end{proposition}
\begin{proof}
	Let $\alpha\vdash\beta$ be satisfied in the model $\mathbb{M}$. Let $g\in v_{1}(\alpha)$. Then $\mathbb{M}, g\models \alpha$, which implies that $\mathbb{M}, g\models \beta$. So $g\in v_{1}(\beta)$, whence $v_{1}(\alpha)\subseteq v_{1}(\beta)$. Let $m\in v_{2}(\beta)$. Then $\mathbb{M}, m\succ\beta$, which implies that $\mathbb{M}, m\succ\alpha$. So $m\in v_{2}(\alpha)$, which implies that $v_{2}(\beta)\subseteq v_{2}(\alpha)$. Therefore $v(\alpha)\sqsubseteq v(\beta)$.
	
	\noindent Conversely, let $v(\alpha)\sqsubseteq v(\beta)$. Then $v_{1}(\alpha)\subseteq v_{1}(\beta)$ and $v_{2}(\beta)\subseteq v_{2}(\alpha)$. Let $\mathbb{M},g\models \alpha$. Then $g\in v_{1}(\alpha)$, which implies that $g\in v_{1}(\beta)$. So $\mathbb{M}, g\models\beta$. Similar to the above, we can show that $\mathbb{M}, m\succ\beta$ implies that $\mathbb{M}, m\succ\alpha$. Therefore $\alpha\vdash\beta$ is satisfied in the model $\mathbb{M}$.
\end{proof}

Next we define satisfaction of an s-hypersequent in a model and validity of an s-hypersequent in the class $\mathcal{K}$ of all contexts. Then we show that \textbf{PDBL} is  sound and complete with respect to the class $\mathcal{K}$.

\begin{definition}
	\label{s-hyper satisf}
	{\rm A hyper-sequent $G:\alpha_{1}\vdash\beta_{1}|\alpha_{2}\vdash\beta_{2}|. . .|\alpha_{n}\vdash\beta_{n}$ is said to be {\it satisfied} in a model $\mathbb{M}$ if and only if at least one of the components of $G$ is satisfied in $\mathbb{M}$.
		
		\noindent  $G$ is said to be {\it true} in $\mathbb{K}$ if and only if $G$ is satisfied in every model $\mathbb{M}$ based on $\mathbb{K}$.
		
		\noindent $G$ is said to be {\it valid} in $\mathcal{K}$ if and only if $G$ is true in every $\mathbb{K}$ for all $\mathbb{K}\in \mathcal{K}$.} 
\end{definition} 
\begin{proposition}
	\label{connectingpropo}
	{\rm Let $\mathbb{M}:=(\mathbb{K}, v)$ be a model. A hyper-sequent $G:\alpha_{1}\vdash\beta_{1}|\alpha_{2}\vdash\beta_{2}|\alpha_{3}\vdash\beta_{3}|\ldots|\alpha_{n}\vdash\beta_{n}$ is satisfied in  $\mathbb{M}$ if and only if $G$ is satisfied by the valuation $v$ on the pdBa $\underline{\mathfrak{H}}(\mathbb{K})$.}
\end{proposition}
\begin{proof}
	Proof follows from Definition \ref{satis-hyper-sequent} and  Proposition \ref{satisdbl}.
\end{proof}
\begin{theorem}[Soundness]
	\label{pdbl case sound}
	{\rm If a hyper-sequent $G:\alpha_{1}\vdash\beta_{1}|\alpha_{2}\vdash\beta_{2}|\alpha_{3}\vdash\beta_{3}|\ldots|\alpha_{n}\vdash\beta_{n}$ is provable in  \textbf{PDBL}  then it is valid in $\mathcal{K}$.
	}
\end{theorem}

\begin{proof}
	%To show $G$ is valid in $\mathcal{K}$, we show that $G$ is true in $\mathbb{K}$ for any $\mathbb{K}\in \mathcal{K}$. Let $\mathbb{K}\in \mathcal{K}$. To show $G$ is true in, we show that $G$ is satisfied in any model based on $\mathbb{K}$.
	Let $\mathbb{M}:=(\mathbb{K}, v)$ be a model based on $\mathbb{K}$.  Theorem \ref{sound1}  and Proposition \ref{connectingpropo} ensure that  $G$ is satisfied in $\mathbb{M}$, as $v$ is also a valuation on the pdBa $\underline{\mathfrak{H}}(\mathbb{K})$.
\end{proof}

\noindent Now, we consider the pair $(\mathbb{K}(\mathcal{L}(\mathfrak{F})), v_{0})$, where $v_{0}$ is as defined in Definition \ref{v0}. By  Lemma \ref{cnnvaluation}, $(\mathbb{K}(\mathcal{L}(\mathfrak{F})), v_{0})$ is a model for $\textbf{PDBL}$. The model is denoted by $\mathbb{M}(\mathfrak{F}):=(\mathbb{K}(\mathcal{L}(\mathfrak{F})), v_{0})$. %Moreover, we have the following.

\begin{theorem}[Completeness]
	\label{comdbl}
	{\rm If an s-hyper-sequent $G:\alpha_{1}\vdash\beta_{1}|\alpha_{2}\vdash\beta_{2}|\ldots|\alpha_{n}\vdash\beta_{n}$ is valid in $\mathcal{K}$ then $G$ is provable in  \textbf{PDBL}.}
	%{\rm Let  $\alpha$ and $\beta$ be formula  then the sequent $\alpha\vdash\beta$ is valid in class $\mathfrak{P}(\mathcal{K})$ then it is  provable in  \textbf{DBL}.
	%$(b.)$ If the sequent $\alpha\vdash\beta$ is valid in class of all frame then it is  provable in {\rm \textbf{PDBL} .}
\end{theorem}

\begin{proof}
	Let $G:= \alpha_{1}\vdash\beta_{1}|\alpha_{2}\vdash\beta_{2}|\ldots|\alpha_{n}\vdash\beta_{n}$ be valid in $\mathcal{K}$. If possible, let us assume that $G$ is not provable in $\textbf{PDBL}$. By Proposition \ref{ldbm-dba1}, for all $i\in\{1,2,\ldots,n\}$, $\alpha_{i}\vdash \beta_{i}$ is not provable  in $\textbf{PDBL}$. By Proposition \ref{ldbm-dba} $[\alpha_{i}]\not\sqsubseteq [\beta_{i}]$ in $\mathcal{L}(\mathfrak{F})$, for all $i\in\{1,2,\ldots,n\}$. Now, we consider the model $\mathbb{M}(\mathfrak{F}):=(\mathbb{K}(\mathcal{L}(\mathfrak{F})),v_{0})$. Using the same argument as the proof of part $1\implies 2$ of Proposition \ref{ldbm-dba1}, we can show that 
 $v_{0}(\alpha_{i})\cancel{\sqsubseteq} v_{0}(\beta_{i})$ for all $i\in \{1,2,3\cdots n\}$,
	which is not possible, as $G$ is valid in the class $\mathcal{K}$.
	So $G$ is provable in $\textbf{PDBL}$. 
\end{proof}

We end the section with an example of characterisation of a class of contexts.
\begin{theorem}
	{\rm  Let $\mathcal{K}$ be a class of contexts. Then the sequent $\top\sqcap\top\dashv\vdash\bot\sqcup\bot$ is valid in  the class $\mathcal{K}$ if and only if  $\mathcal{K}$ contains only contexts of the kind $\mathbb{K}:=(G,M,R)$, where $R=G\times M$.
	}
\end{theorem}

\begin{proof}
	
	For any model $\mathbb{M}:=(\mathbb{K}, v)$, $\top\sqcap\top\dashv\vdash\bot\sqcup\bot$ is satisfied in $\mathbb{M}$ if and only if $v(\top\sqcap\top)=v(\bot\sqcup\bot)$, i.e if and only if $(G, G^{\prime})=(M^{\prime},M)$, i.e. if and only if $G^{\prime}=M$ and $G=M^{\prime}$, which is equivalent to $R=G\times M$.
	%Let $\top\sqcap\top\dashv\vdash\bot\sqcup\bot$ be valid in the class  $\mathcal{K}$ and $\mathbb{K}:=(G,M,R)\in \mathcal{K}$. Then $\top\sqcap\top\dashv\vdash\bot\sqcup\bot$ is true in $\mathbb{K}$. Let $\mathbb{M}:=(\mathbb{K}, v)$ be a model based on the context $\mathbb{K}$. Then  $v(\top\sqcap\top)=v(\bot\sqcup\bot)$. Now $v(\top\sqcap\top)=v(\top)\sqcap v(\top)=(G,\emptyset)\sqcap (G,\emptyset)=(G, G^{\prime})$ and $v(\bot\sqcup\bot)=v(\bot)\sqcup v(\bot)=(\emptyset, M)\sqcup (\emptyset, M)=(M^{\prime}, M)$. So $G=M^{\prime}$ and $G^{\prime}=M$, which implies that $R=G\times M$. 
	%Conversely let $\mathcal{K} $ be the class of context such that $\mathbb{K}:=(G,M,R)\in \mathcal{K}$ if and only if $R=G\times M$. To show $\top\sqcap\top\dashv\vdash\bot\sqcup\bot$ is valid in $\mathcal{K}$, let $\mathbb{K}:=(G,M, R)\in \mathcal{K}$. Now we show $\top\sqcap\top\dashv\vdash\bot\sqcup\bot$ is true in $\mathbb{K}$. Let $\mathbb{M}=(\mathbb{K}, v)$ be a model based on the context $\mathbb{K}$.  $G^{\prime}=M$ and $G=M^{\prime}$, as $R=G\times M$. So $v(\top\sqcap\top)=v(\bot\sqcup\bot)$, which implies that $\top\sqcap\top\dashv\vdash\bot\sqcup\bot$ is satisfied in $\mathbb{M}$.  So $\top\sqcap\top\dashv\vdash\bot\sqcup\bot$ is true in $\mathbb{K}$, which implies that $\top\sqcap\top\dashv\vdash\bot\sqcup\bot$ is valid in $\mathcal{K}$.
\end{proof}

\subsection{Relational semantics for MPDBL}
\label{modalsemantic}

\begin{definition}
	\label{modalmodel}
	{\rm A model for \textbf{MPDBL} is a pair $\mathbb{M}:=(\mathbb{KC}, v)$, where $\mathbb{KC}:=((G,R), (M,S), I)$ is a Kripke context and $v$ is a valuation on $\underline{\mathfrak{H}}(\mathbb{K})$.
		%	. The  map  $v:\textbf{OV}\cup \textbf{PV}\rightarrow\mathfrak{H}(\mathbb{K})$ such that $v(p)\in \mathfrak{H}(\mathbb{K})_{\sqcap}$ and $v(P)\in \mathfrak{H}(\mathbb{K})_{\sqcup}$, is called valuation on $\mathfrak{H}(\mathbb{K})$.
	}
	% \begin{enumerate}
	%	\item $v(\alpha\sqcap\beta)=v(\alpha)\sqcap v(\beta)$ and $v(\alpha\sqcup\beta)=v(\alpha)\sqcup v(\beta)$.
	% 	\item $v(\neg \alpha)=\neg v(\alpha)$ and $v(\lrcorner\alpha)=\lrcorner v(\alpha)$.
	%	\item $v(\square\alpha)=f_{R}(v(\alpha))$ and $v(\textbf{C}\alpha)=f_{S}(v(\alpha))$.
	%\item  $v(\lozenge\alpha)=f_{R}^{\delta}(v(\alpha))$ and $v(\blacklozenge\alpha)=f_{S}^{\delta}(v(\alpha))$.
	%\end{enumerate}}
\end{definition}
\noindent Satisfaction and co-satisfaction for propositional formulae are given in a similar manner as in  Definition \ref{satisfiction and co-satisfiction}, adding the following for modal formulae.
\begin{definition}
	\label{modal formula satisf}
	{\rm Let $\mathbb{M}:=(\mathbb{KC}, v)$ be a model based on the Kripke context $\mathbb{KC}:=((G,R), (M,S), I)$. In the following $g\in G$ and $m\in M$.
		\begin{enumerate}
			\item[15.] $\mathbb{M},g\models \square\alpha$ if and only if for all $g_{1}\in G$ ($gRg_{1}\Longrightarrow\mathbb{M},g_{1}\models \alpha$).
			\item[16.] $\mathbb{M},m\succ\square\alpha$ if and only if for all $g, g_{1}\in G(( gRg_{1} \Longrightarrow \mathbb{M},g_{1}\models\alpha)\Longrightarrow gIm).$
			\item[17.] $\mathbb{M},m\succ\blacksquare\alpha$ if and only if for all $m_{1}\in M$ ($mSm_{1}\Longrightarrow\mathbb{M},m_{1}\succ\alpha$).
			\item[18.] $\mathbb{M},g\models \blacksquare\alpha$ if and only if for all $m,m_{1}\in M(( mSm_{1}\Longrightarrow\mathbb{M},m_{1}\succ\alpha)\Longrightarrow gIm)$
	\end{enumerate}}
\end{definition}
\begin{lemma}
	\label{satdualmodope}
	{\rm For all $g\in G$ and for  all $m\in M$, the following hold.
		\begin{enumerate}
			\item $\mathbb{M},g\models\lozenge\alpha$ if and only if there exists $g_{1}\in G$ such that $gRg_{1}$ and $\mathbb{M},g_{1}\models\alpha$.
			\item $\mathbb{M},m\succ\lozenge\alpha$ if and only if for all $g\in G$  ((there exists $g_{1}\in G$ with $\mathbb{M},g_{1}\models \alpha$ and $gRg_{1}$) $ \Longrightarrow gIm$).
			\item $\mathbb{M},m\succ\blacklozenge\alpha$ if and only if there exists $m_{1}\in M$ such that $mSm_{1}$ and $\mathbb{M},m_{1}\succ\alpha$.
			\item $\mathbb{M},g\models\blacklozenge\alpha$ if and only if for all $m\in M$ ((there exists $m_{1}\in M$ with $mSm_{1}$ and $\mathbb{M},m_{1}\succ\alpha$)  $\Longrightarrow gIm$).
	\end{enumerate}}
	
\end{lemma}
\begin{proof}
	1. Let $g\in G$. Then $\mathbb{M},g\models\lozenge\alpha$ if and only if $\mathbb{M},g\models\neg\square\neg\alpha$, which is equivalent to $\mathbb{M},g\not\models\square\neg\alpha$. Now $\mathbb{M},g\not\models\square\neg\alpha$ if and only if there exists $g_{1}\in G$ such that $gRg_{1}$ and $\mathbb{M},g_{1}\not\models\neg\alpha$, which is equivalent to say that there exists $g_{1}\in G$ such that $gRg_{1}$ and $\mathbb{M},g_{1}\models\alpha$.

	\noindent 2. Let $m\in M$. Then $\mathbb{M},m\succ\lozenge\alpha$ if and only if $\mathbb{M},m\succ\neg\square\neg\alpha$, which is equivalent to  say that for all $g\in G(\mathbb{M},g\not\models\square\neg\alpha~\mbox{implies}~gIm)$. Now $\mathbb{M},g\not\models\square\neg\alpha$ if and only if   there exists $g_{1}\in G$ such that $gR g_{1}$ and $\mathbb{M},g_{1}\not\models\neg\alpha$. So $\mathbb{M},m\succ\lozenge\alpha$ if and only if  for all $g\in G$ ((there exists $g_{1}\in G$ such that $gRg_{1}$ and $\mathbb{M},g_{1}\models\alpha)   \Longrightarrow gIm$).

	\noindent 3. Let $m\in M$. Then $\mathbb{M},m\succ\blacklozenge\alpha$ if and only if $\mathbb{M},m\succ\lrcorner\blacksquare\lrcorner\alpha$, which is equivalent to $\mathbb{M},m\not\succ\blacksquare\lrcorner\alpha$ if and only if there exists a $m_{1}\in M$ such that $mSm_{1}$ and $\mathbb{M},m_{1}\succ\alpha$.

	\noindent 4. Let $g\in G$. Then $\mathbb{M},g\models\blacklozenge\alpha$ if and only if $\mathbb{M},g\models\lrcorner\blacksquare\lrcorner\alpha$, which is equivalent to say that for all $m\in M$ ($\mathbb{M},m\not\succ\blacksquare\lrcorner\alpha$ implies that $gIm$).
	Now $\mathbb{M},m\cancel{\succ}\blacksquare\lrcorner\alpha$ if and only if  there exists $m_{1}\in M$ such that $m S m_{1}$ and $\mathbb{M},m_{1}\succ\alpha$. So $\mathbb{M},g\models\blacklozenge\alpha$ if and only if for all $m\in M$((there exists $m_{1}\in M$ such that $mSm_{1}$ and $\mathbb{M},m_{1}\succ\alpha) \Longrightarrow gIm$).
\end{proof}

\noindent Given a formula $\alpha$, recall the definitions of $v_{1}(\alpha)$ and $v_{2}(\alpha)$. Let $\mathbb{M}:=(\mathbb{K}, v)$ be a model. Then we have the following.
\begin{proposition}
	\label{modal valuation}
	
	\noindent 	{\rm 
		\begin{enumerate}
			\item If $\alpha=\square\phi$ then $v_{1}(\alpha)=\underline{v_{1}(\phi)}_{R}$ and $v_{2}(\alpha)=(\underline{v_{1}(\phi)}_{R})^{\prime}$.
			\item  If $\alpha=\blacksquare\phi$ then $v_{1}(\alpha)=(\underline{v_{2}(\phi)}_{S})^{\prime}$ and $v_{2}(\alpha)=\underline{v_{2}(\phi)}_{S}$.
			\item  If $\alpha=\lozenge \phi$ then $v_{1}(\alpha)=\overline{v_{1}(\phi)}^{R}$ and $v_{2}(\alpha)=(\overline{v_{1}(\phi)}^{R})^{\prime}$.
			\item If $\alpha=\blacklozenge\phi$ then  $v_{1}(\alpha)=(\overline{v_{2}(\phi)}^{S})^{\prime}$ and $v_{2}(\alpha)=\overline{v_{2}(\phi)}^{S}$.
	\end{enumerate}}
\end{proposition}
\begin{proof}
	1.  Let $\alpha=\square\phi$. Then 
	$v_{1}(\alpha)=\{g\in G:\mathbb{M},g\models\square\phi\}=\{g\in G:~\mbox{for all}~g_{1}\in G, (Rgg_{1} \Longrightarrow \mathbb{M},g_{1}\models \phi)\}=\{g\in G:g\in \underline{v_{1}(\phi)}_{R}\}=\underline{v_{1}(\phi)}_{R}$.

	\noindent	$v_{2}(\alpha)=\{m\in M: \mathbb{M},m\succ\square\phi\}=\{m\in M:~\mbox{for all}~g,g_{1}\in G, ((Rgg_{1} \Longrightarrow\mathbb{M},g_{1}\models\phi) \Longrightarrow gIm)\}=(\underline{v_{1}(\phi)}_{R})^{\prime}$. 
	
	%Therefore $(v_{1}(\alpha),v_{2}(\alpha))=(\underline{v_{1}(\phi)}_{R},(\underline{v_{1}(\phi)}_{R})^{\prime})$.
	
	\noindent 2. Let $\alpha=\blacksquare\phi$. Then 
	$v_{2}(\alpha)=\{m\in M:~\mathbb{M},m\succ\blacksquare\phi\}=\{m\in M: ~
	\mbox{for all}~ m_{1}\in M( Smm_{1} \Longrightarrow \mathbb{M},m_{1}\succ\phi)\}=\underline{v_{2}(\phi)}_{S}$.
	
	\noindent $v_{1}(\alpha)=\{g\in G: \mathbb{M},g\models \blacksquare\phi\}=\{g\in G: ~\mbox{for all}~m,m_{1}\in M(( Smm_{1} \Longrightarrow\mathbb{M},m_{1}\succ\phi) \Longrightarrow gIm)\}=(\underline{v_{2}(\phi)}_{S})^{\prime}$. 
	
	%Therefore $(v_{1}(\alpha),v_{2}(\alpha))=((\underline{v_{2}(\phi)}_{S})^{\prime},\underline{v_{2}(\phi)}_{S})$. Hence in the both cases $(v_{1}(\alpha),v_{2}(\alpha))\in \mathfrak{H}(\mathbb{K})$. Therefore by mathematical induction $(v_{1}(\alpha),v_{2}(\alpha))\in \mathfrak{H}(\mathbb{K})$ for any formula $\alpha$. 
	
	\noindent 3. Let $\alpha= \lozenge\phi$. Then by Lemma \ref{satdualmodope}(1), \ref{satdualmodope}(2), $v_{1}(\alpha)=\overline{v_{1}(\neg\neg\phi)}^{R}$  and $v_{2}(\alpha)=(\overline{v_{1}(\neg\neg\phi)}^{R})^{\prime}$, which implies that   $v_{1}(\alpha)=\overline{v_{1}(\neg\neg\phi)}^{R}=\overline{v_{1}(\phi)^{cc}}^{R}=\overline{v_{1}(\phi)}^{R}$  and $v_{2}(\alpha)=(\overline{v_{1}(\neg\neg\phi)}^{R})^{\prime}=(\overline{v_{1}(\phi)}^{R})^{\prime}$.
	% So $v(\beta)=(\overline{v_{1}(\alpha)}^{R}, (\overline{v_{1}(\alpha)}^{R})^{\prime})$, whence $v(\beta)= f^{\delta}_{R}(v(\alpha))$.
	
	\noindent 4. The proof is similar to that of  3.
\end{proof}

\begin{corollary}
	\label{evsemi}
	{\rm For any formula $\alpha$, $(v_{1}(\alpha), v_{2}(\alpha))\in \mathfrak{H}(\mathbb{K})$.}
\end{corollary}
\begin{proof}
	The  corollary is proved using mathematical induction on the number of connectives in $\alpha$, and  the cases are obtained using Corollary  \ref{evproto}, Proposition \ref{modal valuation} and the induction hypothesis. 
	%To prove the corollary, we  use mathematical induction on the number of connectives that occur in the formula $\alpha$.
	%\noindent Let us assume that the result is true for  all formulae with  number of connectives less or equal to $n$. 
	%\noindent Let $\alpha$ be a formula with $n+1$  connectives. Then the following cases are  possible, $\alpha=\beta\sqcap\gamma$, $\alpha=\beta\sqcup \gamma$, $\alpha=\neg\beta$, $\alpha=\lrcorner\beta$, $\alpha=\square\beta$, $\alpha=\blacksquare\beta$, where $\beta, \gamma$ are formulae with number of connectives less or equal to $n$.
	%	\noindent The proof for the propositional cases follows from the proof of Corollary  \ref{evproto}.
	%	\noindent Now, if $\alpha=\square\phi$, by Proposition \ref{modal valuation} $(v_{1}(\alpha), v_{2}(\alpha))= (\underline{v_{1}(\phi)}_{R}, (\underline{v_{1}(\phi)}_{R})^{\prime})$. If $\alpha=\blacksquare\phi$,  by Proposition \ref{modal valuation} $(v_{1}(\alpha), v_{2}(\alpha))= ((\underline{v_{2}(\phi)}_{S})^{\prime}, \underline{v_{2}(\phi)}_{S})$, which implies that $(v_{1}(\alpha), v_{2}(\alpha))\in \mathfrak{H}(\mathbb{K})$. 
	%So by mathematical induction $(v_{1}(\alpha), v_{2}(\alpha))\in \mathfrak{H}(\mathbb{K})$ for all formula $\alpha $.
\end{proof}

\noindent The result analogous to  Corollary \ref{PDBL-valution}  is also true here. Moreover, the following holds.

\begin{corollary}
	\label{mpdbl-valuation}
	{\rm %For any formula $\alpha\in\mathfrak{F}_{1}$. The following hold.
		
		\noindent	\begin{enumerate}
			\item $(v_{1}(\square\alpha), v_{2}(\square\alpha))=f_{R}(v_{1}(\alpha), v_{2}(\alpha))$.
			\item $(v_{1}(\blacksquare\alpha), v_{2}(\blacksquare\alpha))=f_{S}(v_{1}(\alpha),v_{2}(\alpha))$.
	\end{enumerate}}
\end{corollary}
\begin{proof}
	The proof follows from  the definition of extension of a valuation, Proposition \ref{modal valuation} and Corollary \ref{evsemi}.
\end{proof}

\begin{corollary}
	\label{MPDBL-valution1}
	{\rm   For any model $\mathbb{M}:=(\mathbb{KC}, v)$ and formula $\alpha$, $v(\alpha)=(v_{1}(\alpha), v_{2}(\alpha))$. }
\end{corollary}
\begin{proof}
	The proof follows from the definition of extension of a valuation,  Corollaries \ref{evsemi} and  \ref{mpdbl-valuation}.
\end{proof}
%Now, we extend the map $v$ from $\textbf{OV}\cup\textbf{PV}$ to the set of formula $\mathfrak{F}_{1}$ by assign each formula $\alpha$ to $v(\alpha):=(v_{1}(\alpha), v_{2}(\alpha))\in \mathfrak{H}(\mathbb{K})$.
%\begin{corollary}
%\label{mPDBL-valution}
%{\rm A  Kripke context $\mathbb{KC}$ and the map $v: \mathfrak{For}(\textbf{MPDBL})\rightarrow \mathfrak{H}(\mathbb{K})$, definde by $v(\alpha)=(v_{1}(\alpha), v_{2}(\alpha))$ consitute  a model $\mathbb{M}=(\mathbb{KC}, v)$.}
%\end{corollary}
%\begin{proof}
%Let $p\in \textbf{OV}$ and $P\in \textbf{PV}$. Then $v(p)=(v_{1}(p), v_{2}(p))=(\{g\}, \{g\}^{\prime})$ for some $g\in G$ and	$v(P)=(v_{1}(P), v_{2}(P))=(\{m\}^{\prime}, \{m\})$ for some $m\in M$. Rest of the proof follows from Proposition \ref{satisfiction-set} and Proposition \ref{modal valuation}.
%Proof of 1 and 2 is similar to the proof of  Proposition \ref{satisfiction-set}. Proof of 3 follows from the proof of Lemma \ref{evsemi}.\noindent Similar to the above, we can show that $v(\gamma)=f^{\delta}_{S}(v(\alpha))$.
%\end{proof}
%\begin{lemma}
%\label{sat-co-sat-relation}
%{\rm Let $\mathbb{M}=(\mathbb{KC}, v)$ be a model based on a Kripke context $\mathbb{KC}=((G,R), (M,S), I)$ and $g\in G$, $m\in M$. Then $\mathbb{M}, g\models \alpha$ and $\mathbb{M}, m\succ \alpha$ implies that $gIm$.}
%\end{lemma}
\noindent  Satisfaction in a model $\mathbb{M}:=(\mathbb{KC}, v)$ of a sequent $\alpha\vdash\beta$ is  given in a similar manner as  in Definition \ref{dfsatis}. 
%Definitions of satisfaction, truth and validity of  s-hypersequents are given in a similar manner as before. 
%\begin{definition}
%	\label{Mdfsatis}
%	{\rm Let $\mathbb{M}$ be a model. For $\alpha,\beta\in \mathfrak{F}_{1}$  a sequent $\alpha\vdash\beta$ is said to be {\it satisfied} in the model $\mathbb{M}$ if and only if the following hold.
%		\begin{enumerate}
%			\item For all $g\in G$ $\mathbb{M},g\models\alpha$ implies that $\mathbb{M},g\models\beta$.
%			\item For all $m\in M$ $\mathbb{M},m\succ \beta$ implies that $\mathbb{M},m\succ\alpha$.
%	\end{enumerate}}
%\end{definition}
Let $\mathcal{KC}$ denote the class of all Kripke contexts. For an s-hypersequent, the definition of satisfaction  in a model and validity  in the class $\mathcal{KC}$ is given as: 
%of all Kripke contexts. 
%After that, we prove  that \textbf{MPDBL} is  sound and complete with respect to the class $\mathcal{KC}$.

\begin{definition}
	\label{Ms-hyper satisf}
	{\rm  $G:\alpha_{1}\vdash\beta_{1}|\alpha_{2}\vdash\beta_{2}|\ldots|\alpha_{n}\vdash\beta_{n}$ is said to be {\it satisfied} in a model $\mathbb{M}$ if and only if at least one of the components of $G$ is satisfied in $\mathbb{M}$.
		
		\noindent  $G$ is said to be {\it true} in $\mathbb{KC}$ if and only if $G$ is satisfied in every model $\mathbb{M}$ based on $\mathbb{KC}$.
		
		\noindent $G$ is said to be {\it valid} in $\mathcal{KC}$ if and only if $G$ is true in every $\mathbb{KC}$ for all $\mathbb{KC}\in \mathcal{KC}$.} 
\end{definition} 
\noindent Results analogous to  Propositions  \ref{satisdbl} and \ref{connectingpropo} are also true here: 
\begin{proposition}
	\label{satismdbl1}
	{\rm Let $\mathbb{M}:=(\mathbb{KC}, v)$ be a model. For $\alpha,\beta\in \mathfrak{F}_{1}$,  a sequent $\alpha\vdash\beta$ is satisfied in  $\mathbb{M}$  if and only if $v(\alpha)\sqsubseteq v(\beta)$ in $\mathfrak{H}(\mathbb{K})$.
		%A  sequent $\alpha\vdash\beta$ is {\it true} in  $\mathbb{K}$  if and only if every model $\mathbb{M}$ based on the context $\mathbb{K}$ satisfies the sequent $\alpha\vdash\beta$. A sequent  $\alpha\vdash\beta$ is {\it valid} in the class  $\mathcal{K}$ if and only if it is true  in  $\mathbb{K}$ for all $\mathbb{K}\in \mathcal{K}$.
	}
\end{proposition}

\begin{proposition}
	\label{connectingpropo1}
	{\rm An s-hypersequent $G:\alpha_{1}\vdash\beta_{1}|\alpha_{2}\vdash\beta_{2}|\alpha_{3}\vdash\beta_{3}|\ldots|\alpha_{n}\vdash\beta_{n}$ is satisfied in a model $\mathbb{M}:=(\mathbb{KC}, v)$ if and only if $G$ is satisfied by the valuation $v$ on the pdBao $\underline{\mathfrak{H}}^{+}(\mathbb{KC})$.}
\end{proposition}

\begin{theorem}[Soundness]
	\label{soundnessmPDBL}
	{\rm If an s-hypersequent  $G:=\alpha_{1}\vdash \beta_{1}~|~\ldots|\alpha_{n}\vdash\beta_{n}$ is provable in $\textbf{MPDBL}$ then it is valid in the class $\mathcal{KC}$.}
\end{theorem}
\begin{proof}
	%To show $G$ is valid in $\mathcal{KC}$, we show that $G$ is true in $\mathbb{KC}$ for any $\mathbb{KC}\in \mathcal{K}$. Let $\mathbb{KC}\in \mathcal{KC}$. To show $G$ is true in, we show that $G$ is satisfied in any model based on $\mathbb{KC}$.
	Let $\mathbb{M}:=(\mathbb{KC}, v)$ be a model based on $\mathbb{KC}$. By Theorem \ref{sundandcompdbl}(1),   $G$ is satisfied by the valuation $v$ on the pdBao $\underline{\mathfrak{H}}^{+}(\mathbb{KC})$. So by Proposition \ref{connectingpropo1}, $G$ is satisfied in $\mathbb{M}$.
\end{proof}

To prove completeness, we recall the Lindenbaum-Tarski algebra  $\mathcal{L}(\mathfrak{F}_{1})$ and the Kripke context $\mathbb{KC}(\mathcal{L}(\mathfrak{F}_{1})):=((\mathcal{F}_{p}(\mathcal{L}(\mathfrak{F}_{1})), R), (\mathcal{I}_{p}(\mathcal{L}(\mathfrak{F}_{1})), S), \Delta)$ based on the context $\mathbb{K}(\mathcal{L}(\mathfrak{F}_{1}))$   defined in Section \ref{mpdblsigma} and Section \ref{dbawo} respectively. Recall the definitions of $R$ and $S$. Using Note \ref{dual-ldbalg-ope}, we have the following.

\begin{itemize}
	\item[-] For all $F,F_{1}\in \mathcal{F}_{p}(\mathcal{L}(\mathfrak{F}_{1}))$, $FRF_{1}$ if and only if $f_{\lozenge}(a)\in F$ for all $a\in F_{1}$.
	
	\item[-] For all $I,I_{1}\in \mathcal{I}_{p}(\mathcal{L}(\mathfrak{F}_{1}))$, $ISI_{1}$ if and only if $f_{\blacklozenge}(a)\in F$ for all $a\in I_{1}$.
\end{itemize}

\noindent Recall the map $v_{0}: \textbf{OV}\cup\textbf{PV}\cup \{\top, \bot\}\rightarrow \mathfrak{H}(\mathcal{L}(\mathfrak{F}_{1}))$  given  in Definition \ref{v0}. $(\mathbb{KC}(\mathcal{L}(\mathfrak{F}_{1})), v_{0})$ is  a model for \textbf{MPDBL} and denoted as  $\mathbb{M}(\mathfrak{F}_{1})$. 
%Moreover, we have the following.
% Completeness theorem is proved using Lindenbaum-Tarski algebra of $\textbf{MPDBL}$. The Lindenbaum-Traski algebra  of $\textbf{MPDBL}$ is defined in standard way. As all the  axiom and rule of inferences of $\textbf{MPDBL}$ is also axiom and rule of inferences for $\textbf{MPDBL}$, one can show that $\mathcal{L}(\textbf{MPDBL})$ is a pdBao. Moreover, 
%  \begin{note}
%\label{to prove mPDBL model}
%{\rm Lemma \ref{to prove model}  and Proposition \ref{ldbm-dba} are also hold for $\mathcal{L}(\textbf{MPDBL})$, as all axiom and rule of inference of $\textbf{PDBL}$ is also axiom and rule of inference for $\textbf{MPDBL}$.}
%\end{note}
% Now, we consider the Kripke context \\ \mathcal{I}_{p}(\mathcal{L}(\textbf{MPDBL})), \Delta)$. We define a map $v:\textbf{PV}\rightarrow \mathfrak{H}(\mathbb{K}(\mathcal{L}(\textbf{MPDBL}))$ by $v(x)=(F_{[x]}, I_{[x]})$, for all $x\in\textbf{PV}\cup\textbf{OV}\cup \textbf{PV} $. Then $\mathbb{M}^{\textbf{MPDBL}}=(\mathbb{KC}(\mathcal{L}(\textbf{MPDBL})), v)$ is a model. 
\begin{lemma}
	{\rm For any formula $\alpha$, $v_{0}(\alpha)=(F_{[\alpha]}, I_{[\alpha]})$.}
\end{lemma}
\begin{proof}
	We will use mathematical induction on the number of connectives in $\alpha$ to prove this lemma. The base case follows from the definition of $v_{0}$.  Let the claim be true for all formulae $\alpha$ with number of connectives less or equal to $n$.
	
	\noindent Let $\alpha$ be a formula with $n+1$ connectives . Then $\alpha\in \{\beta\sqcup\gamma, \beta\sqcap\gamma,\neg\beta,\lrcorner\gamma\}$ or $\alpha\in \{\square\beta,\blacksquare\beta\}$, where $\beta$ and $\gamma$ are formulae with number of connectives less or equal to $n$.  Proof for the propositional cases is similar to the proof of Lemma \ref{cnnvaluation}. To complete the proof it is sufficient to show that $v_{0}(\alpha)=(F_{[\alpha]}, I_{[\alpha]})$, for $\alpha\in \{\square\beta,\blacksquare\beta\}$.
	\vskip 4pt
	\noindent  Case I: Let $\alpha=\square\beta$. 
	By Proposition \ref{modal valuation}, $v_{1}(\alpha)=\underline{v_{1}(\beta)}_{R}$. By induction hypotheses $v_{1}(\beta)=F_{[\beta]}$. Therefore $v_{1}(\alpha)=\underline{F_{[\beta]}}_{R}$. By Lemma \ref{canonical box and dimon}(1),  $v_{1}(\alpha)= \underline{F_{[\beta]}}_{R}=F_{f_{\square}([\beta])}=F_{[\square\beta]}=F_{[\alpha]}$.
	
	\noindent  By Proposition \ref{modal valuation}, $v_{2}(\alpha)=(\underline{v_{1}(\beta)}_{R})^{\prime}=(F_{[\alpha]})^{\prime}$. By Lemmas \ref{derivation}, \ref{complement of Fx}, and axiom 17a, $v_{2}(\alpha)=(F_{[\alpha]})^{\prime}=I_{[\alpha]_{\sqcap \sqcup}}=I_{[\alpha]_{\sqcap}}=I_{[\alpha]}$. Therefore $v_{0}(\alpha)=(F_{[\alpha]}, I_{[\alpha]})$.
	\vskip 4pt
	\noindent Case II: Let $\alpha=\blacksquare\phi$. Then similar to the above, we can show that $v_{2}(\alpha)=I_{[\alpha]}$ and $v_{1}(\alpha)=(v_{2}(\alpha))^{\prime}=F_{[\alpha]}$. In this case we use Lemmas \ref{derivation}, \ref{complement of Fx}, and axiom 17b. Therefore $v_{0}(\alpha)=(F_{[\alpha]},I_{[\alpha]})$.
	%\noindent  Hence by mathematical induction,  $v_{0}(\alpha)=(F_{[\alpha]},I_{[\alpha]})$ for all formula $\alpha$.
\end{proof} 
\begin{theorem}[Completeness]
	\label{compl-mPDBL}
	{\rm If an s-hypersequent $G:\alpha_{1}\vdash\beta_{1}~|~\ldots~|~\alpha_{n}\vdash\beta_{n}$  is valid in $\mathcal{KC}$ then $G$ is provable in \textbf{MPDBL}.}
\end{theorem}
\begin{proof}
	Let the s-hypersequent $G:\alpha_{1}\vdash\beta_{1}~|~\ldots~|~\alpha_{n}\vdash\beta_{n}$  be valid in $\mathcal{KC}$. If possible,  assume that $G$ is not provable in \textbf{MPDBL}. By   Propositions \ref{ldbm-mdba} and  \ref{ldbm-mdba1}, $[\alpha_{i}]\not\sqsubseteq [\beta_{i}]$ for all $i\in \{1,2,3,\ldots, n\}$. The rest of the proof is similar to the proof of Theorem \ref{comdbl}. % **  giving us a contradiction.
\end{proof}

Let $\mathcal{KC}_{RT}$ be the class of all reflexive and transitive Kripke contexts. In the following theorems, we show that \textbf{MPDBL4} is sound and complete with respect to the class $\mathcal{KC}_{RT}$.

%Next, we define a logic \textbf{MPDBL4} for the class $\mathcal{KC}_{RT}$, that is a logic, which is sound and complete with respect to the class $\mathcal{KC}_{RT}$.
%The logic \textbf{MPDBL4} is obtained from \textbf{MPDBL} by adding the following axiom.
%  \begin{center}
%	$\begin{array}{l l}
%		18a)~ \square\alpha\vdash\alpha &
%		18b)~ \alpha\vdash\blacksquare\alpha\\
%		19a)~ \square\square\alpha\dashv\vdash\square\alpha & 
%		19b)~ \blacksquare\blacksquare\alpha\dashv\vdash\blacksquare\alpha
%	\end{array}$
%\end{center}

\begin{theorem}[Soundness]
	{\rm If an s-hypersequent $G:\alpha_{1}\vdash\beta_{1}\ldots \alpha_{n}\vdash\beta_{n}$ is provable in \textbf{MPDBL4} then $G$ is valid in $\mathcal{KC}_{RT}$.}
\end{theorem}
\begin{proof}
	%	To show $G$ is valid in $\mathcal{KC}_{RT}$, we show that $G$ is true in $\mathbb{KC}$ for any $\mathbb{KC}\in \mathcal{KC}_{RT}$. Let $\mathbb{KC}\in \mathcal{KC}_{RT}$. To show $G$ is true in, we show that $G$ is satisfied in any model based on $\mathbb{KC}$.
	The proof is a consequence of Theorem \ref{sundandcompdbl}(2) and Proposition \ref{connectingpropo1}, as for a model	$\mathbb{M}:=(\mathbb{KC}, v)$, $v$ is also  a valuation on the tpdBa $\underline{\mathfrak{H}}^{+}(\mathbb{KC})$. 
\end{proof}

\begin{theorem}[Completeness]
	\label{compl-mPDBL4}
	{\rm If an s-hypersequent $G:=\alpha_{1}\vdash\beta_{1}\ldots\alpha_{n}\vdash\beta_{n}$ is valid in $\mathcal{KC}_{RT}$ then $G$ is provable in \textbf{MPDBL4}.}
\end{theorem}
\begin{proof}
	In case of \textbf{MPDBL4}, $\mathcal{L}_{\Sigma}(\mathfrak{F}_{1})$ forms a tpdBa. By Theorem \ref{rttdBa} it follows that  $\mathbb{KC}(\mathcal{L}(\mathfrak{F}_{1}))$ is reflexive and transitive. Rest of the proof is similar to the proof of Theorem \ref{compl-mPDBL}.
	%To prove this theorem, we show that and rest of the proof is similar to the proof of Theorem \ref{compl-mPDBL}. To show $\mathbb{KC}(\textbf{MPDBL4})=((\mathcal{F}_{p}(\textbf{MPDBL4}), R), (\mathcal{I}_{p}(\textbf{MPDBL4}), S), \Delta)$ is a reflexive and transitive Kripke context, we prove that $R$ and $S$ are reflexive and transitive. To show $R$ is reflexive, let $F\in \mathcal{F}_{p}(\textbf{MPDBL4})$ and $[\alpha]\in F$. As $F$ is primary filter, $\neg [\alpha]\notin F$, which implies that $I(\neg [\alpha])\notin F$, otherwise $\neg[\alpha]\in F$. So $\neg I(\neg[\alpha])\in F$, which implies that $I^{\delta}[\alpha]\in F$. So $RFF$, which implies that $R$ is reflexive.
	% \noindent To show $R$ is transitive, let $F, F_{1}, F_{2}\in \mathcal{F}_{p}(\textbf{MPDBL4})$ such that $RFF_{1}$ and $RF_{1}F_{2}$. Then for all $[\alpha]\in F_{1}$, implies that $I^{\delta}[\alpha]\in F$ and for all $[\alpha]\in F_{2}$ implies that $I^{\delta}[\alpha]\in F_{1}$. Let $[\alpha]\in F_{2}$. Then $I^{\delta}[\alpha]\in F_{1}$, which implies that $I^{\delta}I^{\delta}[\alpha]\in F$. By Lemma \ref{tdBadual} $I^{\delta}[\alpha]\in F$, which implies that $RFF_{2}$. So $R$ is transitive.
	% \noindent Similar to the above, we can show that $S$ is reflexive and transitive.
\end{proof}

Let $\mathcal{KC}_{RST}$ be the class of all reflexive, symmetric and transitive Kripke contexts. Next, we propose a logic \textbf{MPDBL5} for the class $\mathcal{KC}_{RST}$. The logic \textbf{MPDBL5} is obtained from  \textbf{MPDBL4} by adding the following sequents as axioms. 
\[21a~\lozenge\alpha\vdash\square\lozenge\alpha \hspace*{20pt}  21b~\blacksquare\blacklozenge\alpha\vdash\blacklozenge\alpha.\]

\begin{theorem}[Soundness]
	{\rm If an s-hypersequent $G:\alpha_{1}\vdash\beta_{1}\ldots \alpha_{n}\vdash\beta_{n}$ is provable in \textbf{MPDBL5} then $G$ is valid in $\mathcal{KC}_{RST}$.}
\end{theorem}
\begin{proof}
	To complete the proof it is sufficient to show that 21a and 21b are valid in $\mathcal{KC}_{RST}$. To show 21a is valid in $\mathcal{KC}_{RST}$, let $\mathbb{KC}\in\mathcal{KC}_{RST} $ and $\mathbb{M}:=(\mathbb{KC}, v)$ be a model based on $\mathbb{KC}$. 
	
	\noindent By Corollary  \ref{mpdbl-valuation} and \ref{MPDBL-valution1}, $v(\square\lozenge\alpha)=f_{R}f^{\delta}_{R}((v_{1}(\alpha), v_{2}(\alpha)))=f_{R}((\overline{v_{1}(\alpha)}^{R}, (\overline{v_{1}(\alpha)}^{R})^{\prime}))=(\underline{(\overline{v_{1}(\alpha)}^{R})}_{R}, (\underline{(\overline{v_{1}(\alpha)}^{R})}_{R})^{\prime} )$.  Let $x\in \overline{v_{1}(\alpha)}^{R}$. Then $R(x)\cap v_{1}(\alpha)\neq\emptyset$. Let $z_{0}\in R(x)\cap v_{1}(\alpha)$ for some $z_{0}\in G$. Then  $xRz_{0}$.  Let $y\in R(x)$. Then $xRy$, which implies that $yRx$, as $R$ is symmetric. So $yRz_{0}$, as $R$ is transitive, which implies that $R(y)\cap v_{1}(\alpha)\neq\emptyset$. So $y\in \overline{v_{1}(\alpha)}^{R}$, whence $R(x)\subseteq \overline{v_{1}(\alpha)}^{R}$. So $x\in \underline{(\overline{v_{1}(\alpha)}^{R})}_{R} $, which implies that $\overline{v_{1}(\alpha)}^{R}\subseteq \underline{(\overline{v_{1}(\alpha)}^{R})}_{R}$. So $(\underline{(\overline{v_{1}(\alpha)}^{R})}_{R})^{\prime}\subseteq (\overline{v_{1}(\alpha)}^{R})^{\prime}$, whence $(\overline{v_{1}(\alpha)}^{R}, (\overline{v_{1}(\alpha)}^{R})^{\prime})\sqsubseteq (\underline{(\overline{v_{1}(\alpha)}^{R})}_{R}, (\underline{(\overline{v_{1}(\alpha)}^{R})}_{R})^{\prime})$. So $v(\lozenge\alpha)\sqsubseteq v(\square\lozenge\alpha)$. 
	
	\noindent By Proposition \ref{satismdbl1}, $\lozenge\alpha\vdash\square\lozenge\alpha$ is satisfied in  $\mathbb{M}$, which implies that $\lozenge\alpha\vdash\square\lozenge\alpha$ is true in $\mathbb{KC}$. So $\lozenge\alpha\vdash\square\lozenge\alpha$ is valid in $\mathcal{KC}_{RST}$.
	
	\noindent Similar to the above proof, we can show that $\blacksquare\blacklozenge\alpha\vdash\blacklozenge\alpha$ is valid in $\mathcal{KC}_{RST}$.
\end{proof}
%Completeness theorem is proved using the  Lindenbaum-Tarski algebra $\mathcal{L}(\textbf{MPDBL5})$ of \textbf{MPDBL5}.
\begin{theorem}[Completeness]
	{\rm If an s-hypersequent $G:=\alpha_{1}\vdash\beta_{1}\ldots\alpha_{n}\vdash\beta_{n}$ is valid in $\mathcal{KC}_{RST}$ then $G$ is provable in \textbf{MPDBL5}.}
\end{theorem}
\begin{proof}
	We show that $\mathbb{KC}(\mathcal{L}_{\Sigma}(\mathfrak{F}_{1}))$ is a symmetric Kripke context and rest of the proof is similar to the proof of Theorems \ref{compl-mPDBL} and \ref{compl-mPDBL4}. 
	
	% \noindent To show  $\mathbb{KC}(\mathcal{L}_{\Sigma}(\mathfrak{F}_{1}))$ is a symmetric Kripke context, we show that $R$ and $S$ are symmetric relations. 
	
	\noindent To show $R$ is symmetric, let $F, F_{1}\in \mathcal{F}_{p}(\mathcal{L}_{\Sigma}(\mathfrak{F}_{1}))$ and $FRF_{1}$. Then for all $[\alpha]\in F_{1}$, $f_{\lozenge}([\alpha])=[\lozenge\alpha]\in F$. By axiom 21a and  Proposition \ref{ldbm-mdba}, $f_{\lozenge}([\alpha])=[\lozenge\alpha]\sqsubseteq [\square\lozenge\alpha]=f_{\square}f_{\lozenge}([\alpha])$. By Lemma \ref{tdBadual}(2), $[\alpha]\sqcap [\alpha]\sqsubseteq f_{\lozenge}([\alpha]) $, as $\mathcal{L}_{\Sigma}(\mathfrak{F}_{1})$ is a pdBao. So $[\alpha]\sqcap [\alpha]\sqsubseteq f_{\square}f_{\lozenge}([\alpha])$.  Let $[\alpha]\in F$. Then $[\alpha]\sqcap [\alpha]\in F$, as $F$ is a filter. So  $f_{\square}f_{\lozenge}([\alpha])\in F$, as $F$ is a filter.  By Lemma \ref{canonical relations}(1), $f_{\lozenge}([\alpha])\in F_{1}$,  which implies that $RF_{1}F$. So  $R$ is symmetric.
	
	\noindent Similar to the above, we can show that $S$ is also symmetric. %So $\mathbb{K}(\mathcal{L}_{\Sigma}(\mathfrak{F}_{1})))$ is symmetric Kripke context. 
\end{proof}
\section{  PDBL and conceptual knowledge}
\label{sematics to meaning1}
In this part, we look at $\textbf{PDBL}$ through the lens of  conceptual  knowledge \cite{lpwille}. 
 As discussed in \cite{lpwille}, the main assumption for conceptual knowledge is that it must be expressible by the  three basic notions of objects, attributes and concepts,  and these three are linked by the four basic relations  ``an object has an attribute'', ``an object belongs to a concept'', ``an attribute abstracts from a concept'' and ``a concept is a subconcept of another concept''. 
  %The main assumption for  conceptual knowledge which is discussed in \cite{lpwille} is that it must be expressible by three basic notions of objects, attributes, and concepts, and these three are linked by four basic relations  ``an object has an attribute'', ``an object belongs to a concept'', ``an attribute abstracts from a concept'', and ``a concept is a subconcept of another concept'' 
These three basic notions and relations are represented mathematically using the concept lattice of a context. Let us recall this briefly.
For each object $g\in G$, there  is $(\{g\}^{\prime\prime}, \{g\}^{\prime})\in \mathcal{B}(\mathbb{K})$,  the smallest concept   containing $g$ in its extent and for each attribute $m\in  M$, the largest concept $(\{m\}^{\prime}, \{m\}^{\prime\prime})\in \mathcal{B}(\mathbb{K})$ 
containing $m$ in its intent. Additionally, $gIm \Leftrightarrow (\{g\}^{\prime\prime}, \{g\}^{\prime}) \sqsubseteq (\{m\}^{\prime}, \{m\}^{\prime\prime})$ in $\mathcal{B}(\mathbb{K})$. Now the concept lattice $\mathcal{B}(\mathbb{K})$ depicts the objects and  attributes of the context $\mathbb{K}$, if each object $g$ is identified with $(\{g\}^{\prime\prime}, \{g\}^{\prime})$ and each attribute $m$  with $(\{m\}^{\prime}, \{m\}^{\prime\prime})$. On the other hand, the four relations are describable in $\mathcal{B}(\mathbb{K})$ as follows. The object $g$ has the attribute $m$ if and only if $(\{g\}^{\prime\prime}, \{g\}^{\prime}) \sqsubseteq (\{m\}^{\prime}, \{m\}^{\prime\prime})$, the
object $g$ belongs to the concept $(A, B)$ if and only if  $(\{g\}^{\prime\prime}, \{g\}^{\prime}) \sqsubseteq (A,B)$, the attribute $m$
abstracts from the concept $(A, B)$ if and only if $(A, B) \sqsubseteq (\{m\}^{\prime}, \{m\}^{\prime\prime})$, and the concept $(A_{1}, B_{1} )$
is a subconcept of the concept $(A_{2} , B_{2} )$ if and only if $(A_{1}, B_{1} )\sqsubseteq (A_{2} , B_{2} )$. 

%Motivated by the above observations, Luksch et al. \cite{lpwille} proposed a mathematical model for conceptual knowledge; such a mathematical model is known as a {\it conceptual knowledge system} \cite{lpwille}. 

%As mentioned in the introduction,  Luksch et al. \cite{lpwille}
%He also pointed out that knowledge involving negation can not be express within the frame work of concept lattice.  To full fill the pourpose, he 
%introduced ``negation of a concept'', through the notion of a semiconcept, lead to the algebra of semiconcepts,  $\underline{\mathfrak{H}}(\mathbb{K})$. 

Luksch et al. \cite{lpwille}  assume that all conceptual knowledge for a given field of interest is derived from a comprehensive formal context $U:=(G_{U}, M_{U}, I_{U})$, which is referred to as a {\it conceptual universe} for the field of interest.   

Theorem \ref{cs} implies that a concept can be expressed using semiconcepts. As  $\textbf{PDBL}$ is a logic for  semiconcepts of a context, %and  $\mathfrak{B}(\mathbb{K})\subseteq \mathfrak{H}(\mathbb{K})$
 it is natural to expect that all the three basic notions and the four basic  relations of conceptual knowledge can also be represented in the system. In the rest of the section, we  establish this.  Recall that a model for $\textbf{PDBL}$ is $\mathbb{M}:=(\mathbb{K}, v)$, where $v:\textbf{OV}\cup\textbf{PV}\cup \{\top, \bot\}\rightarrow\mathfrak{H}(\mathbb{K})$ is a valuation.  From Corollary \ref{evproto}, it follows that each formula $\alpha$ represents a semiconcept $v(\alpha)=(v_{1}(\alpha), v_{2}(\alpha))$, where $v_{1}(\alpha):=\{g\in G~:~\mathbb{M}, g\models\alpha\}$ and $v_{2}(\alpha):=\{m\in M~:~\mathbb{M}, m\succ\alpha\}$. Let  $\alpha$ be a formula and $g\in G$. Then $\mathbb{M},g\models \alpha$ 
represents that $\alpha$ is satisfied at $g$ in a model $\mathbb{M}$. In other words, the satisfaction relation $\models$ may be considered as a relation between $G$ and $\mathfrak{F}$, i.e., $\models\subseteq G\times \mathfrak{F}$. $\mathbb{M},g\models \alpha$ if and only if $g\in v_{1}(\alpha)$, which is equivalent to ``the object $g$ belongs to the semiconcept $v(\alpha)$''. %So the satisfaction relation $\models \subseteq G\times \mathfrak{F}$ represents the relation ``the object $g$ belongs to the semiconcept $v(\alpha)$''.
 Similar to the above, the co-satisfaction relation $\succ\subseteq M\times \mathfrak{F}$ represents the relation  ``the property $m$ abstracts from the semiconcept $v(\alpha)$''. %Moreover, Proposition \ref{rep of concept} says that the satisfaction and co-satisfaction relations are also able to represent the relations  `object  belongs to concept' and `attribute abstracts from concept'.
 
The next proposition tells us when a formula represents a concept.  
\begin{proposition}
	\label{rep of concept}
	{\rm Let $\alpha\in\mathfrak{F}$  and $\mathbb{M}:=(\mathbb{K}, v)$ be a model. Then the sequents $\alpha\sqcap\alpha\dashv\vdash\alpha\sqcup\alpha, \alpha\sqcap\alpha\dashv\vdash\alpha$ and $\alpha\dashv\vdash\alpha\sqcup\alpha$ are satisfied in $\mathbb{M}$ if and only $v(\alpha)$ is a concept of $\mathbb{K}$.  }
\end{proposition}
\begin{proof}
	Let $\mathbb{K}:=(G, M, I)$ and $A\subseteq G$, $B\subseteq M$. Now observe that the pair $(A,B)$ is a concept of $\mathbb{K}$ if and only if $(A,B)\sqcup (A,B)=(A,B)\sqcap (A,B)=(A, B)$. As $v$ is a valuation, $v$ preserve $\sqcap$, $\sqcup$ %$v: \mathbf{OV}\cup \mathbf{PV}\cup \{\top, \bot\}\rightarrow \mathfrak{H}(\mathbb{K})$, 
	and in $\mathfrak{H}(\mathbb{K})$ $\sqsubseteq$ is partial order. The proof is a consequence of the  observations.
\end{proof}

\begin{definition}
	{\rm For a model $\mathbb{M}:=(\mathbb{K}, v)$, $\mathfrak{F}_{\mathbb{M}}:=\{\alpha\in \mathfrak{F}~:~ \alpha\sqcap\alpha\dashv\vdash\alpha\sqcup\alpha, \alpha\sqcap\alpha\dashv\vdash\alpha~\mbox{ and}~ \alpha\dashv\vdash\alpha\sqcup\alpha~\mbox{ are satisfied in}~ \mathbb{M}\}$.}
\end{definition} 
In other words, $\mathfrak{F}_{\mathbb{M}}=\{\alpha\in \mathfrak{F} : v(\alpha)\in \mathcal{B}(\mathbb{K})\}$. In the next proposition, we characterize the set $\mathfrak{F}_{\mathbb{M}}$ for a class of models  that is based on the class $\mathcal{K}_{*}$ of contexts defined as $\mathcal{K}_{*}:=\{ \mathbb{K}:=(G, M, I) : I=G\times M\}$.
% that is the set of all concept of the context $\mathbb{K}$. 
\begin{proposition}
	\label{characoncept}
	{\rm  Let $\mathbb{M}:=(\mathbb{K}, v)$ be a model based on the context $\mathbb{K}\in \mathcal{K}_{*}$. The following hold.
		\begin{enumerate}
			\item For $\mathbb{K}\in \mathcal{K}_{*}$, $\mathcal{B}(\mathbb{K})=\{(G, M)\}$,
			\item $\top\sqcap\top\in \mathfrak{F}_{\mathbb{M}}$.
			\item For any $\alpha\in \mathfrak{F}$, $\alpha\in \mathfrak{F}_{\mathbb{M}}$ if and only if for any  $\beta\in \mathfrak{F}_{\mathbb{M}}$, $\beta \dashv\vdash\alpha$ is satisfied in the model $\mathbb{M}$.
		\end{enumerate}
	}
\end{proposition}
\begin{proof}
	\noindent  1. Let  $(A, B)$ be a concept of $\mathbb{K}$, which implies that $B=A^{\prime}=M$ and $A=B^{\prime}=G$. So $(A, B)=(G, M)$, which implies that $\mathcal{B}(\mathbb{K})=\{(G, M)\}$.

	\noindent 2. $v(\top\sqcap\top)=v(\top)\sqcap v(\top)=(G, \emptyset)\sqcap (G,\emptyset)=(G, G^{\prime})=(G, M)\in \mathcal{B}(\mathbb{K})$. 	%and $v((\top\sqcap\top)\sqcup \top\sqcap\top))=(G, M)\sqcup (G, M)=(M^{\prime}, M)=(G, M)$, which implies that $v((\top\sqcap\top)\sqcap (\top\sqcap\top))= v(\top\sqcap\top)= v((\top\sqcap\top)\sqcup \top\sqcap\top))$. 
	So $\top\sqcap \top\in \mathfrak{F}_{\mathbb{M}}$.
	
	\noindent  3. Let $\alpha, \beta \in  \mathfrak{F}_{\mathbb{M}} $. Then $v(\alpha)$ and $v(\beta)$ are concepts of $\mathbb{K}$, which implies that $v(\alpha)=(G, M)=v(\beta)$. So  $\beta\dashv\vdash\alpha$ is satisfied in the model $\mathbb{M}$.
	
	\noindent Conversely, let $\alpha\in \mathfrak{F}$, and take $\beta\in \mathfrak{F}_{\mathbb{M}}$ such that $\beta \dashv\vdash\alpha$ is satisfied in the model $\mathbb{M}$. Then $v(\alpha)=v(\beta)=(G, M)\in \mathcal{B}(\mathbb{K})$. So $\alpha\in \mathfrak{F}_{\mathbb{M}}$.
\end{proof}
\noindent As a consequence of Proposition \ref{characoncept}(2,3), for any $\alpha\in\mathfrak{F}_{\mathbb{M}}$, $\top\sqcap\top\dashv\vdash\alpha$ is satisfied in each model $\mathbb{M}$ based on $\mathbb{K}\in \mathcal{K}_{*}$.

\vskip 3pt 
From Proposition \ref{rep of concept}, it follows that the restriction $\models_{1}$ of the satisfaction relation $\models$ to the set $G\times \mathfrak{F}_{\mathbb{M}}$, represents the relation  ``the object $g\in G$ belongs to the concept $v(\alpha), \alpha\in \mathfrak{F}_{\mathbb{M}}$'', while  the restriction $\succ_{1}$ of the co-satisfaction relation $\succ$ to the set $M\times \mathfrak{F}_{\mathbb{M}}$ represents the relation  ``the property $m\in M$ abstracts from  the concept $v(\alpha), \alpha\in \mathfrak{F}_{\mathbb{M}}$''. 

\noindent Let $\mathbb{M}:=(\mathbb{K}, v)$ be a model based on $\mathbb{K}$ and $\alpha\vdash \beta$ be a valid sequent in $\textbf{PDBL}$ such that $\alpha, \beta\in \mathfrak{F}_{\mathbb{M}}$. This implies that $v(\alpha)\sqsubseteq v(\beta)$  and $v(\alpha)$, $v(\beta)$ are concepts of $\mathbb{K}$. So the restriction of the relation $\vdash \subseteq \mathfrak{F}\times\mathfrak{F}$ to the set $\mathfrak{F}_{\mathbb{M}}\times\mathfrak{F}_{\mathbb{M}}$ represents the subconcept-superconcept relation.

%that is, we consider the sequent $\alpha\vdash\beta$, where $\alpha$ and $\beta$ represent concept in $\mathbb{M}$

\vskip 3pt 

Thus we have shown that using \textbf{PDBL} and its models, one can express the notion of concept, the relations ``object belongs to a concept'', ``property abstracts from a concept'' and ``a concept is a subconcept of another concept''. What about the other two basic notions of objects and attributes, and   the relation ``an object has an attribute''?  These 
%two entities and the relation  
can also be represented in the system \textbf{PDBL}, by modifying the definition of a model.

For any context $\mathbb{K}:=(G, M, I)$ and for $g\in G$, $(\{g\}, \{g\}^{\prime})$ is called the {\it semiconcept of $\mathbb{K}$ generated by the object $g$} and for $m\in M$, $(\{m\}^{\prime}, \{m\} )$ is called the {\it semiconcept of $\mathbb{K}$ generated by the attribute $m$}.
%Now, consider the context $U:=(G_{U}, M_{U}, I)$. 
We note the following maps.
\begin{itemize}
	\item[-] $\zeta: G\rightarrow \{(\{g\}, \{g\}^{\prime})~:~g\in G\}(\subseteq \mathfrak{H}(\mathbb{K}))$, where $g\mapsto (\{g\}, \{g\}^{\prime})$ for all $g\in G$. 
	\item[-] $\eta: M\rightarrow \{(\{m\}^{\prime}, \{m\})~:~m\in M\}(\subseteq \mathfrak{H}(\mathbb{K}))$, where $m \mapsto (\{m\}^{\prime}, \{m\})$ for all $m\in M$. 
\end{itemize}
\noindent $\zeta$ and $\eta$ are  bijections. So the sets $\mathcal{OS}:=\{(\{g\}, \{g\}^{\prime})~:~g\in G_\}$ and  $\mathcal{PS}:=\{(\{m\}^{\prime}, \{m\})~:~m\in M\}$ of semiconcepts  can be used to describe respectively, objects and  properties  in  $\mathbb{K}$. Taking a cue from  hybrid modal logic \cite{blackburn2002moda} and the above observation, we have the following.

%\begin{definition}
%	{\rm  A \textbf{PDBL} model $\mathbb{M}:=(\mathbb{K}, v)$ is called a {\it named model} if for any $p\in \textbf{OV}$, $v(p)=(\{g\}, \{g\}^{\prime})$ for some $g\in G$, and for any $P\in \textbf{PV}$, $v(P)=(\{m\}^{\prime}, \{m\})$ for some $m\in M$. Moreover, a named model is called {\it surjective} if  for each semiconcept $(\{g\}, \{g\}^{\prime})$ generated by the object $g $, there is $p\in \textbf{OV}$ such that $v(p)=(\{g\}, \{g\}^{\prime})$,  and for each semiconcept $(\{m\}^{\prime}, \{m\})$ generated by the attribute $m $, there is $P\in \textbf{PV}$ such that $v(P)=(\{m\}^{\prime}, \{m\})$. }
%\end{definition}

\begin{definition}
        {\rm  A \textbf{PDBL} model $\mathbb{M}:=(\mathbb{K}, v)$ is called a
{\it named model} if for any $p\in \textbf{OV}$, $v(p)=(\{g\},
\{g\}^{\prime})$ for some $g\in G$, and for any $P\in \textbf{PV}$,
$v(P)=(\{m\}^{\prime}, \{m\})$ for some $m\in M$. Moreover,  for each
semiconcept $(\{g\}, \{g\}^{\prime})$ generated by the object $g $, there
is $p\in \textbf{OV}$ such that $v(p)=(\{g\}, \{g\}^{\prime})$,  and for
each semiconcept $(\{m\}^{\prime}, \{m\})$ generated by the attribute $m
$, there is $P\in \textbf{PV}$ such that $v(P)=(\{m\}^{\prime}, \{m\})$.
}
\end{definition}

\noindent Under such an interpretation of $\textbf{PDBL}$, the object variables and property variables represent the objects and attributes of  conceptual knowledge. \\
Now let $\mathbb{M}:=(\mathbb{K}, v)$ be a named model. For $g\in G$ and $m\in M$, $gIm$ (``object $g$ has the attribute $m$'') if and only if $(\{g\}, \{g\}^{\prime})\sqsubseteq  (\{m\}^{\prime}, \{m\})$, 
%(note that $\{g\}\subseteq \{g\}^{\prime\prime}$), 
which is equivalent to saying that the sequent $p\vdash P$ is satisfied in  the model $\mathbb{M}$, where $p$ represents the object $g$ and $P$ represents the property $m$ in  $\mathbb{M}$.

Note that a named model is also a model  according to Definition \ref{pdBamodel}. So if an s-hypersequent $G$ is valid in the class $\mathcal{K}$ then it is also satisfied in  all named models based on all contexts $\mathbb{K}\in \mathcal{K}$. From the above, we can conclude that if we give  Definitions \ref{dfsatis} and  \ref{s-hyper satisf} in terms of named models then \textbf{PDBL} remains sound with respect to $\mathcal{K}$. 

\vskip 3pt 
 Next, we give an example of a named model based on the context given in Table \ref{CU}.
	\begin{table}[h]
		\centering
		\caption { %The Zink:
			$\mathbb{K}_{z}$}	\label{CU}

		\scalebox{.70}{\begin{tabular}{|c|c|c|c|c|c|c|c|c|c|c|c|c|c|c|c|c|c|c|c|c|c|c|c|c|c|c|}
				\hline
				& \rotatebox[origin=l]{90}{Elfenbein Korpus}& \rotatebox[origin=l]{90}{Leder Uberzug}&\rotatebox[origin=l]{90}{Pergament Uberzug}&\rotatebox[origin=l]{90}{Messingschallst{\"u}ck}& \rotatebox[origin=l]{90}{aufgesetztes Mundst{\"u}ck}&\rotatebox[origin=l]{90}{eingedrehtes Mundst{\"u}ck}&\rotatebox[origin=l]{90}{gerade Form}&\rotatebox[origin=l]{90}{geboge Form}&\rotatebox[origin=l]{90}{$es^{1}$(Stimmgr{\"o}{\ss }e)}&\rotatebox[origin=l]{90}{$d^{1}$(Stimmgr{\"o}{\ss }e)}&\rotatebox[origin=l]{90}{a(Stimmgr{\"o}{\ss }e)}&\rotatebox[origin=l]{90}{g(Stimmgr{\"o}{\ss }e)}&\rotatebox[origin=l]{90}{b(Stimmgr{\"o}{\ss }e)}\\ 
				\hline
				%	\mbox{Gerader Zink}  && & & & & & $\times$ &&&&&&&&&\\
				%	\hline
				%	\mbox{Stiller Zink}  && & & & &$\times$ & $\times$&&&&&& & &&\\
				%\hline
				% \mbox{Krummer Zink} & & & & & $\times$& & & $\times$&&&&& & &&\\
				%\hline
				% \mbox{Cornettino} & & & & & $\times$& & & $\times$&&&&& & &&\\
				%\hline
				%\mbox{Cornetto} & & & & & $\times$& & & $\times$&&&&& & &&\\
				% \hline
				% \mbox{Tenorzink} & & & & & $\times$& & & $\times$&&&&& & &&\\
				% \hline
				%  \mbox{Serpent}& & & & & $\times$& & & $\times$&&&&& & &&\\
				% \hline
				% \mbox{Violoncel-Serpent} & & & & & $\times$& & & $\times$&&&&& & &&\\
				% \hline
				\mbox{1558} &&$ $&$ $&$*$&$*$&$ $&$*$&$ $&$ $&$ $&$ $&$*$&$ $  \\
				\hline
				\mbox{1559}&&&&&&$*$&$*$&$ $&$ $&$ $&$*$&&$ $\\
				\hline
				\mbox{1560} &&&&& &$*$&$*$&$ $&$ $&$ $&$*$&$ $&$ $  \\
				\hline
				\mbox{1561}& &&&&&$*$&$*$&$ $&$ $&$ $&$ $&$*$&$ $ \\
				\hline
				\mbox{1562} &&&&&&$*$&$*$&$ $&$ $&$ $&$ $&$*$&$ $\\
				\hline
				\mbox{1563}& &$*$&&$ $&$*$&$ $&$ $&$*$&$*$&$ $&$ $&$ $&$ $\\
				\hline
				\mbox{1564} &&$*$&&$ $&$*$&$ $&$ $&$*$&$*$&$ $&$ $&$ $&$ $\\
				\hline
				\mbox{4030} &&$*$&&$ $&$*$&$ $&$ $&$*$&$*$&$ $&$ $&$ $&$ $\\
				\hline
				\mbox{1565} &&$*$&&$ $&$*$&$ $&$ $&$*$&&$*$&$ $&$ $&$ $\\
				\hline
				\mbox{1566} &&$*$&&$ $&$*$&$ $&$ $&$*$&$ $&$ $&$ $&$ $&$*$ \\
				\hline
				\mbox{1567} &&$*$&&$ $&$*$&$ $&$ $&$*$&$ $&$ $&$ $&$ $&$*$ \\
				\hline
				\mbox{1569} &&$*$&&$ $&$*$&$ $&$ $&$*$&$ $&$ $&$ $&$ $&$*$ \\
				\hline
				\mbox{1571} &$*$&&$ $&$ $&$*$&$ $&$ $&$*$&$ $&$ $&$ $&$ $&$*$ \\
				\hline
				\mbox{4031}&&$*$&$ $&$ $&$*$&$ $&$ $&$*$&$ $&$ $&$ $&$ $&$*$ \\
				\hline
			\end{tabular}
			
		}

	\end{table}

\begin{example}
	\label{namedmodelexm}
	{\rm  
	The context $\mathbb{K}_{z}:=(G_{z}, M_{z}, I_{z})$ in Table \ref{CU} is part of the conceptual universe for the field of interest, a family of musical instruments described in \cite{lpwille}. The set $G_{z}$ of objects  contains some concrete zink and is represented by the numbers in the first column of  Table \ref{CU}. The set $M_{z}$ of properties of the zink are written in the first row of  Table \ref{CU}. Each cell (i, j) with * encodes the information that the zink $g$ in the $i^{th}$ row has the property $m$ in the $j^{th}$ column, that is, $gI_{z} m$. Moreover, for this example, we assume that each empty cell (i, j) encodes the information that the zink $g$ in the $i^{th}$ row lacks the property $m$ in the $j^{th}$ column, that is, $g\cancel{I_{z}} m$.
 
 Let $\{p_{1}, p_{2}, p_{3},\ldots\}$  and $\{P_{1}, P_{2}, P_{3},\ldots\}$ be  enumerations of the sets $\textbf{OV}$ and $\textbf{PV}$ respectively. 
	%The following  $\mathbb{M}$ is  a named model based on the context given in Table \ref{CU}. 
		The context contains 14 objects and so the set of semiconcepts generated by the objects is also finite. Let $\{x_{1}, x_{2}, \ldots x_{14}\}$ be an enumeration of the set of semiconcepts generated by the objects.  Similarly, let $\{y_{1}, y_{2}, \ldots y_{13}\}$ be the enumeration of the set of semiconcepts generated by the properties. Now, consider the model  $\mathbb{M}_{z}:=(\mathbb{K}_{z}, v_{r})$, where  $v_{r}: \textbf{OV}\cup\textbf{PV}\cup \{\top, \bot\}\rightarrow \mathfrak{H}(\mathbb{K}_{z})$ is defined as follows:

		\begin{equation*}
		\begin{aligned}
		v_{r}(p_{i}) &= x_{i},     && i\leq 14 \\
		&= x_{1}, && i\geq 15
		\end{aligned}
		\quad 
		\begin{aligned}
		v_{r}(P_{i}) &= y_{i},     && i\leq 13 \\
		&= y_{1}, && i\geq 14
		\end{aligned}
		\end{equation*}
		\hspace{0.8cm} $v_{r}(\top) = \top_{\underline{\mathfrak{H}}(\mathbb{K}_{z})}$,~     
		$v_{r}(\bot)= \bot_{\underline{\mathfrak{H}}(\mathbb{K}_{z})}$. 
		\vspace{.3cm}
		
		\noindent From the definition, it follows that $\mathbb{M}_{z}$ is a named model.
	}
\end{example}

\noindent Gerader Zink and Stiller Zink are two typical instances of concepts from the zink family.   A zink is a Gerader Zink if and only if it has the property gerade Form, while a zink is a Stiller Zink  if and only if it has the properties  gerade Form and eingedrehtes Mundst{\"u}ck. Let us assume that \\
$y_{1}:=( \{\mbox{gerade Form}\}^{\prime},  \{\mbox{gerade Form}\})=(\{ 1558, 1559, 1560, 1561, 1562\}, \{\mbox{gerade Form}\})$ and \\
$y_{2}:=(\{\mbox{eingedrehtes}~ \mbox{Mundst{\"u}ck}\}^{\prime}, \{\mbox{eingedrehtes}~ \mbox{Mundst{\"u}ck}\})= (\{ 1559, 1560, 1561, 1562\}, \{\mbox{eingedrehtes}~ \mbox{Mundst{\"u}ck}\})$.\\
Then $P_{1}$ represents the concept Gerader Zink in the model $\mathbb{M}_{z}:=(\mathbb{K}_{z}, v_{r})$, while  $P_{2}$ represents the semiconcept generated by eingedrehtes Mundst{\"u}ck. 
%\noindent How do we represent the concept Stiller Zink in the named model $\mathbb{M}_{z}$? For that, first, we prove the following proposition. 

 \noindent Observe that the  concept  $(\{  1559, 1560, 1561, 1562\}, \{\mbox{gerade~ Form},\mbox{eingedrehtes}~ \mbox{Mundst{\"u}ck}\})$  of $\mathbb{K}_{z}$ represents the  concept  Stiller Zink. Now \\
 $(\{  1559, 1560, 1561, 1562\}, \{\mbox{gerade~ Form},\mbox{eingedrehtes}~ \mbox{Mundst{\"u}ck}\})\\
 =(\{ 1558, 1559, 1560,$ $  1561, 1562\}, \{\mbox{gerade~ Form}\})\sqcap (\{  1559, 1560, 1561, 1562\}, \{\mbox{eingedrehtes}~ \mbox{Mundst{\"u}ck}\})\\=v_{r}(P_{1})\sqcap v_{r}(P_{2})=v_{r}(P_{1}\sqcap P_{2})$. \\So $P_{1}\sqcap P_{2}$ represents the concept  Stiller Zink in the model $\mathbb{M}_{z}$.  In fact, we have the following.
\begin{proposition}
	{\rm Let $\mathbb{M}:=(\mathbb{K}, v)$ be a named model based on  a finite context $\mathbb{K}$ and $(A, B)$ be a concept of $\mathbb{K}$. Then there is a formula $\alpha$ in \textbf{PDBL} such that $v(\alpha)=(A, B)$.
	 %a named model $\mathbb{M}$ based on the context $\mathbb{K}$.
	 }
\end{proposition}
\begin{proof}
Let 
%$\mathbb{M}:=(\mathbb{K},v)$ be a named model based on  a 
$\mathbb{K}:=(\{g_{1},\ldots g_{n}\}, \{m_{1}\ldots m_{n}\}, I)$,  and $(A, B):=(\{g_{r_{1}}\ldots, g_{r_{k}}\}, \{m_{r_{1}},\ldots, m_{r_{k}}\} )$ be a concept of $\mathbb{K}$. Then  $(A, B)=(\{g_{r_{1}}\ldots, g_{r_{k}}\}, \{m_{r_{1}},\ldots, m_{r_{k}}\} )=(B^{\prime}, B^{\prime\prime})= \sqcap_{i=1}^{k}(\{m_{r_{i}}\}^{\prime},  \{m_{r_{i}}\})$.  So there are some $P_{i}, i=1,2,\ldots k$, such that $v(P_{i})=(\{m_{r_{i}}\}^{\prime},  \{m_{r_{i}}\})$ in the model $\mathbb{M}$. The required formula is $\alpha:=\sqcap_{i=1}^{k} P_{i}$.
%Let  $\alpha:=\sqcap_{i=1}^{k} P_{i}$, $v(\alpha)=v(\sqcap_{i=1}^{k} P_{i})=(A, B)$.
\end{proof}

\noindent Let us return to the example. 
   $\eta(\mbox{g(Stimmgr{\"o}{\ss }e)})\sqcap \neg \eta(\mbox{a(Stimmgr{\"o}{\ss }e)})=(\{1558, 1561, 1562\}, \{\mbox{g(Stimmgr{\"o}{\ss }e)}\})=\eta(\mbox{g(Stimmgr{\"o}{\ss }e)})$ in $\underline{\mathfrak{H}}(\mathbb{K}_{z})$
 %, which implies in the  conceptual knowledge system  of $\mathbb{K}_{z}$ the term   $t=\eta(\mbox{g(Stimmgr{\"o}{\ss }e)})\sqcap \neg \eta(\mbox{a(Stimmgr{\"o}{\ss }e)})$ 
 encodes the fact that  in $\mathbb{K}_{z}$, if an object has  the property $\mbox{g(Stimmgr{\"o}{\ss }e)}$ then it does not possess the property  $\mbox{a(Stimmgr{\"o}{\ss }e)}$.  
% For the trem $t:=\eta(\mbox{a(Stimmgr{\"o}{\ss }e)})\sqcap  \eta(\mbox{g(Stimmgr{\"o}{\ss }e)})$, we have $\eta(\mbox{a(Stimmgr{\"o}{\ss }e)})\sqcap  \eta(\mbox{g(Stimmgr{\"o}{\ss }e)})=\bot$. So the trem $\eta(\mbox{a(Stimmgr{\"o}{\ss }e)})\sqcap  \eta(\mbox{g(Stimmgr{\"o}{\ss }e)})$ encod tha fact : there is no object in $\mathbb{K}_{z}$ that have the properties a(Stimmgr{\"o}{\ss }e) and g(Stimmgr{\"o}{\ss }e).
%the  identities $\eta(\mbox{a(Stimmgr{\"o}{\ss }e)})\sqcap  \eta(\mbox{g(Stimmgr{\"o}{\ss }e)})=\bot$ and $\eta(\mbox{g(Stimmgr{\"o}{\ss }e)})\sqcap \neg \eta(\mbox{a(Stimmgr{\"o}{\ss }e)})=\eta(\mbox{g(Stimmgr{\"o}{\ss }e)})$  describe the fact that in $\mathbb{K}_{z}$, if an object has  the property $\mbox{g(Stimmgr{\"o}{\ss }e)}$ then it does not possess the property  $\mbox{a(Stimmgr{\"o}{\ss }e)}$. Using the infromation, we can  deduce the empty cell $(1, 11)$ from the crossed cell $(1, 12)$.
From this fact, we also derive that there is no object in $\mathbb{K}_{z}$ that has the properties a(Stimmgr{\"o}{\ss }e) and g(Stimmgr{\"o}{\ss }e) together, that is, $\eta(\mbox{a(Stimmgr{\"o}{\ss }e)})\sqcap  \eta(\mbox{g(Stimmgr{\"o}{\ss }e)})=\bot=(\emptyset , M)$.
%By abusing notation,  we write $t=x$  if $t_{\mathbb{K}_{z}}=x$ for $t\in T(X)$ and $x\in \mathfrak{H}(\mathbb{K}_{z})$. 
In fact, using \textbf{PDBL}, we can show that  $\eta(\mbox{a(Stimmgr{\"o}{\ss }e)})\sqcap  \eta(\mbox{g(Stimmgr{\"o}{\ss }e)})=\bot$ is  deducible  from $\eta(\mbox{g(Stimmgr{\"o}{\ss }e)})\sqcap \neg \eta(\mbox{a(Stimmgr{\"o}{\ss }e)})=\eta(\mbox{g(Stimmgr{\"o}{\ss }e)})$. 
%Now let us show $\eta(\mbox{a(Stimmgr{\"o}{\ss }e)})\sqcap  \eta(\mbox{g(Stimmgr{\"o}{\ss }e)})=\bot$ is  deducible  from $\eta(\mbox{g(Stimmgr{\"o}{\ss }e)})\sqcap \neg \eta(\mbox{a(Stimmgr{\"o}{\ss }e)})=\eta(\mbox{g(Stimmgr{\"o}{\ss }e)})$. 
 For that, we prove the following theorem.

\begin{theorem}
	\label{infprof}
	{\rm The following rules are derivable in \textbf{PDBL}.
		
		$\begin{array}{ll}
		\infer[(R10)]{  \bot\vdash\alpha\sqcap\beta}{ \alpha\sqcap \neg \beta\vdash\alpha } &~~~~~~
		\infer[(R11)]{\alpha\sqcap\beta\vdash\bot}{\alpha\vdash\alpha\sqcap\neg\beta}
		\end{array}$}
\end{theorem}
\begin{proof}
	The proof of (R11) is similar to that  of (R10) and we only prove (R10).\\
Note that from Proposition \ref{pro1.5} and Theorem \ref{algcomdbl}, one obtains in  \textbf{PDBL},\\	
\hspace*{5.5cm} $\alpha\sqcap\bot\vdash \bot$ and $\bot\vdash\alpha\sqcap\bot$ \hfill{(*)}
%are provable in \textbf{PDBL}$\ldots$ (*) 

\noindent Then we have:	
	\begin{center}
		$\infer{\bot\vdash\alpha\sqcap\beta}{\infer{\mbox{(*)}~~\bot\vdash\alpha\sqcap\bot~~\alpha\sqcap\bot\vdash\alpha\sqcap\beta~(R4)}{	\infer{(R1)^{\prime}~~\alpha\sqcap\bot\vdash \alpha\sqcap (\neg\beta\sqcap\beta)~~\alpha\sqcap(\neg\beta\sqcap\beta)\vdash\alpha\sqcap\beta}{\infer{\mbox{Theorem \ref{thempdbl}(7a)}~\bot\vdash\neg\beta\sqcap\beta~~\alpha\sqcap(\neg\beta\sqcap\beta)\vdash\alpha\sqcap\beta~~(R4)}{\mbox{Theorem \ref{thempdbl}(2a)}~~\alpha\sqcap (\neg\beta\sqcap\beta)\vdash(\alpha\sqcap\neg\beta)\sqcap\beta~~\infer{(\alpha\sqcap\neg\beta)\sqcap\beta\vdash\alpha\sqcap\beta~~(R1)}{\alpha\sqcap\neg\beta\vdash\alpha}}}}}$
	\end{center}
\end{proof}
\noindent Let $y_{3}:=(\{\mbox{a(Stimmgr{\"o}{\ss }e)}\}^{\prime}, \{\mbox{a(Stimmgr{\"o}{\ss }e)}\})$ and $y_{4}:=(\{\mbox{g(Stimmgr{\"o}{\ss }e)}\}^{\prime}, \{\mbox{g(Stimmgr{\"o}{\ss }e)}\})$ in Example \ref{namedmodelexm}.  As  $\eta(\mbox{g(Stimmgr{\"o}{\ss }e)})\sqcap \neg \eta(\mbox{a(Stimmgr{\"o}{\ss }e)})=\eta(\mbox{g(Stimmgr{\"o}{\ss }e)})$, the sequents $P_{4}\sqcap \neg P_{3}\vdash P_{4}$ and $P_{4}\vdash P_{4}\sqcap\neg P_{3}$ are satisfied in the model $\mathbb{M}_{z}$. By soundness and Theorem \ref{infprof}, $\bot\vdash P_{4}\sqcap P_{3}$ and $P_{4}\sqcap P_{3}\vdash \bot$ are also satisfied in  $\mathbb{M}_{z}$, which implies that $v_{r}(\bot)\sqsubseteq v_{r}(P_{4}\sqcap P_{3})$ and $v_{r}(P_{4}\sqcap P_{3})\sqsubseteq v_{r}(\bot)$. So $v_{r}(\bot)=v_{r}(P_{4}\sqcap P_{3})=v_{r}(P_{4})\sqcap v_{r}(P_{3})$, whence $\eta(\mbox{a(Stimmgr{\"o}{\ss }e)})\sqcap  \eta(\mbox{g(Stimmgr{\"o}{\ss }e)})=\bot$.

\vskip 3pt
We end this part by demonstrating  that under the interpretation given in terms of named models, we can characterize the clarified context \cite{ganter2012formal}.  Let us recall the definition of  a clarified context.

\begin{definition}
	{\rm \cite{ganter2012formal} A context $\mathbb{K}:= (G, M, I)$ is called a {\it clarified context} if $I$ satisfies the following.
		\begin{enumerate}
			\item For all $g_{1}, g_{2}\in G$, $\{g_{1}\}^{\prime}=\{g_{2}\}^{\prime} \Longrightarrow g_{1}=g_{2}$.
			\item For all $m_{1}, m_{2}\in M$, $\{m_{1}\}^{\prime}=\{m_{2}\}^{\prime} \Longrightarrow m_{1}=m_{2}$.
	\end{enumerate}}
\end{definition}

\begin{theorem}
	{\rm The  following inference rules are valid in the class $\mathcal{K}$ for any $p, q\in \textbf{OV}$ and $P, Q\in \textbf{PV}$ if and only if  $\mathcal{K}$ is the class of all clarified contexts.
	\vspace{.5cm}	
		
		$\begin{array}{ll}
		\infer[\ldots (1)]{p\vdash q}{ p\sqcup p\vdash q\sqcup q & q\sqcup q\vdash p\sqcup p}&
		\infer[\ldots (2)]{q\vdash p}{ p\sqcup p\vdash q\sqcup q & q\sqcup q\vdash p\sqcup p}\\
		\infer[\ldots (3)]{P\vdash Q}{ P\sqcap P\vdash Q\sqcap Q & Q\sqcap Q\vdash P\sqcap P } &
		\infer[\ldots (4)]{Q\vdash P}{ P\sqcap P\vdash Q\sqcap Q & Q\sqcap Q\vdash P\sqcap P }
		\end{array}$
		
	}
\end{theorem}
\begin{proof}
	Let $\mathcal{K}$ be the class of all clarified contexts and
	$ p\sqcup p\vdash q\sqcup q$,  $ q\sqcup q\vdash p\sqcup p$ be valid in  the class $\mathcal{K}$. To show $p\vdash q$  and $q\vdash p$ are valid in the class $\mathcal{K}$,
	let $\mathbb{M}:=(\mathbb{K}, v)$ be a named model based on  $\mathbb{K}:=(G,M, R)\in\mathcal{K}$. Then $v(p)=(\{g\}, \{g\}^{\prime})$ and $v(q)=(\{g_{1}\}, \{g_{1}\}^{\prime} )$ for some $g,g_{1}\in G$, and $v(P)=(\{m\}^{\prime}, \{m\})$,  $v(Q)=(\{m_{1}\}^{\prime}, \{m_{1}\})$ for some $m, m_{1}\in M$. By  Corollary \ref{PDBL-valution}, $v(p\sqcup p)=v(p)\sqcup v(p)=(\{g\}, \{g\}^{\prime})\sqcup (\{g\}, \{g\}^{\prime} )=( \{g\}^{\prime\prime},  \{g\}^{\prime})$. Similarly $v(q\sqcup q)=( \{g_{1}\}^{\prime\prime},  \{g_{1}\}^{\prime})$. Then $( \{g\}^{\prime\prime},  \{g\}^{\prime})=v(p\sqcup p)\sqsubseteq v(q\sqcup q)=( \{g_{1}\}^{\prime\prime},  \{g_{1}\}^{\prime})$  and $( \{g_{1}\}^{\prime\prime},  \{g_{1}\}^{\prime})=v(q\sqcup q)\sqsubseteq v(p\sqcup p)=( \{g\}^{\prime\prime},  \{g\}^{\prime})$, which implies that $( \{g\}^{\prime\prime},  \{g\}^{\prime})=( \{g_{1}\}^{\prime\prime},  \{g_{1}\}^{\prime})$. So $\{g\}^{\prime}=\{g_{1}\}^{\prime}$, which implies that $g=g_{1}$, as $\mathbb{K}$ is a clarified context. So  $v(p)=v(q)$, which implies that $v(p)\sqsubseteq v(q)$ and $v(q)\sqsubseteq v(p)$. So $p\vdash q$ and $q\vdash p$ are satisfied in $\mathbb{M}$, which implies that  $p\vdash q$  and $q\vdash p$ are true in $\mathbb{K}$. So  $p\vdash q$  and $q\vdash p$ are valid in the class $\mathcal{K}$, which implies that (1) and (2) are valid in $\mathcal{K}$. \\
	Similar to the above proof we can show that (3) and (4) are valid in the class $\mathcal{K}$.\\
	Conversely, let (1)-(4) be valid in the class $\mathcal{K}$. To show $\mathcal{K}$ is the class of all clarified contexts, let $\mathbb{K}:=(G,M,R)\in \mathcal{K}$, and $g, g_{1}\in G$ , $m,m_{1}\in M$ such that $\{g\}^{\prime}=\{g_{1}\}^{\prime}$ and $\{m\}^{\prime}=\{m_{1}\}^{\prime}$. Let $\mathbb{M}:=(\mathbb{K}, v)$ be  a named model such that $v(p):=(\{g\}, \{g\}^{\prime}), v(q):=(\{g_{1}\}, \{g_{1}\}^{\prime})$ and $v(P):=(\{m\}^{\prime}, \{m\}), v(Q):=(\{m_{1}\}^{\prime}, \{m_{1}\}) $. Then $ p\sqcup p\dashv\vdash q\sqcup q$ and $P\sqcap P\dashv\vdash Q\sqcap Q$  are satisfied in $\mathbb{M}$, as $v(p\sqcup p)=v(q\sqcup q)$ and $v(P\sqcap P)=v(Q\sqcap Q)$. So $p\dashv\vdash q$ and $P\dashv\vdash Q$ are satisfied in $\mathbb{M}$, which implies that  $v(p)=v(q)$ and $v(P)=v(Q)$. So $g=g_{1}$ and $m=m_{1}$, which implies that $\mathbb{K}$ is a clarified context. Hence  $\mathcal{K}$ is the class of all clarified contexts.
\end{proof}

\section{  MPDBL and rough sets}
\label{rough set approx}
In this section,  we explain our approach to FCA from the perspective of rough set theory. For that, recall the approximation spaces $(G, E_{1})$ and $(M, E_{2})$  for a  context $\mathbb{K}:=(G,M,I)$ given in Section \ref{Appropefca}. The relations $E_1,E_2$ are defined as:  $g_{1}E_{1} g_{2}$ if and only if $I(g_{1})=I(g_{2})$ for $g_{1},g_{2}\in G$, and
 $m_{1}E_{2}m_{2}$ if and only if  $I^{-1}(m_{1})= I^{-1}(m_{2})$ for $m_{1}, m_{2}\in M$. We consider the  Kripke context   $\mathbb{KC}_{SD}:=((G, E_{1}), (M, E_{2}), I)$  based on $\mathbb{K}$. $\mathbb{KC}_{SD}$ is an example of a reflexive, symmetric and transitive Kripke context, and
%$(G, E_{1})$ and $(M, E_{2})$  described in Section \ref{intro}.
 %Taking a cue from  rough set theory,
  %$\mathbb{KC}_{SD}$
   may be understood to represent a  basic classification skill of an  agent about objects in $G$ and properties 
 in $M$, with respect to information given in the context $\mathbb{K}$. 
 %Taking into account the two component approximation spaces in a Kripke context, 
 %motivation from  rough set theory, we n
 We now propose the notion of definability of a semiconcept $(A,B)$, $A \subseteq G, B \subseteq M$.
\begin{definition}
	\label{pmapro}
	{\rm  A {\it definable semiconcept} of $\mathbb{KC}_{SD}:=((G, E_{1}), (M, E_{2}), I)$ is a semiconcept $(A,B)$ **of $\mathbb{K}$ where $A$ and $ B$ are categories in the Pawlakian approximation spaces $(G, E_{1})$ and $(M, E_{2})$ respectively. %In particular, if $(A,B)$ is a concept, it will be termed a {\it definable concept} of $\mathbb{KC}_{SD}$.
	}
\end{definition}

\noindent As noted earlier after Theorem \ref{cs}, a semiconcept  is either a left semiconcept of the form $(A, A^{\prime} )$ or a right semiconcept of the form $(B^{\prime},B)$. In Corollary \ref{propertyofapprox} below, we shall see that the former kind  is definable if and only if  the first component is a category in $(G, E_{1})$, and the  latter one is definable if and only if the second component is a category in $(M, E_{2})$. 
%*Later, we will see that if one of the component of a semiconcept is a category then the semiconcept is definable.*
\begin{proposition}
	\label{definblconcept}
	{\rm  If a semiconcept $(A, B)$ is a concept of $\mathbb{K}$ then it is a definable semiconcept of $\mathbb{KC}_{SD}$.}
\end{proposition}
\begin{proof}
	Let $(A, B)$ be a concept of $\mathbb{K}$. Then $A^{\prime\prime}=A$ and $B^{\prime\prime}=B$.  Let $g\in A$ and $g_{1}\in E_{1}(g)$. By definition of $E_1$, this  means $I(g)=I(g_{1})$.  As $g \in A$, $g I m$ for any $m \in A^\prime$. So $g_{1}Im$  for all such $m$, and thereby, $g_1 \in A^{\prime\prime}=A$.
	%Therefore, for all $m\in M$, $gIm$ if and only if $g_{1}Im$, which implies that $g_{1}\in A$, as $A^{\prime\prime}=A$. 
	Thus $E_{1}(g)\subseteq A$, which gives $g\in \underline{A}$. So $\underline{A}=A$, by Proposition \ref{pra}(v), and by Proposition \ref{pra}(vii), $A$ is a category in $(G, E_{1})$. Similarly, we can show that $B$ is a category in $(M, E_{2})$. Therefore, $(A, B)$ is a definable semiconcept.
\end{proof}

It is not always the case that a semiconcept is a  definable semiconcept. 
Let us use an example to demonstrate this fact. 
\begin{example}
	\label{semiconceptapprox} 
{\rm $(G,M,I)$ is a modified subcontext of a context provided by Wille \cite{ganter2012formal}, where 
$G:=\{Leech, Bream, Frog,$ $ Dog, Cat\}$ and $M:=\{a,b,c,g\}$, and  $a$:=requires water to survive, $b$:=lives in water, $c$:=lives on land, and $g$:=can move. Table  \ref{context-1} provides $I$, where * as an entry corresponding to object $x$ and property $y$ indicates that $xIy$ holds. 

%Each cell(i, j) with cross encode the information that the zink in the $i^{th}$ row has the property in the $j^{th}$ column. 

%Let us illustrate these notions by an example. The following context $(G,M,I)$ is a subcontext of a context given by Wille \cite{ganter2012formal} with some modifications. $G:=\{Leech, Bream,Frog,Dog, Cat\}$ and $M:=\{a,b,c,g\}$, where a:= needs water to live, b:= lives in water, c:= lives on land, g:=can move around.
%$I$ is given by Table \ref{context-1}, where * as an entry corresponding to object $x$ and property $y$ means $xIy$ holds.

% **Give caption of Table and label the table.
\begin{table}[ht]
	\centering
	
	\caption {Context $\mathbb{K}$} \label{context-1}
	\scalebox{.70}{\begin{tabular}{ | l | l | l | l | l |}
			\hline	
			& $a$ & $b$& $c$ & $g$\\ \hline
			Leech & * &*  &  & * \\ \hline
			Bream & * & * &  & * \\ \hline
			Frog & * & * & * & * \\ \hline
			Dog & * &  & * & * \\ \hline
			Cat & * &  &  * & *  \\ \hline
	\end{tabular}}

\end{table}
\noindent The induced Kripke context is $\mathbb{K}^{L}_{SD}:=((G, \{\{Leech, Bream\},\{Frog\}, \{Dog, Cat\}\}), (M, \{a,g\}, \{b\},\{c\}\}), I)$. $\mathbb{K}^{L}_{SD}$ represents that an agent cannot distinguish the properties a and g, while Leech and Bream as well as Dog and Cat are indistinguishable for an agent based on the information given in  $\mathbb{K}$. 
 %$(\{Leech, Bream, Frog\}, \{a, b, g\})$ is a concept of $\mathbb{K}$ such that $\{Leech, Bream,$ $ Frog\}$ and $\{a, b, g\}$ are categories in $(G, \{\{Leech, Bream\},\{Frog\}, \{Dog, Cat\}\})$, and  $(M, \{\{a,\mbox{g}\}, \{b\},\{c\}\})$ respectively. So  $(\{Leech, Bream, Frog\}, \{a, b, g\})$ is an example of a definable concept of $\mathbb{KC}^{L}_{SD}$. On the other hand,
$(\{Leech, $ $Bream , Dog\}, \{a, g\})$ is a non-definable semiconcept of $\mathbb{KC}^{L}_{SD}$.
%We know that $A$ and $B$ may not be categories in $(G, E_{1})$ and $(M, E_{2})$ respectively.
}
\end{example}
The question then  is:  can we approximate a semiconcept of $\mathbb{K}$ by  definable semiconcepts of $\mathbb{KC}_{SD}$? For that purpose, the   {\it lower and upper approximations of a semiconcept} are  defined. 
We split the set $\mathfrak{H}(\mathbb{K})$ as 
$\mathfrak{H}(\mathbb{K}):=(\mathfrak{H}(\mathbb{K})_{\sqcap}\setminus \mathcal{B}(\mathbb{K}))\cup (\mathfrak{H}(\mathbb{K})_{\sqcup}\setminus \mathcal{B}(\mathbb{K}))\cup \mathcal{B}(\mathbb{K})$.

\begin{definition}
	\label{approx meet semi}
	{\rm
	
	\noindent 	\bla
	\item Let $(A, B)\in \mathcal{B}(\mathbb{K})$. The {\it lower approximation} and  {\it upper approximation} of $(A, B)$ are $(A,B)$ itself.
		\item  Let $(A, B)\in \mathfrak{H}(\mathbb{K})_{\sqcap}\setminus \mathcal{B}(\mathbb{K})$.  The {\it lower approximation} and  {\it upper approximation} of $(A, B)$ are defined as $\underline{(A, B)}:=(\underline{A}_{E_{1}}, (\underline{A}_{E_{1}})^{\prime})$ and $\overline{(A, B)}:=(\overline{A}^{E_{1}}, (\overline{A}^{E_{1}})^{\prime})$ respectively.

	\item Let $(A, B)\in \mathfrak{H}(\mathbb{K})_{\sqcup}\setminus \mathcal{B}(\mathbb{K})$. The {\it lower approximation} and {\it upper approximation} of $(A, B)$ are defined as  $\underline{(A, B)}:=((\overline{B}^{E_{2}})^{\prime}, \overline{B}^{E_{2}})$ and  $\overline{(A, B)}:=((\underline{B}_{E_{2}})^{\prime}, \underline{B}_{E_{2}})$ respectively.
	 
\ela}
\end{definition}
\begin{observation}
	{\rm  Recall the operators $f_{R}$ and $f_{S}$ defined in Section \ref{dbawo}.
	\vskip 3pt 
		\begin{enumerate}
			\item For $(A, B)\in \mathfrak{H}(\mathbb{K})_{\sqcap}\setminus \mathcal{B}(\mathbb{K}) $, $\underline{(A,B)}=f_{E_{1}}((A, B))$, $\overline{(A, B)}=f^{\delta}_{E_{1}}((A, B))$.
		\item For $(A, B)\in \mathfrak{H}(\mathbb{K})_{\sqcup}\setminus \mathcal{B}(\mathbb{K})$, $\underline{(A, B)}=f_{E_{2}}((A, B))$, 
		$\overline{(A, B)}=f^{\delta}_{E_{2}}((A, B))$. 
	\end{enumerate}}
\end{observation}

\begin{example}
	\label{exampleapprox}
{\rm  In Example \ref{semiconceptapprox} considered   above, the lower and upper approximations of the semiconcept  $(\{Leech,$ $ Bream , Dog\}, \{a, g\})$ are $(\{ Leech, Bream\}, \{a,b,g\})$  and  $(\{Leech, Bream , Cat, Dog\}, \{a, g\})$ respectively. }
 \end{example}

% \vskip 3pt
 %For  $\mathbb{KC}_{SD}$ based on a context $\mathbb{K}:=(G, M, I)$,  the following hold.
 
 \begin{proposition}
 	\label{defsemi}
 	{\rm  Let $A\subseteq G$ and $B\subseteq M$. %the following hold.
 		\begin{enumerate}
 			\item $\underline{(\underline{A}_{E_{1}})^{\prime}}_{E_{2}}=(\underline{A}_{E_{1}})^{\prime}$ and $\overline{(\overline{A}^{E_{1}})^{\prime}}^{E_{2}}=(\overline{A}^{E_{1}})^{\prime}$.
 			\item $\underline{(\underline{B}_{E_{2}})^{\prime}}_{E_{1}}=(\underline{B}_{E_{2}})^{\prime}$ and $\overline{(\overline{B}^{E_{2}})^{\prime}}^{E_{1}}=(\overline{B}^{E_{2}})^{\prime}$.
 	\end{enumerate}} 
 \end{proposition}
 \begin{proof}
 	1. By Proposition \ref{pra}(v),   $\underline{(\underline{A}_{E_{1}})^{\prime}}_{E_{2}}\subseteq (\underline{A}_{E_{1}})^{\prime}$. Let $m\in (\underline{A}_{E_{1}})^{\prime}$. Then  $gIm$ for all $g\in\underline{A}_{E_{1}} $, which implies that  $gIm$ for all $g\in G$ such that $E_{1}(g)\subseteq A$. Let $m_{1}\in E_{2}(m)$, which implies that $gIm$ if and only if $gIm_{1}$ for all $g\in G$. So $gIm_{1}$ for all $g\in G$, $E_{1}(g)\subseteq A$, which implies that $gIm_{1}$ for all $g\in \underline{A}_{E_{1}}$. So $m_{1}\in (\underline{A}_{E_{1}})^{\prime}$, which implies that $E_{2}(m)\subseteq (\underline{A}_{E_{1}})^{\prime}$. So $m\in \underline{(\underline{A}_{E_{1}})^{\prime}}_{E_{2}} $, whence $(\underline{A}_{E_{1}})^{\prime}\subseteq \underline{(\underline{A}_{E_{1}})^{\prime}}_{E_{2}}$. So $\underline{(\underline{A}_{E_{1}})^{\prime}}_{E_{2}}= (\underline{A}_{E_{1}})^{\prime}$. By Proposition \ref{pra}(v),  $(\overline{A}^{E_{1}})^{\prime}\subseteq \overline{(\overline{A}^{E_{1}})^{\prime}}^{E_{2}}$. Let $m\in \overline{(\overline{A}^{E_{1}})^{\prime}}^{E_{2}}$. Then $E_{2}(m)\cap (\overline{A}^{E_{1}})^{\prime}\neq\emptyset$, which implies that there exists $m_{1}\in (\overline{A}^{E_{1}})^{\prime}$ such that $m_{1}\in E_{2}(m)$. Now  for all $g\in \overline{A}^{E_{1}}$,  $gIm_{1}$,  which implies that for all $g\in \overline{A}^{E_{1}}$, $gIm$ , as $m_{1}\in E_{2}(m)$. So $m\in (\overline{A}^{E_{1}})^{\prime} $, whence $\overline{(\overline{A}^{E_{1}})^{\prime}}^{E_{2}}\subseteq (\overline{A}^{E_{1}})^{\prime}$. So $\overline{(\overline{A}^{E_{1}})^{\prime}}^{E_{2}}=(\overline{A}^{E_{1}})^{\prime}$.
 	\vskip 3pt
 	\noindent 2. The proof is similar to that of 1.
 \end{proof}

\begin{corollary}
	\label{propertyofapprox}
	{\rm 
	
	\noindent	
		\begin{enumerate} 
		\item The lower and upper approximations of semiconcepts of $\mathbb{K}$ are definable semiconcepts of $\mathbb{KC}_{SD}$.
		\item If $(A, B)$ is a definable semiconcept of $\mathbb{KC}_{SD}$ then $\underline{(A, B)}=(A, B)$ and $\overline{(A, B)}=(A, B)$.
		\item For $(A, B)\in \mathfrak{H}(\mathbb{K})$, $\underline{(\underline{(A, B)})}=\underline{(A, B)}$ and $\overline{(\overline{(A, B)})}=\overline{(A, B)}$.
		\item A left semiconcept $(A, B)$ of $\mathbb{K}$ is definable if and only if $A$ is a category in $(G, E_{1})$.	
		
		\item A right semiconcept $(A, B)$ of $\mathbb{K}$ is definable if and only if $B$ is a category in  $(M, E_{2})$.
			
		\end{enumerate}}
\end{corollary}
\begin{proof}
	1. The proof follows  from  Propositions \ref{definblconcept},  \ref{defsemi} and  \ref{pra}(vii and viii).
	
	\noindent 2. Let $(A, B)$ be a definable semiconcept. If $(A, B) \in \mathcal{B}(\mathbb{K})$, the proof follows from  Definition \ref{approx meet semi}(a).  
	%and $A, B$ be categories in $(G, E_{1})$ and $(M, E_{2})$ respectively. Now if $()$
	
	\noindent Let $(A, B)\in\mathfrak{H}(\mathbb{K})_{\sqcap}\setminus \mathcal{B}(\mathbb{K})$. Then $\underline{(A, B)}=(\underline{A}_{E_{1}}, (\underline{A}_{E_{1}})^{\prime})=(A, A^{\prime})$, as $A$ is a category in $(G, E_{1})$. So  $\underline{(A, B)}= (A, A^{\prime})=(A, B)$.
	Similarly, one gets the result for the  case when $(A, B)\in\mathfrak{H}(\mathbb{K})_{\sqcup}\setminus \mathcal{B}(\mathbb{K})$.\\
	%\noindent Case II: Let $(A, B)\in\mathfrak{H}(\mathbb{K})_{\sqcup}\setminus \mathcal{B}(\mathbb{K})$. Then $\underline{(A, B)}=((\overline{B}^{E_{2}})^{\prime}, \overline{B}^{E_{2}})=(B^{\prime}, B)$,  as $B$ is categories in $(G, E_{2})$ and by Proposition \ref{defsemi}(2). So  $\underline{(A, B)}= (B^{\prime}, B)=(A, B)$.
%	\noindent Case III: Let $(A, B)\in \mathcal{B}(\mathbb{K})$. Then $\underline{(A, B)}=((\underline{A}_{E_{1}})^{\prime\prime}\cap(\overline{B}^{E_{2}})^{\prime}, ((\underline{A}_{E_{1}})^{\prime\prime}\cap(\overline{B}^{E_{2}})^{\prime})^{\prime} )=(A^{\prime\prime}\cap B^{\prime}, (A^{\prime\prime}\cap B^{\prime})^{\prime})=(A, B)$.
		\noindent Similar to the above proof, we can show that $\overline{(A,B)}=(A, B)$.
		
	\noindent	3. The proof follows from  1 and 2.
	
	\noindent 4. Let $(A, B)$ be a left semiconcept, so that $(A, B)=(A, A^{\prime})$. 
	%Then $(A, B)=(A, B)\sqcap (A, B)=(A, A^{\prime})$. 
	If $A$ is a category then by Proposition \ref{defsemi}(1) it follows that $A^{\prime}$ is a category. So $(A, B)$ is definable semiconcept.  The other direction follows from the definition of definability of semiconcepts.
	
	\noindent 5. The proof   is similar to that  of 4.
	\end{proof}
Now, we prove  sequents in $\textbf{MPDBL5}$ that will yield fundamental properties of the approximations of semiconcepts, as will be seen in Proposition \ref{property of approx}  below.
\begin{proposition}
	\label{drivablity of appx sqeuent}
	{\rm  Let $ \gamma\in \mathfrak{F}_{1}$. %, and define $\triangle\gamma:= (\square\gamma\sqcup\square\gamma)\sqcap \blacklozenge\gamma$
		%and $\blacktriangle\gamma:= \lozenge\gamma\sqcup(\blacksquare\gamma\sqcap\blacksquare\gamma)$.
		 Then the following sequents are derivable in \textbf{MPDBL5}.
		
		\vspace{.3cm}
		\noindent$\begin{array}{ll}
		(1)~\blacklozenge\gamma\vdash\gamma\sqcup\gamma & (2)~\gamma\sqcap\gamma\vdash\lozenge\gamma \\
		%(1b)~ \triangle\gamma\vdash\gamma\sqcup\gamma.
		%& (2b)~\gamma\sqcap\gamma\vdash\blacktriangle\gamma.
		\end{array}$}

\end{proposition}
\begin{proof}
	We give the  proof for (1). The proof of (2) is similar. 

\noindent (1) %We first use axiom 18b and then use (R5). The third step is obtained using axiom (11b) and  (R4).
	\begin{center}
		$\infer{\blacklozenge\gamma\vdash\gamma\sqcup\gamma}{\infer{\lrcorner\blacksquare\lrcorner\gamma\vdash\gamma\sqcup\gamma~~(R4)}{\infer{(R3)^{\prime}~~\lrcorner\blacksquare\lrcorner\gamma\vdash\lrcorner\lrcorner\gamma ~~ \lrcorner\lrcorner \gamma\vdash\gamma\sqcup\gamma~~(\mbox{Theorem \ref{thempdbl}(10b))}}{\lrcorner\gamma\vdash\blacksquare\lrcorner \gamma~~(19b)}}}$
	\end{center}
	%In the following derivation, we use the sequent derived in 1(a),  axioms 18a, 2a and rules (R4), (R7), (R8), and also (3a) and (1a) of Theorem 9.
%(1b)\begin{center}		$\infer{(\square\gamma\sqcup\square\gamma)\sqcap \blacklozenge\gamma\vdash\gamma\sqcup\gamma~(R4)}{\infer{(\mbox{Theorem \ref{thempdbl}}(3a))~(\square\gamma\sqcup\square\gamma)\sqcap \blacklozenge\gamma\vdash (\square\gamma\sqcup\square\gamma)\sqcap (\blacklozenge\gamma\sqcap\blacklozenge\gamma)~~(\square\gamma\sqcup\square\gamma)\sqcap (\blacklozenge\gamma\sqcap\blacklozenge\gamma)\vdash (\gamma\sqcup\gamma)~(R4)}{\infer{(R6)~(\square\gamma\sqcup\square\gamma)\sqcap (\blacklozenge\gamma\sqcap\blacklozenge\gamma)\vdash (\gamma\sqcup\gamma)\sqcap (\gamma\sqcup\gamma)~~ (\gamma\sqcup\gamma)\sqcap (\gamma\sqcup\gamma)\vdash (\gamma\sqcup\gamma)~(2a) }{\infer{\square\gamma\sqcup\square\gamma\vdash\gamma\sqcup\gamma~~\blacklozenge\gamma\sqcap\blacklozenge\gamma\vdash(\gamma\sqcup\gamma)~(R4) }{\infer{(R7)~\square\gamma\sqcup\square\gamma\vdash\gamma\sqcup\gamma}{(19a)~\square\gamma\vdash\gamma~~ \square\gamma\vdash\gamma} &	\infer{\blacklozenge\gamma\sqcap\blacklozenge\gamma\vdash(\gamma\sqcup\gamma)\sqcap (\gamma\sqcup\gamma)~(R6)}{\blacklozenge\gamma\vdash\gamma\sqcup\gamma~~\blacklozenge\gamma\vdash\gamma\sqcup\gamma~\mbox{From above}} & (\gamma\sqcup\gamma)\sqcap (\gamma\sqcup\gamma)\vdash \gamma\sqcup\gamma~(2a)}}}}$\end{center}
\end{proof}

\begin{proposition}
	\label{property of approx} {\rm  For any $(A, B)\in \mathfrak{H}(\mathbb{K})$, $\underline{(A, B)}\sqsubseteq (A, B)$ and $(A, B)\sqsubseteq \overline{(A, B)}$.
}
\end{proposition}
\begin{proof}
	 Let $(A, B)\in \mathfrak{H}(\mathbb{K})$. If $(A, B) \in \mathcal{B}(\mathbb{K})$, the results are trivially true. 
	
	\noindent Case I: Let $(A, B)\in\mathfrak{H}(\mathbb{K})_{\sqcap}\setminus \mathcal{B}(\mathbb{K})$. We consider a model $\mathbb{M}:=(\mathbb{KC}_{SD}, v)$ and $p\in \textbf{OV}$ such that $v(p)=(A, B)$. As $\square p\vdash p$ is a valid sequent of $\textbf{MPDBL5}$, $v(\square p)\sqsubseteq v(p)$,  by Proposition \ref{satismdbl1}. Using Proposition \ref{modal valuation} and Definition \ref{approx meet semi}(b), $v(\Box p)= (\underline{A}_{E_{1}},(\underline{A}_{E_{1}})^{\prime})= \underline{v(p)}$. So we get $\underline{v(p)}\sqsubseteq v(p)$, which implies that $\underline{(A, B)}\sqsubseteq (A, B)$.
	\vskip 3pt 
	\noindent Case II: Let $(A, B)\in\mathfrak{H}(\mathbb{K})_{\sqcup}\setminus \mathcal{B}(\mathbb{K})$. We consider a model $\mathbb{M}:=(\mathbb{KC}_{SD}, v)$ and $P\in \textbf{PV}$ such that $v(P)=(A, B)$. By Proposition \ref{drivablity of appx sqeuent}(1), $\blacklozenge P\vdash P\sqcup P$.  $v(\blacklozenge P)\sqsubseteq v(P\sqcup P)$, by Proposition \ref{satismdbl1}. Using Proposition \ref{modal valuation} and Definition \ref{approx meet semi}(c), 
	$v(\blacklozenge P)= ((\overline{B}^{E_{2}})^{\prime}, \overline{B}^{E_{2}})=\underline{v(P)}$. Hence  $\underline{v(P)}\sqsubseteq v(P\sqcup P)=v(P)\sqcup v(P)= v(P)$, as $v(P) \in \mathfrak{H}(\mathbb{K})_{\sqcup}$ which implies that $\underline{(A, B)}\sqsubseteq (A, B)$.
	\vskip 3pt 
	\vskip 3pt
	
\noindent 	Similar to the above derivations,  using the valid sequents $P\vdash\blacksquare P$ and  $p\sqcap p\vdash\lozenge p$  of $\textbf{MPDBL5}$ (the last from Proposition \ref{drivablity of appx sqeuent}(2), we get the property    $(A, B)\sqsubseteq \overline{(A, B)}$ in $\mathfrak{H}(\mathbb{K})$.
	% Then $\underline{(A, B)}=((\underline{A}_{E_{1}})^{\prime\prime}\cap(\overline{B}^{E_{2}})^{\prime}, ((\underline{A}_{E_{1}})^{\prime\prime}\cap(\overline{B}^{E_{2}})^{\prime})^{\prime} )=(A^{\prime\prime}\cap B^{\prime}, (A^{\prime\prime}\cap B^{\prime})^{\prime})=(A, B)$.
	\end{proof}

	\begin{observation}
	\label{upperlowrapprox}
	{\rm From the proof of Proposition \ref{property of approx}, we have the  following.
		
		\begin{enumerate}
		
		\item For a formula $\alpha$, if $v(\alpha)$ is a left semiconcept of $\mathbb{K}$ then $v(\square\alpha)$ is its lower approximation and $v(\lozenge\alpha)$ is its upper approximation. Moreover, the valid sequents $\square\alpha\vdash\alpha$ and $\alpha\sqcap\alpha\vdash\lozenge\alpha$ respectively represent the properties that the lower approximation  
		%$v(\square\alpha)$ 
		of a left semiconcept 
		%$v(\alpha)$ 
		lies ``below'' the left semiconcept, while its upper approximation 
		%$v(\lozenge\alpha)$ of the left semiconcept $v(\alpha)$ 
		lies ``above" it.
		\item On the other hand, if $v(\alpha)$ is a right semiconcept of $\mathbb{K}$ then $v(\blacklozenge\alpha)$ is its lower approximation and $v(\blacksquare\alpha)$ is its upper approximation. The valid sequents $\blacklozenge\alpha\vdash\alpha\sqcup\alpha$ and $\alpha\vdash\blacksquare\alpha$  translate into properties of approximations of right semiconcepts: the lower approximation  
		%$v(\square\alpha)$ 
		of a right semiconcept 
		%$v(\alpha)$ 
		lies below it, while its upper approximation 
		%$v(\lozenge\alpha)$ of the left semiconcept $v(\alpha)$ 
		lies above it.		
%		\item For a formula $\blacklozenge\alpha$, if $v(\alpha)$ is a right semiconcept of a context $\mathbb{K}$ then $v(\blacklozenge\alpha)$ is lower approximation of $v(\alpha)$  and the valid sequent  $\blacklozenge\alpha\vdash\alpha\sqcup\alpha$ represent the fact that lower approximation $v(\blacklozenge\alpha)$ of the right semiconcept $v(\alpha)$ lies below the right semiconcept.
%		
%		\item For a formula $\blacksquare\alpha$, if $v(\alpha)$ is a right semiconcept of a context $\mathbb{K}$ then $v(\blacksquare\alpha)$ is upper approximation of $v(\alpha)$. Moreover, the valid sequent $\alpha\vdash\blacksquare\alpha$ represent the fact that upper approximation  $v(\blacksquare\alpha)$ of the right semiconcept $v(\alpha)$ lies above the right semiconcept.
%		
%		\item For a formula $\lozenge\alpha$, if $v(\alpha)$ is a left semiconcept of a context $\mathbb{K}$ then $v(\lozenge\alpha)$ is upper approximation of $v(\alpha)$  and the valid sequent  $\alpha\sqcap\alpha\vdash\lozenge\alpha$ represent the fact that upper approximation $v(\lozenge\alpha)$ of the left semiconcept $v(\alpha)$ lies above the left semiconcept.
		\end{enumerate}	}										
\end{observation}

\noindent Let $\underline{(A, B)}=(X, Y)$. By Proposition \ref{defsemi}, it follows that the extent $X$ and intent $Y$ are categories in $(G, E_{1})$ and $(M, E_{2})$, respectively. By Proposition \ref{property of approx}, we have $X\subseteq A$ and $B\subseteq Y$, which implies that $X\subseteq \underline{A}\subseteq A$ and $\overline{B}\subseteq Y$. From  $X\subseteq \underline{A}\subseteq A$, we  can say that the objects in the extent $X$ are certainly classified as objects in the extent $A$ of the concept $(A, B)$. Now $\overline{B}\subseteq Y$ implies that $Y^{c}\subseteq \underline{(B^{c})}$. So we can say that  properties in  $Y^{c}$  are certainly classified as properties in the intent $B^{C}$ of $\lrcorner (A, B)$.
Similarly, we have  $\underline{A}\subseteq ext(\overline{(A, B)})$ and $int(\overline{(A, B)})\subseteq \overline{B}$. So the properties in $int(\overline{(A, B)})$  are the  properties that  possibly  belong to  the intent $B$ of the concept $(A, B)$. Now  $\underline{A}\subseteq ext(\overline{(A, B)})$ implies that $ ext(\overline{(A, B)})^{c}\subseteq \overline{(A^{c})}$. So the objects  in the set  $ ext(\overline{(A, B)})^{c}$ are the objects that possibly belong to the extent $A^{C}$ of $\neg(A, B)$.

We end the section by making a comparison with the work in \cite{saquer2001concept}. Consider a semiconcept $(A, A^{\prime})$ that is not a concept. According to \cite{saquer2001concept}, this is a non-definable concept such that $A^{\prime}$ is feasible and A is not. As mentioned in Section \ref{Appropefca}, the lower and upper approximations of $(A, A^{\prime})$ as given in \cite{saquer2001concept} are \\$\underline{(A, A^{\prime})}:=((\underline{A}_{E_{1}})^{\prime\prime}\cap A^{\prime\prime}, ((\underline{A}_{E_{1}})^{\prime\prime}\cap A^{\prime\prime})^{\prime} )=((\underline{A}_{E_{1}})^{\prime\prime}, (\underline{A}_{E_{1}})^{\prime})$ and $\overline{(A, A^{\prime})}=((\overline{A}^{E_{1}})^{\prime\prime}, (\overline{A}^{E_{1}})^{\prime} )$. \\Now observe that $\underline{(A, A^{\prime})}$ and $\overline{(A, A^{\prime})}$ are both concepts. In our case, the lower and upper approximations of a semiconcept that is not a concept, may not be a concept: 
recall  Example \ref{exampleapprox}. The lower approximation  $(\{ Leech, Bream\}, \{a,b,g\})$ of the semiconcept  $(\{Leech,Bream , Dog\}, \{a, g\})$  is not a concept.
Furthermore,  in our case (as pointed out in Observation \ref{upperlowrapprox}),
%it follows from Proposition \ref{property of approx} that 
the lower approximation of a semiconcept lies below the semiconcept, while  its upper approximation lies above  it -- thereby justifying  that these are {\it lower} and {\it upper} approximations of the semiconcept. In general, these properties do not hold  for the lower and upper approximations of concepts defined in \cite{saquer2001concept}.

\section{Conclusions}
\label{conclusion}
%The logic \textbf{MPDBL4} for topological pdBas is obtained as a special case of $\textbf{MPDBL}\Sigma$, where $\Sigma$ is any set of sequents in $\textbf{MPDBL}$. This gives a scheme of  obtaining  several other logics that may express   properties of  dBaos and corresponding classes of Kripke contexts besides the ones considered here. For topological pdBas and correspondingly, reflexive and transitive Kripke contexts,   \textbf{MPDBL4} with $\Sigma$ containing the modal axioms for  reflexivity and transitivity, serves the purpose. One may well include other axioms (such as  symmetry) in $\Sigma$, and investigate the resulting modal systems. 

%**to be checked**\\
 %$\textbf{MPDBL}$  is based on the logic $\textbf{PDBL}$ for pdBas.
%The logic $\textbf{PDBL}$ is shown to be a logic for conceptual knowledge.
The hyper-sequent calculus \textbf{PDBL} is defined  for the class of pdBas and extended to \textbf{MPDBL} for the class of pdBaos.  For any set $\Sigma$ of sequents   in $\textbf{MPDBL}$, $\textbf{MPDBL}\Sigma$ is defined.
As  particular cases of $\textbf{MPDBL}\Sigma$, the logics \textbf{MPDBL4} (for tpdBas) and $\textbf{MPDBL5}$ are obtained.  This presents a technique for constructing various logics in general, that may represent attributes of pdBaos and associated classes of Kripke contexts.
%$\textbf{MPDBL4}$ with $\Sigma$ holding the modal axioms for reflexivity and transitivity satisfies the purpose for tpdBas and, as a result, reflexive and transitive Kripke contexts. For reflexive, symmetric  and transitive Kripke contexts, $\textbf{MPDBL5}$ serves the purpose.

  For conceptual knowledge, it is established that using \textbf{PDBL} and its models, one can express the basic notion of concept, and the relations ``object belongs to a concept'', ``property abstracts from a concept'' and ``a concept is a subconcept of another concept''.
Further, the basic notions of objects and attributes and the relation ``an object has an attribute" are expressible by using the named $\textbf{PDBL}$ models.
%The relations of satisfaction and co-satisfaction, respectively, describe the conceptual knowledge relations ``object belongs to the concept" and ``property  belongs to the concept."
%Taking a cue from  hybrid modal logic, a named model is defined for $\textbf{PDBL}$. 
%When $\textbf{PDBL}$ is interpreted in the named model, it is shown that object variables and property variables reflect objects and properties of conceptual knowledge, respectively.
%A named model is deemed to be an instance of the \textbf{PDBL} model.
%As a result, when 
When interpretations of $\textbf{PDBL}$ are restricted to the collection of named models, the logic is sound
%is interpreted in the named model, it is also 
with respect to the class of of all contexts. An open question  then is, whether a logic may possibly be derived from $\textbf{PDBL}$ that would be complete with respect to this restricted collection of models.
%When  $\textbf{PDBL}$ is interpreted in the named model, the completeness of $\textbf{PDBL}$ with respect to the class of all contexts is still an open problem. 

This work attempts to highlight the significance of  semiconcepts of a context from different points of view. In particular, the observations of Section \ref{sematics to meaning1} indicate that a mathematical model  for knowledge related to semiconcepts may well be defined, akin to the conceptual knowledge system defined in \cite{lpwille,WILLE1992493}. The significance of such a model may be worth exploring.

\newpage
\begin{appendices}
{\rm \section*{Proofs}\label{secA1}	

\noindent 	{\it Proof of Theorem \ref{thempdbl}:}
	The proofs are straightforward and one makes use of axioms 2a, 3a, 4a, Proposition \ref{drive rule1}  and the rule $(R4)$ in most cases. 
	$1a$: 
	%-- it uses   axioms 2a, 3a, 4a, Proposition \ref{drive rule1}  and the rule $(R4)$. 	
	\begin{center}
		$\infer{(\alpha\sqcap\beta)\vdash(\beta\sqcap\alpha)~(R4)}{4a~(\alpha\sqcap\beta)\vdash(\alpha\sqcap\beta)\sqcap(\alpha\sqcap\beta) &\infer{(\alpha\sqcap\beta)\sqcap(\alpha\sqcap\beta)\vdash\beta\sqcap\alpha~(R6)}{3a~\alpha\sqcap\beta\vdash\beta & \alpha\sqcap\beta\vdash\alpha~2a}}$
	\end{center}
	
\noindent	Interchanging $\alpha$ and $\beta$ in the above, we get  $(\beta\sqcap\alpha)\vdash(\alpha\sqcap\beta)$. \\
	
	%Therefore $(\alpha\sqcap \beta)\dashv\vdash (\beta\sqcap \alpha)$.\\
	  
	%Proposition \ref{drive rule1} is also used in some of the proofs. 
	%Let $\alpha,\beta,\gamma\in\mathfrak{F}$.\\

	%The proofs are straightforward and one makes use of axioms 2a, 3a, 4a, Proposition \ref{drive rule1}  and the rule $(R4)$ in most cases. 
	%**Give the axiom/rule used**	

		%\noindent	$1a$: 	\begin{center}
		%		$\infer{(\alpha\sqcap\beta)\vdash(\beta\sqcap\alpha)}{(\alpha\sqcap\beta)\vdash(\alpha\sqcap\beta)\sqcap(\alpha\sqcap\beta) &\infer{(\alpha\sqcap\beta)\sqcap(\alpha\sqcap\beta)\vdash\beta\sqcap\alpha}{\alpha\sqcap\beta\vdash\beta & \alpha\sqcap\beta\vdash\alpha}}$
		%	\end{center}
		%	Interchanging $\alpha$ and $\beta$ in the above, we get  $(\beta\sqcap\alpha)\vdash(\alpha\sqcap\beta)$. 
		
	{\small		\noindent $2a.$
		
		$\infer{(\alpha\sqcap\beta)\sqcap\gamma\vdash\beta\sqcap\gamma~(R4)~\mbox{--~(I)}}{4a~(\alpha\sqcap\beta)\sqcap\gamma\vdash((\alpha\sqcap\beta)\sqcap\gamma)\sqcap((\alpha\sqcap\beta)\sqcap\gamma) & \infer{((\alpha\sqcap\beta)\sqcap\gamma)\sqcap((\alpha\sqcap\beta)\sqcap\gamma)\vdash\beta\sqcap\gamma~(R6)}{ \infer{(R4)~(\alpha\sqcap\beta)\sqcap\gamma\vdash\beta}{2a~(\alpha\sqcap\beta)\sqcap\gamma\vdash(\alpha\sqcap\beta) & \alpha\sqcap\beta\vdash\beta~3a}& (\alpha\sqcap\beta)\sqcap\gamma\vdash\gamma~3a}}$
		
	Now,\\
		$\infer{(\alpha\sqcap\beta)\sqcap\gamma\vdash\alpha\sqcap(\beta\sqcap\gamma)~(R4)}{4a~(\alpha\sqcap\beta)\sqcap\gamma\vdash((\alpha\sqcap\beta)\sqcap\gamma)\sqcap((\alpha\sqcap\beta)\sqcap\gamma)&\infer{((\alpha\sqcap\beta)\sqcap\gamma)\sqcap((\alpha\sqcap\beta)\sqcap\gamma)\vdash\alpha\sqcap(\beta\sqcap\gamma)~(R6)}{ \infer{(\alpha\sqcap\beta)\sqcap\gamma\vdash\alpha~(R4)}{2a~(\alpha\sqcap\beta)\sqcap\gamma\vdash \alpha\sqcap\beta & \alpha\sqcap\beta\vdash\alpha~2a}& (\alpha\sqcap\beta)\sqcap\gamma\vdash\beta\sqcap\gamma~\mbox{(from (I) above)} }}$
		
		Similarly we can show that $\alpha\sqcap(\beta\sqcap\gamma)\vdash(\alpha\sqcap\beta)\sqcap\gamma$. \\		%Therefore $\alpha\sqcap(\beta\sqcap\gamma)\dashv\vdash(\alpha\sqcap\beta)\sqcap\gamma$.\\
		
		\noindent $3a.$\\
		$\infer{(\alpha\sqcap\alpha)\sqcap\beta\vdash\alpha\sqcap\beta~(R4)}{4a~(\alpha\sqcap\alpha)\sqcap\beta\vdash((\alpha\sqcap\alpha)\sqcap\beta)\sqcap((\alpha\sqcap\alpha)\sqcap\beta)&\infer{((\alpha\sqcap\alpha)\sqcap\beta)\sqcap((\alpha\sqcap\alpha)\sqcap\beta)\vdash\alpha\sqcap\beta~(R6)}{\infer{(R4)~(\alpha\sqcap\alpha)\sqcap\beta\vdash\alpha}{2a~(\alpha\sqcap\alpha)\sqcap\beta\vdash\alpha\sqcap\alpha & \alpha\sqcap\alpha\vdash\alpha~2a}&(\alpha\sqcap\alpha)\sqcap\beta\vdash\beta~3a}}$\\
		
		$\infer{\alpha\sqcap\beta\vdash(\alpha\sqcap\alpha)\sqcap\beta~(R4)}{4a~\alpha\sqcap\beta\vdash(\alpha\sqcap\beta)\sqcap(\alpha\sqcap\beta)&\infer{(\alpha\sqcap\beta)\sqcap(\alpha\sqcap\beta)\vdash(\alpha\sqcap\alpha)\sqcap\beta~(R6)}{\infer{(R4)~\alpha\sqcap\beta\vdash\alpha\sqcap\alpha}{4a~\alpha\sqcap\beta\vdash(\alpha\sqcap\beta)\sqcap(\alpha\sqcap\beta)&\infer{(\alpha\sqcap\beta)\sqcap(\alpha\sqcap\beta)\vdash\alpha\sqcap\alpha~(R6)}{2a~(\alpha\sqcap\beta)\vdash\alpha & (\alpha\sqcap\beta)\vdash\alpha~2a}}& \alpha\sqcap\beta\vdash\beta~3a}}$\\
		%So  $(\alpha\sqcap \alpha)\sqcap \beta\dashv\vdash (\alpha\sqcap \beta)$.\\
		%$4a.$ Proof follows from axiom $(2a)$ and $(R3)$.\\
		\vskip 15pt
		$(4a)$ follows from axiom 2a and $(R3)$.
		\vskip 15pt
		$5a.$ \\
		$\infer{\alpha\sqcap(\alpha\sqcup\beta)\vdash\alpha\sqcap\alpha~(R4)}{4a~\alpha\sqcap(\alpha\sqcup\beta)\vdash(\alpha\sqcap(\alpha\sqcup\beta))\sqcap(\alpha\sqcap(\alpha\sqcup\beta))&\infer{(\alpha\sqcap(\alpha\sqcup\beta))\sqcap(\alpha\sqcap(\alpha\sqcup\beta))\vdash\alpha\sqcap\alpha~(R6)}{2a~\alpha\sqcap(\alpha\sqcup\beta)\vdash\alpha & \alpha\sqcap(\alpha\sqcup\beta)\vdash\alpha~2a}}$\\
		$6a.$ Proof is identical to that of $5a.$
		\vskip 15pt
		\noindent $(7a), (8a)$ follow from axiom 11a. 
		\vskip 15pt
	\noindent	Note that
	%and **$7b,8b$ follow from axiom (1b). 
	the proofs of $(ib), i=1,2,3,4,5,6,7,8,$ are obtained using the axioms and rules dual to those used to derive $(ia)$.}
%		$\infer{\alpha\sqcap(\alpha\vee\beta)\vdash\alpha\sqcap\alpha~(R4)}{4a~\alpha\sqcap(\alpha\vee\beta)\vdash(\alpha\sqcap(\alpha\vee\beta))\sqcap(\alpha\sqcap(\alpha\vee\beta))&\infer{(\alpha\sqcap(\alpha\vee\beta))\sqcap(\alpha\sqcap(\alpha\vee\beta))\vdash\alpha\sqcap\alpha~(R7)}{2a~\alpha\sqcap(\alpha\vee\beta)\vdash\alpha & \alpha\sqcap(\alpha\vee\beta)\vdash\alpha~2a}}$
	
\vskip 15pt
	
\noindent {\it Proof of Proposition \ref{ldbm-dba}}		For $1\implies2$, we make use of  $(R1)^{\prime}$, $(R4)$, axiom 2a and Theorem \ref{thempdbl}(2a, 3a). 
	% For $2\implies 1$,   (R5) is used.  
		\begin{prooftree}
			\AxiomC{$\alpha\vdash\beta$}
			\UnaryInfC{$\alpha\sqcap\alpha\vdash\alpha\sqcap\beta$}
			\UnaryInfC{$\alpha\sqcap\beta\vdash\alpha$}
			\UnaryInfC{$\alpha\sqcap(\alpha\sqcap\beta)\vdash\alpha\sqcap\alpha~~~\alpha\sqcap\beta\vdash\alpha\sqcap(\alpha\sqcap\beta)$}
			\UnaryInfC{$\alpha\sqcap\beta\vdash\alpha\sqcap\alpha$}
		\end{prooftree}
		So $\alpha\sqcap\alpha\dashv\vdash\alpha\sqcap\beta$, which implies that $[\alpha]\sqcap [\alpha]=[\alpha\sqcap\alpha]=[\alpha\sqcap\beta]=[\alpha]\sqcap [\beta]$. Dually we can show that $[\alpha]\sqcup [\beta]=[\beta]\sqcup [\beta]$. Therefore $[\alpha]\sqsubseteq [\beta]$.\\
		For $2\implies 1$, suppose $[\alpha]\sqsubseteq [\beta]$. Then $[\alpha]\sqcap [\beta]=[\alpha]\sqcap [\alpha]$. So $[\alpha\sqcap\beta]=[\alpha\sqcap\alpha]$. Similarly we can show that $[\alpha\sqcup\beta]=[\beta\sqcup\beta]$. Therefore $\alpha\sqcap\beta\dashv\vdash\alpha\sqcap\alpha$ and $\alpha\sqcup\beta\dashv\vdash\beta\sqcup\beta$. Now using (R5), $\alpha\vdash\beta$.}
\end{appendices}

	
	
	%\end{proof}

\end{document}